\documentclass[11pt,a4paper,leqno]{article}

\usepackage{theorem} 
\usepackage[latin1]{inputenc}
\usepackage{amsmath}

\usepackage{amsfonts}
\usepackage{amssymb}
\usepackage{amscd}
\usepackage{makeidx}
\usepackage{graphicx}
\usepackage{epstopdf}
\usepackage{epsfig}
\usepackage{pstricks}

\newcommand{\ds}{\displaystyle}

\newtheorem{thm}{Theorem}[section] 
\newtheorem{cor}[thm]{Corollary} 
\newtheorem{lem}[thm]{Lemma} 
\newtheorem{prop}[thm]{Proposition} 
\newtheorem{prob}[thm]{Problem} 
\newtheorem{remark}[thm]{Remark} 

\newenvironment{proof}[1][Proof]{\textbf{#1:}\ }  {\hfill\rule{1ex}{1ex}}

\theorembodyfont{\normalfont\upshape} 
\newtheorem{defn}[thm]{Definition} 

\makeindex

\def\cb{{\vskip-4mm{\hfill\hbox{$\vrule\vcenter to 2mm{\hrule\vfil\hbox{\kern
2mm}\vfil\hrule}\vrule$}\quad}} \medskip}

\def\hfl#1#2{\smash{\mathop{\hbox to 12mm{\rightarrowfill}}
\limits^{\scriptstyle#1}_{\scriptstyle#2}}}

\begin{document}

\title{On the geometry of horseshoes in higher dimensions}
\author{Carlos Gustavo Tamm de Araujo Moreira \thanks{IMPA, \ gugu@impa.br} \\ Waliston Luiz Lopes Rodrigues Silva \thanks{UFSJ, \ waliston@yahoo.com}}
\date{...}

\maketitle

\begin{abstract}

	The criterion of the recurrent compact set was introduced by Moreira and Yoccoz to prove that stable intersections of regular Cantor sets on the real line are dense in the region where the sum of their Hausdorff dimensions is bigger than 1. We adapt this concept to the context of horseshoes in ambient dimension higher than 2 and prove that horseshoes with upper stable dimension bigger than 1 satisfy, typically and persistently, the adapted criterion of the recurrent compact set. As consequences we show some persistent geometric properties of these horseshoes. In particular, typically and persistently, horseshoes with upper stable dimension bigger than 1 present blenders.
	
\end{abstract}

\def\F{\tilde{F}}
\def\omegab{\underline{\omega}}
\def\tb{\underline{t}}
\def\real{\mathbf{R}}
\def\dint{\displaystyle \int}
\

Partially supported by the Balzan Research Project of J.Palis

\newpage

\tableofcontents

\newpage

\section{Introduction}

Fractal dimensions, mainly the Hausdorff dimension, had frequently played a central rôle in the field of the Dynamical Systems in the last decades. Moreira, Palis, Takens and Yoccoz (\cite{MY},\cite{MY1},\cite{PT}, \cite{PTan} and \cite{PY}), proved that, in dimension 2, homoclinic bifurcations associated to first homoclinic tangencies of a horseshoe, $\Lambda,$ hyperbolicity prevails if and only if the Hausdorff dimension of $\Lambda$ is smaller than 1. Also, if the Hausdorff dimension of the horseshoe is bigger than 1, then, tipically, there is persistently positive density of persistent tangencies at the first parameter of bifurcation and the union of hyperbolicity and persistent tangencies has full Lebesgue density at the first parameter of bifurcation. Moreira, Palis and Viana generalize this panorama for horseshoes in higher ambient dimensions \cite{MPV}.

The understanding of the geometry of horseshoes and their intersections with their stable and unstable manifolds is crucial in all works cited in the last paragraph. In particular, the differentiability of the stable and unstable holonomies, in dimension 2, is used in an essential way to obtain these results. This is not true, in general, for foliations in ambient dimension higher than 2 - in general these foliations are not more than H\"older-continuous.

In this work, we use the concept of \textbf{upper stable dimension}, introduced in \cite{MPV}. Its manipulation is simpler than that of the Hausdorff dimension and it is an upper bound for the Hausdorff dimension and the limit capacities of the \textbf{stable Cantor sets} - given by the intersection of the horseshoe and a local stable manifold of some point in the horseshoe. We prove that, tipically, horseshoes in dimension higher than 2 with upper stable dimension bigger than 1 satisfy the following: the image of any of its stable Cantor sets by generic real functions of class $C^1$ persistently contains intervals. In order to obtain such a result, we develop a criterion inspired in the \textbf{criterion of the recurrent compact set}, introduced in \cite{MY} to prove that stable intersections of regular Cantor sets in the real line are typical when the sum of their Hausdorff dimensions is bigger than 1, and then we prove that this new criterion is tipically satisfied when the upper stable dimension is bigger than 1. To perform this task we suppose that the horseshoes's tangent bundle admits a sharp splitting - meaning that $T_{\Lambda}M = E^{ss} \oplus E^{ws} \oplus E^u$ with $\mbox{dim}(E^{ws}) = 1$. We emphasize that this hypothesis is robust and that we can make use of a technique described in \cite{MPV} to prove that horseshoes in dimension higher than 2 having upper stable dimension bigger than 1, typically, contains subhorseshoes admitting sharp splitting for the tangent bundle and still having upper stable dimension higher than 1 - this allows us to extend our results to the general case, where the tangent splitting is not necessarily sharp.

Among the relevant geometric properties of horseshoes possessing recurrent compact sets, we highlight the existence ofblenders. This concept was introduced by Bonatti and D\'iaz, in \cite{BD}, in order to present a new class of examples of non-hyperbolic $C^1-$robustly transitive diffeomorphisms. Blenders are useful to connect two sadles with different indexes in the same transitive set and they have been constructed to obtain topological and ergodic properties of some Dynamical Systems. The study of its applications is pursued in works such as \cite{DNP} in which is obtained $C^1-$local coexistence of infinte sinks or sources and \cite{RHRHTU} in which the authors give a positive answer to a longstanding conjecure by Pugh and Shub on ergodic stability of partially hyperbolic systems in the $C^1$ topology admitting central direction with dimension 2. Until now - as far as we know - blenders were constructed taking as a departure point some specific horseshoe and its existence is due to the presence of some heterodimensional cycle near it. In this work - as a consequence of the criterion of the recurrent compact set - we stablish a typical criterion for the existence of blenders: typical horseshoes in ambient dimension higher than 2 with upper stable dimension bigger than 1 carry blenders. A good reference discussing, among other themes `beyond hyperbolicity', the notion of blender can be found in \cite{BDVLivro}. We thank professors Ali Tahzibi, Christian Bonatti and Lorenzo D\'iaz for useful conversations on this subject.

To accomplish our main objective - to prove that the criterion of the recurrent compact set in our context is typically satisfied - we adapt two techniques found in the literature: the \textbf{probabilistic argument} and a \textbf{Marstrand-like argument}. The first technique was introduced by Paul Erd\"os and was employed originally in graph theory, but at a later time it has become a valuable instrument in diverse fields of mathematics. A good reference to illustrate the probabilistic argument working in combinatorics can be found in \cite{GErdos}. This technique was employed, also, in \cite{MY} to prove typical existence of recurrent compact sets for pairs of Cantor sets.

Marstrand proved that, tipically, projections along straight lines forming a fixed angle with the x-axis of a compact set in the plane with Hausdorff dimension bigger than 1 have positive Lebesgue measure, \cite{Mar}. A new proof of this fact can be found in \cite{K} and yet another proof, of a combinatorial flavor which can be useful to give us further insights on the geometry of horseshoes can be found in \cite{LM2} (generalizing the proof given in \cite{LM1} for the case of products of regular cantor sets). In order to prove our main results, we need to adapt a Marstrand-like argument - found in \cite{SSU} - which stablish that perturbation families of iterated function systems (IFS) having a certain fractal dimension (similar to the upper stable dimension) higher than one exhibit an invariant set with positive Lebesgue measure almost surely. We thank professor K\'aroly Simon for helpful discussions about his result.

There are questions related to the fractal geometry of horseshoes in dimension higher than 2 about which we believe that our methods can be useful. We remember that it is more difficult to estimate the Hausdorff dimension of a stable Cantor set - intersection of the horseshoe with a local stable manifold - than estimating its upper stable dimension. We can mention some interesting problems in this direction: we don't know whether the Hausdorff dimensions of stable Cantor sets remain constant as we vary the stable manifold in which they live; it would be interesting to know whether, typically, the Hausdorff dimension of the stable Cantor sets varies continuously with the horseshoe (this is false if we omit the word ``typically" - there is an example of a horseshoe in \cite{BDV} not satisfying the continuity of the Hausdorff dimensions as the diffeomorphism varies). These problems for horseshoes in dimension 2 are already positively solved: in \cite{McCM} and \cite{PV} it is proved that the Hausdorff dimension of $C^1-$horseshoes in dimension 2 varies continuously in the $C^1-$topology.

We start our work in the next section in which we stablish the notations and the context of our work. Then, we introduce the concept of upper stable dimension in section \ref{Dimensoes_superior_estavel_e_de_Hausdorff} and state our main result and some of its corollaries in section \ref{sec_top}. The remaining part of this work will be dedicated to prove the main theorem - horseshoes in ambient dimension higher than 2 with upper stable dimension bigger than 1 and admitting sharp splitting of its tangent bundle satisfy the criterion of the recurrent compact set typically and robustly.

\section{Context and notations}
\label{secao_contexto_notacoes}

Let $M$ be a $n-$dimensional manifold with $n \geq 3,$ let $f:M \rightarrow M$ be a diffeomorphism of class $C^k$ $(k \in \mathbb{N}^* \cup \{ \infty \})$ and $\Lambda \subset M$ a \textbf{horseshoe} - a hyperbolic, locally isolated and topologically transitive set.

\begin{remark}
Actually we need suppose $f^n$ topologically mixing. We will perform some perturbations on $f^n$ and we observe that this perturbation can be, in fact, performed as a perturbation on $f$ since $\Lambda$ is, in fact, a basic piece.
\end{remark}

\begin{remark}
When $\Lambda$ is topologically \textbf{transitive } there is a Markov partition, ${\cal P} = \{P_1,...,P_N\},$ such that for every $P$ and $Q$ in ${\cal P},$ $f^n(P) \cap Q \ne \emptyset$ for some integer $n.$ When $\Lambda$ is topologically \textbf{mixing} there is a Markov partition, ${\cal P} = \{P_1,...,P_N\},$ and $n>0$ integer such that $f^n(P) \cap Q \ne \emptyset$ for every $P$ and $Q$ in ${\cal P}.$
\end{remark}

\begin{remark}
We say the tangent bundle, $T_\Lambda M,$ has \textbf{sharp splitting} if it can be decomposed as a direct sum of three subbundles, $T_\Lambda M = E^{ss} \oplus E^{ws} \oplus E^u,$ in such a way that $\mbox{dim}(E^{ws}) = 1$ and

\begin{enumerate}
	\item $| df |_{E^{ss}}(x)v | \leq \lambda^{ss}(x) |v|,$ for $v \in E^{ss}(x),$

	\item $| df |_{E^{ws}}(x)v | = \lambda^{ws}(x) |v|,$ for $v \in E^{ws}(x),$
	
	\item $| df |_{E^u}(x)v | \geq \lambda^u(x) |v|,$ for $v \in E^{u}(x),$
\end{enumerate}

where $0 < \ | \lambda^{ss}(x) \ | < \ | \lambda^{ws}(x) \ | < 1 < \ | \lambda^u(x) \ |$ for every $x \in \Lambda.$
\end{remark}

(As in the classical definition of hyperbolic set, we adopt an adapted metric).

We observe that if $f \in C^{\infty},$ then by the $C^r-$section theorem (see \cite{HPS}, \cite{S}) there is a strong stable foliation in $W_{loc}^s(\Lambda)$, ${\cal F}^{ss},$ of class $C^{1+ \varepsilon}$ and tangent to $E^{ss}$ and also a stable foliation in $M$, ${\cal F}^s,$ of class $C^{1 + \varepsilon}$ and tangent to $E^s.$ Besides, the sharp splitting is a $C^1-$robust property by the cone field argument.

Let $\sigma: \Sigma \rightarrow \Sigma$ be the mixing subshift of finite type associated to the Markov partition, ${\cal P} = \{P_1,...,P_N\},$ for the horseshoe $(f^{-1},\Lambda),$ i.e., conjugated to $f^{-1}.$ By letting a subshift being mixing we mean that there is a $N \times N$ matrix, $A \in \{ 0,1 \}^{N \times N},$ such that $\theta \in \Sigma$ if and only if $A_{\theta_i,\theta_{i+1}} = 1$ for every $i \in \mathbb{Z}$ and, besides, there is a natural $n$ such that $A_{i,j}^n > 0$ for every $i,j$ beteween $1$ and $N.$

Now we fix some notations. In the following definitions we make a slight abuse of notation: when we write a word with an index in its letters, we are fixing the position of the word through those indexes, i.e., the notation $(\theta_m, \theta_{m+1},...,\theta_n)$ represents, actually, the function $\theta: \{ m, m+1, ..., n \} \rightarrow \Sigma,$ with $\theta(j) = \theta_j$ for $m \leq j \leq n,$ and not only merely the vector $(\theta_m, \theta_{m+1},...,\theta_n).$ We observe, also, that we consider $0 \in \mathbb{N}$.

\begin{itemize}

\item $\Sigma^- := \Bigl\{ \theta^-:= (..., \theta_{-n}, ..., \theta_0) ; A_{\theta_{-i},\theta_{-i+1}} = 1 \mbox{ for every } i \in \mathbb{N}^* \Bigr\}$ is the set of the backward infinite words.

\item $\Sigma^+ := \Bigl\{ \theta:= (\theta_1, ..., \theta_m, ...) ; A_{\theta_i,\theta_{i+1}} = 1 \mbox{ for every } i \in \mathbb{N}^* \Bigr\}$ is the set of the forward infinite words

\item $\Sigma^{+*} := \Bigl\{ \underline{\theta}:= ({\theta}_1, ..., {\theta}_m) ; m \in \mathbb{N}, A_{{{\theta}}_i,{\theta}_{i+1}} = 1 \mbox{ for every } i \in \{ 1, ... , m-1  \} \Bigr\}$ is the set of the forward finite words.

\item $\Sigma^{m} := \Bigl\{ \underline{\theta}:= ({\theta}_1, ..., {\theta}_m) ; A_{{{\theta}}_i,{\theta}_{i+1}} = 1 \mbox{ for every } i \in \{ 1, ... , m-1  \} \Bigr\}$ is the set of the forward finite words with size $m.$

\item $\Sigma^{-*} := \Bigl\{ \underline{\theta}^-:= ({\theta}_{-m}, ..., {\theta}_0) ; m \in \mathbb{N}, A_{{{\theta}}_{-i},{\theta}_{-i+1}} = 1 \mbox{ for every } i \in \{ 1, ... , m  \} \Bigr\}$ is the set of the backward finite words.

\item $\Sigma^{*} := \Bigl\{ \underline{\theta}:= ({\theta}_{m}, ..., {\theta}_n) ; m \in \mathbb{Z}, n \in \mathbb{Z}, A_{{{\theta}}_{i},{\theta}_{i+1}} = 1 \mbox{ for every } i \in \{ m, ... , n-1  \} \Bigr\}$ is the set of the finite words.

\item $\Sigma_{k} := \Bigl\{ {\theta}:= ({\theta}_1, ..., {\theta}_m, ...) ; A_{{k,\theta_1}} =1 \mbox{ and } A_{{{\theta}}_i,{\theta}_{i+1}} = 1 \mbox{ for every } i \in \mathbb{N}^* \Bigr\}$ is the set of the forward infinite words which follow the letter $k$.

\item $\Sigma^*_{k} := \Bigl\{ \underline{\theta}:= ({\theta}_1, ..., {\theta}_m ) ; m \in \mathbb{N}, A_{{k,\theta_1}} =1 \mbox{ and } A_{{{\theta}}_i,{\theta}_{i+1}} = 1 \mbox{ for every } i \in \{ 1, ... , m-1  \} \Bigr\}$ is the set of the forward finite words which follow $k.$

\end{itemize}

We note that when a symbol alluding to a word is underlined, as in $\underline{\theta},$ we want to refer to a finite word, otherwise we mean a infinite word. But when we refer to a letter of a finite word we omit the underline.

For each $g$ in some $C^1-$neighbourhood of $f,$ denote by $E^{g,ss},$ $p^g,$ $W^{g,s}$ and $\Lambda^g$ the hyperbolic continuations of $E^{ss},$ $p,$ $W^{s}$ and $\Lambda.$

Along this work, we create perturbation families in many parameters for certain diffeomorphisms. The objective will be clear in the sequel. We say that $\{ \phi^{\underline{\gamma}}: M \rightarrow M \}_{{\underline{\gamma}} \in \Gamma}$ is a $C^k-$continuous family of diffeomorphisms if $\phi^{\underline{\gamma}}$ is $C^k$ and $\phi^{\underline{\gamma}}$ varies $C^k-$continuously with respecto to ${\underline{\gamma}}.$

As indicated in the above definition, the parameter of perturbation will be indicated in supscript. Along this work, we perform two perturbations - the first moulding a Marstrand-like result, based on the work by Simon, Solomyak and Urba\'nski,\cite{SSU}, while the second perturbation will be performed in order to find a recurrent compact set (this perturbation will be based in the probabilistic argument adapted from \cite{MY}). In the first perturbation we use the symbol $\underline{t}$ as parameter, as for the second one we use as parameter the symbol $\omegab,$ and as space of parameters the symbol $\Omega.$ The symbols $\underline{\gamma}$ and $\Gamma$ will be used generically as parameters and space of parameters respectively. These notations will be introduced in the deserved time. We think the exposition will be plainer this way.

Given a continuous family of perturbations, $\{ f^{\underline{\gamma}} \}_{{\underline{\gamma}} \in \Gamma},$ inside a $C^1-$neighbourhood of $f$ sufficiently small, we denote by $E^{{\underline{\gamma}},ss},$ $p^{\underline{\gamma}},$ $W^{{\underline{\gamma}},s}$ and $\Lambda^{\underline{\gamma}}$ the hyperbolic continuations of $E^{ss},$ $p,$ $W^s$ and $\Lambda.$ We observe that if ${\cal P} = \{P_1,...,P_N\}$ is a Markov partition for $f,$ then, without loss of generality, it is also a Markov partition for any diffeomorphism sufficiently close to $f.$

We denote by $W^{s}_{loc}(p)$ the connected component of $W_{loc}^{s}(p) \cap P$ to which $p$ belongs, where $P \in {\cal P}.$

For every $g$ sufficiently $C^1-$close to $f$ there is a homeomorphism $h^{g}:\Sigma \rightarrow \Lambda^{g}$ such that each infinite word $\theta$ in $\Sigma,$ associates $$\displaystyle{h^{g}(\theta):=\bigcap_{j \geq 0} g^{-j}(P_{\theta_{-j}}) \cap \bigcap_{j \geq 1} g^j( P_{\theta_{j}} )}.$$

We observe that $\displaystyle{h^{g}(\underline{\theta}):=\bigcap_{j = k_1 }^{k_2} g^{j}(P_{\underline{\theta}_{j}})}$ for any finite word $\underline{\theta} := ( {\theta}_{k_1}, ... , {\theta}_{k_2} ) \in \Sigma^*.$

Beyond that, $g^{-1} \circ h^{g}(\theta) = h^{g} \circ \sigma( \theta ),$ where $\sigma$ is the subshift, i.e., $\sigma(\theta)_{i} = \theta_{i+1}.$

\begin{figure}[h]
    \centering
		\includegraphics[width=3.5cm]{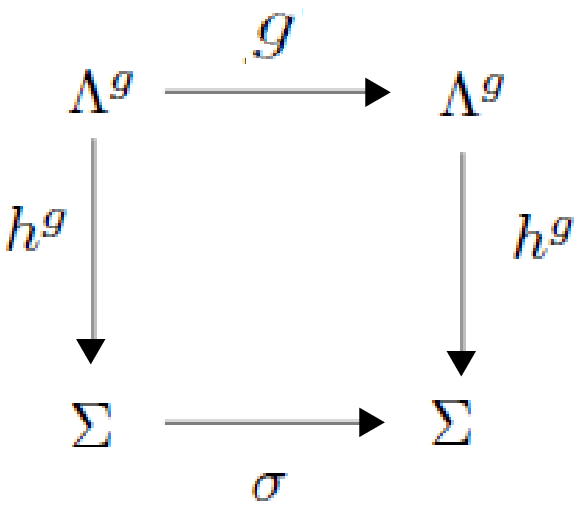}
\end{figure}

Fixed $\theta^- \in \Sigma^-,$ it is worth to observe that $h^{g}(\theta^-) = W^{g,s}_{loc}(p^{g}),$ for some $p^{g} \in \Lambda^{g},$ where $p^{g}$ is a hyperbolic continuation of $p \in h(\theta^-).$

\vspace{10mm}

Frequently, we talk on a horseshoe and  its symbolic conjugate. We will transit between these two contexts freely - we will not be worried about the formal syntax of our sentences since their meaning will be precise. We say, for example, that $\theta^- \in \Sigma^-$ is a leaf (because its conjugate, in the horseshoe, is a leaf); we say that $\underline{\theta} \in \Sigma^*$ is a cylinder (for the same reason); we say that $\theta^- \in \underline{\theta}^-$ (meaning $h(\theta^-) \subset h(\underline{\theta})^-$ - this signifies saying $\theta^-$ finishes with $\underline{\theta}^-$);

Note: Along this work, the symbol $\asymp$ used between two functions ($r(x) \asymp s(x)$) means that there is a constant $k > 1$ such that $\displaystyle{k^{-1} \leq \frac{|r(x)|}{|s(x)|} \leq k},$ for every $x$ in the intersection of the domains of these functions.

\section{The upper stable dimension}
\label{Dimensoes_superior_estavel_e_de_Hausdorff}

In the sequel we define the upper stable dimension of a horseshoe. This concept of fractal dimension - taken from \cite{MPV} - has easier manipulation than the Hausdorff dimension. In general, it's not difficult to provide natural upper bounds for the Hausdorff dimension, but to find natural lower bounds for the Hausdorff dimension seems to be a difficult task - it would be necessary, in principle, to obtain additional informations on the geometry of $\Lambda.$ In this sense we believe that the present work provides some useful tools - Marstrand-like theorems, criterion of the recurrent compact set and the probabilistic argument - for the treatment of questions concerning the Hausdorff dimension for hyperbolic sets since it aims to describe, in a certain view, the relative positions of points in the horseshoe.

It's possible, as in \cite{BDV}, to prove that the Hausdorff dimension of horseshoes in ambient dimension higher than 2 can varies discontinuously in general. However, it's not known whether these discontinuities of the Hausdorff dimension may happen robustly. Furthermore, it's not known whether the Hausdorff dimensions of stable Cantor sets (intersections of a horseshoe with some local stable manifold of a point in it) in ambient dimension higher than $2$ depends on the stable manifold. These two issues are solved in dimension $2$. In particular, the fact that the Hausdorff dimension of stable Cantor sets of horseshoes in dimension $2$ keeps constant as we vary the stable manifold in which it lives was useful to stablish a criterion (Hausdorff dimension of the original horseshoe smaller than $1$) for the prevalence of hyperbolicity at the initial bifurcating parameter in homoclinic bifurcations in dimension 2, as shown by Palis and Takens. As a comprehensive reference on this subject we suggest the book \cite{PT}.

It's worthwhile to remember that we denote by $\underline{\theta} := ( \theta_1, ... , {\theta}_n ) \in \Sigma^{+*}$ the typical (vertical) cylinder of the horseshoe - the name is a reference to $\displaystyle{V_{\underline{\theta}} := \bigcap_{i=1}^{n} f^i(P_{{\theta}_i})}.$

The definition of upper stable dimension (proposed in \cite{MPV}) consists in adapting the dimension formula (ver \cite{P}) of dynamically defined Cantor sets in dimension 1.

\begin{defn}{(Upper stable dimension)}:

Given a vertical cylinder, $\underline{\theta} \in \Sigma^{+*},$ we define its diameter by $D_s(\underline{\theta}) := \sup_{\theta^-; \underline{\theta} \in \Sigma_{\theta^-_0}} \Bigl\{ d_s(\theta^-,\underline{\theta})  \Bigr\},$ where $d_s(\theta^-,\underline{\theta}) := \mbox{diam}(W_{\theta^-} \cap V_{\underline{\theta}}).$

Now we can define $\lambda_n$ by $\displaystyle{\sum_{\underline{\theta} \in \Sigma^{+n}} D_s(\underline{\theta})^{\lambda_n} = 1}$ and the \textbf{upper stable dimension} of $(f,\Lambda)$ by $\bar{d}_s(f,\Lambda) := \lim_{n \rightarrow \infty} \lambda_n$ (see \cite{MPV}).

\end{defn}

We note that although the upper stable dimension can depend on the diffeomorphism which define the horseshoe $(f,\Lambda)$ we will denote $\bar{d}_s(f,\Lambda)$ by $\bar{d}_s(\Lambda)$ unless it is not clear on what diffeomorphism is the set $\Lambda$ representing the horseshoe $(f,\Lambda)$ is defined.

According to \cite{MPV}, $\bar{d}_s$ is upper semicontinuous. We prove $\bar{d}_s$ is continuous in the horseshoes having a splitting of its tangent bundle with a weak-stable subbundle with dimension $1$ and with contraction weaker than its strong-stable subbundle.

\begin{prop}

\

The upper stable dimension $\bar{d}_s$ is continuous in the horseshoes having a sharp splitting of its tangent bundle.

\end{prop}

\begin{proof}

In first place, we prove that $\bar{d}_s$ is upper semicontinuous in $(f,\Lambda).$ In order to do this, we have to prove that, for $n$ sufficiently big, $\lambda_n \geq \bar{d}_s.$ With this in hands, we only have to observe that $\lambda_n^{g}$ varies continuously with $g,$ that $\bar{d}_s(\Lambda^g) \leq \lambda_n^g$ and that $\lim_{n \rightarrow \infty} \lambda^f_n = \bar{d}_s(\Lambda)$ to conclude that $\bar{d}_s$ is upper semicontinuous in $(f,\Lambda).$

We observe that $D_s(\underline{c})D_s(\underline{d}) \asymp D_s(\underline{c}\underline{d}),$ since $D_s(\underline{\tau}) \asymp d_s(\theta^-,\underline{\tau}),$ for every $\underline{\tau}$ and $\theta^-$ such that $\underline{\tau} \in \Sigma^*_{\theta^-}$ because $df$ has bounded distortion on the directions transversal to the strong-stable one. So, there is $c$ with $0 < c < 1$ such that $c D_s(\underline{c}) D_s(\underline{d}) \leq D_s(\underline{c} \underline{d}) \leq c^{-1} D_s(\underline{c}) D_s(\underline{d}),$ for every $\underline{c}$ and $\underline{d}.$

By definition of $\lambda_n,$ $\displaystyle{ \sum_{\underline{\theta} \in \Sigma^{n}} D_s(\underline{\theta})^{\lambda_n} = 1.}$ Therefore, $\displaystyle{\left( \sum_{\underline{\theta} \in \Sigma^{n}} D_s(\underline{\theta})^{\lambda_n} \right)^k = 1,}$ and so \\ $\displaystyle{\sum_{\underline{\theta} \in \Sigma^{kn}} c^{(k-1)\lambda_n} D_s(\underline{\theta})^{\lambda_n} \leq 1,}$ since $D_s(\underline{\theta}^1...\underline{\theta}^k) \leq c^{-(k-1)}D_s(\underline{\theta}^1).....D_s(\underline{\theta}^k)$ for every $\underline{\theta}^1,...,\underline{\theta}^k \in \Sigma^{n}$ satisfying $\underline{\theta}^1...\underline{\theta}^k \in \Sigma^{nk}.$

Now, if $n$ is sufficiently big, then $D_s(\underline{\theta})^{\varepsilon} \leq c^{k-1}$ for any $\underline{\theta} \in \Sigma^{kn}.$

Henceforth, $\displaystyle{\sum_{\underline{\theta} \in \Sigma^{kn}} D_s(\underline{\theta})^{\lambda_n(1 + \varepsilon)} \leq 1},$ and since $\lambda_{kn}$ satisfies $\displaystyle{\sum_{\underline{\theta} \in \Sigma^{kn}} D_s(\underline{\theta})^{\lambda_{kn}} = 1,}$ we have \\ $\lambda_{kn}(1+\varepsilon) \leq {\lambda_n}.$ By making $k \rightarrow \infty,$ we conclude that $\lambda_n \geq \bar{d}_s.$

Now, let's prove that $\bar{d}_s$ is lower semicontinuous. For this sake we will create a sequence, $(\tilde{\lambda}_n)_{n \geq 1},$ such that $\lim_{n \rightarrow \infty} \tilde{\lambda}_n = \bar{d}_s$ and such that for any $\varepsilon > 0,$ if $n$ is sufficiently big and $g$ sufficiently $C^1-$close to $f,$ then $(1-\varepsilon)\tilde{\lambda}^g_n < \bar{d}_s(\Lambda^g).$ Then, we only have to observe, as before, that $\tilde{\lambda}_n^g$ varies continuously to conclude that $\bar{d}_s(\Lambda^g)$ is lower semicontinuous.

Let $r \in \mathbb{N}$ be sufficiently big in such a way that for every $c$ and $d$ in $\Sigma^1,$ there is an admissible word with $r$ letters beginning with $c$ and finishing with $d$ and let $a,b \in \Sigma^1$ be chosen in such a way that $ba$ is admissible in $\Sigma.$ We define $\tilde{\lambda}_n$ in such a way that it satisfies $$\displaystyle{\sum_{ \substack{ \underline{\theta} \in \Sigma^{n} \\ \theta_1 = a, \theta_n = b }} D_{s}(\underline{\theta})^{\tilde{\lambda}_n} = 1, \mbox{ for every } n \geq 2r.}$$

We prove that $(1-\varepsilon)\tilde{\lambda}^g_n \leq \bar{d}_s(\Lambda^g)$ if $n$ is sufficiently big and if $g$ is sufficiently $C^1-$close to $f.$

As $\displaystyle{ \sum_{  \substack{\underline{\theta} \in \Sigma^{n-2} \\ a\underline{\theta}b \in \Sigma^n }  } D_s(a \underline{\theta}b)^{\tilde{\lambda}_n} = 1},$ then, for every $k \geq 1,$ $\displaystyle{ \left( \sum_{ \substack{\underline{\theta} \in \Sigma^{n-2} \\ a\underline{\theta}b \in \Sigma^n } } D_s(a \underline{\theta} b)^{\tilde{\lambda}_n} \right)^k = 1}.$

That is, for every $k \geq 1,$ $\displaystyle{ \sum_{ \substack{ \underline{\theta}^1,..., \underline{\theta}^k \in \Sigma^{n-2} \\ a\underline{\theta}^1b,..., a\underline{\theta}^kb \in \Sigma^{n} }} D_s(a \underline{\theta}^1 b)^{\tilde{\lambda}_n}.....D_s(a \underline{\theta}^k b)^{\tilde{\lambda}_n} = 1}.$

As $a\underline{\theta}^1 b ... a \underline{\theta}^k b$ is admissible (since $ba$ is), then

\begin{tabular}{rl}

$\displaystyle{ \sum_{ \substack{ \underline{\theta}^1,..., \underline{\theta}^k \in \Sigma^{n-2} \\ a\underline{\theta}^1b,..., a\underline{\theta}^kb \in \Sigma^{n} } }  D_s(a \underline{\theta}^1 b)^{\tilde{\lambda}_n} ... D_s(a \underline{\theta}^k b)^{\tilde{\lambda}_n}}$ & $\displaystyle{ \leq \sum_{ \substack{ \underline{\theta}^1,..., \underline{\theta}^k \in \Sigma^{n-2} \\ a\underline{\theta}^1b,..., a\underline{\theta}^kb \in \Sigma^{n} } }  \left( c^{-(k-1)} D_s(a \underline{\theta}^1 b ... a \underline{\theta}^k b) \right)^{\tilde{\lambda}_n}  }$ \\ \\

$$ & $\displaystyle{ \leq \sum_{ \substack{ \underline{\theta} \in \Sigma^{kn-2} \\ a\underline{\theta}b \in \Sigma^{kn} } } c^{-(k-1) \tilde{\lambda}_n } D_s(a \underline{\theta} b)^{\tilde{\lambda}_n} }.$

\end{tabular}

Henceforth, if $n$ is sufficiently big, in such a way that $D_s(a \underline{\theta} b)^{-\varepsilon} > c^{-(k-1)}$ (this happens robustly in $g \in C^1$), then $\displaystyle{ \sum_{ \substack{ \underline{\theta} \in \Sigma^{kn-2} \\ a\underline{\theta}b \in \Sigma^{kn} } } D_s(a \underline{\theta} b)^{\tilde{\lambda}_n(1-\varepsilon)} \geq 1}$ and, since

$\displaystyle{ \sum_{ \substack{ \underline{\theta} \in \Sigma^{kn-2} \\ a\underline{\theta}b \in \Sigma^{kn} } } D_s(a \underline{\theta} b)^{\tilde{\lambda}_{kn}} = 1},$ then $\tilde{\lambda}_{kn} > \tilde{\lambda}_n(1-\varepsilon)$ for every $k \geq 1.$

As $\tilde{\lambda}_n^g$ depends continuously on $g$ in the $C^1$ topology, then, with no loss of generality, \\ $\tilde{\lambda}^g_{kn} \geq \tilde{\lambda}^g_n(1-\varepsilon)$ for every $k \geq 1$ and every $n$ sufficiently big and for all $g$ sufficiently $C^1-$close to $f.$ Henceforth, $\tilde{\lambda}^g_n(1-\varepsilon) \leq \limsup_{k \rightarrow \infty} \tilde{\lambda}^g_{kn} $ for every $g$ sufficiently $C^1-$close to $f.$

Now, since $\tilde{\lambda}_n \leq \lambda_n$ and $\lim_{n \rightarrow \infty} \lambda_n = \bar{d}_s,$ then $\tilde{\lambda}^g_n(1-\varepsilon) \leq \bar{d}_s(g)$ for every $g$ sufficiently $C^1-$close to $f$ and $n$ sufficiently big.

Now, we prove that $\lim_{n \rightarrow \infty} \tilde{\lambda}_n = \bar{d}_s.$ For this sake we prove that for every $\varepsilon > 0,$ \\ $\tilde{\lambda}_n \geq \lambda_{n - 2r}(1 - \varepsilon)$ if $n$ is sufficiently big. This is the same as saying that $\tilde{\lambda}_n > \bar{d}_s(1 - \varepsilon)$ if $n$ is sufficiently big, since $\lim_{n \rightarrow \infty } \lambda_n = \bar{d}_s.$

There is $0 < c < 1$ such that $\displaystyle{\sum_{\underline{\theta} \in \Sigma^{n-2r}} (c D_s(\underline{\theta}))^{\tilde{\lambda}_n} \leq \sum_{\underline{\theta} \in \Sigma^{n-2r}} D_s(\underline{a}^{\underline{\theta}} \underline{\theta} \underline{b}^{\underline{\theta}})^{\tilde{\lambda}_n},}$ where $\underline{a}^{\underline{\theta}}, \underline{b}^{\underline{\theta}} \in \Sigma^r$ are such that $a^{\underline{\theta}}_1 = a,$ $b^{\underline{\theta}}_r = b$ and $\underline{a}^{\underline{\theta}} \underline{\theta} \underline{b}^{\underline{\theta}}$ is admissible, since $D_s(\underline{c})D_s(\underline{d}) \asymp D_s(\underline{c}\underline{d})$ and $r$ is constant.

Henceforth, as $\displaystyle{ \sum_{\underline{\theta} \in \Sigma^{n-2r}} D_s(\underline{\theta})^{\tilde{\lambda}_n(1 + \varepsilon)} \leq \sum_{\underline{\theta} \in \Sigma^{n-2r}} (c D_s(\underline{\theta}))^{\tilde{\lambda}_n} }$ if $n$ is chosen sufficiently big, then $\displaystyle{ \sum_{\underline{\theta} \in \Sigma^{n-2r}} D_s(\underline{\theta})^{\tilde{\lambda}_n(1 + \varepsilon)} \leq 1,}$ which implies $(1 + \varepsilon) \tilde{\lambda}_n \geq \lambda_{n - 2r}.$

\end{proof}

\section{ The criterion of the recurrent compact set and its consequences }
\label{sec_top}

We state in this section our main result. It guarantees that close to any horseshoe, $(f,\Lambda),$ of class $C^k$ ($k \in \mathbb{N}^* \cup \{ \infty \} $) satisfying $\bar{d}_s(\Lambda) > 1$ there is a hyperbolic continuation of class $C^{\infty},$ $(g,\Lambda^g),$ which is $C^k-$close to the original horseshoe that satisfies the criterion of the recurrent compact set. This will imply some geometric properties as the existence of blenders.

\subsection{Recurrent compact sets and the criterion of the recurrent compact set}

This concept was introduced in \cite{MY} in order to prove that stable intersections of regular Cantor sets are dense in the region where the sum of their Hausdorff dimensions is bigger than 1. We develop a version of this criterion on the context of horseshoes and conclude that it will imply some geometric properties for the horseshoes satisfying it - among them we highlight the existence of blenders. In our case, this criterion is related to the renormalization operator - which essencially expands the pieces (intersections of local stable manifolds with vertical cylinders) by the inverse application of the diffeomorphism, and project them along the strong stable foliation. If we may apply renormalization operators indefinitely, the domais of the iterations of these operators will be a nested sequence of sets converging to some point in the horseshoe.

Before introducing the criterion we need stablish some concepts involved in its definition.

Let $(f,\Lambda)$ be a $C^{\infty}$ horseshoe. For each element $P$ in the Markov partition $\mathcal{P}$ which is associated to this horseshoe, we fix a point $x_P \in P \cap \Lambda$ and a submanifold, $\mathcal{H}_P,$ with dimension 2 transversal to $E^{ss}(x_P)$ such that for every $x \in P,$ $\mathcal{F}^{ss}_{loc}(x) \cap \mathcal{H}_P = \mathcal{F}^{ss}_{loc} \pitchfork \mathcal{H}_P$ and consists of exactly one point in the interior of $\mathcal{H}_P.$ We observe that if $g$ is $C^1-$close to $f$ and if the partition $\mathcal{P}$ is composed by sufficiently small elements, then $\mathcal{F}_{loc}(x) \cap \mathcal{H}_P$ is exactly one point in the interior of $\mathcal{H}_P$ for any foliation $\mathcal{F}$ sufficiently $C^1-$close to $\mathcal{F}^{ss}$ and for every $x \in P.$ We denote the union of these submanifolds by $\mathcal{H} := \bigcup_{P \in \mathcal{P}} \mathcal{H}_P,$ which we call the wall.

We denote $\mathcal{H} \cap W_{loc}^{g,s}(\Lambda)$ by $H^g,$ $\mathcal{H} \cap W^g_{\theta^-}$ by $H^g_{\theta^-}$ and the projection of $W^g_{\theta^-}$ on $H^g_{\theta^-}$ along the strong stable foliation of $g$ by $\Pi^g_{\theta^-}.$

We observe that $H$ is diffeomorphic to the cartesian product of a Cantor set and a interval and that $H^g_{\theta^-}$ is $C^1-$close to $H_{\theta^-}$ for every $g$ sufficiently $C^1-$close to $f.$ We identify, under this viewpoint, $H$  with $H^g = \mathcal{I} \times \mathcal{K}$ for every $g$ sufficiently $C^1-$close to $f,$ where $\mathcal{K}$ is a Cantor set (which corresponds topologically to the unstable cantor set of $\Lambda$) and $\mathcal{I} = H_{\theta^-} =  H^g_{\theta^-}$ is an interval.

Now, we define the renormalization operators which will have a central role in the definition of the criterion of the recurrent compact set.

\begin{defn}{Renormalization operator}

The \textbf{renormalization operator} corresponding to the tube $\underline{a} \in \Sigma^{+*}$ of $g$ is defined by $R^{g}_{\underline{a}}: H \rightarrow H,$ where

$\displaystyle{ R^{g}_{ \underline{a}} (x,\theta^-)=\left\{
      \begin{array}{l}
       \Pi^{g}_{\theta^- \underline{a}} \Biggl( g^{-|\underline{a}|} \biggl( (\Pi^{g}_{\theta^-})^{-1}(x) \cap h^{g}(\theta^{-},\underline{a}) \biggr) \Biggr),\,\,\mbox{if}\,\, x \in int \Biggl(\Pi_{\theta^{-}}^{g} \Bigl( h^g(\theta^{-},\underline{a}) \Bigr) \Biggr) \\
       \infty,\,\,\mbox{otherwise}
       \end{array} \right.}$

\end{defn}

\begin{figure}[h]
    \centering
		\includegraphics[width=9cm]{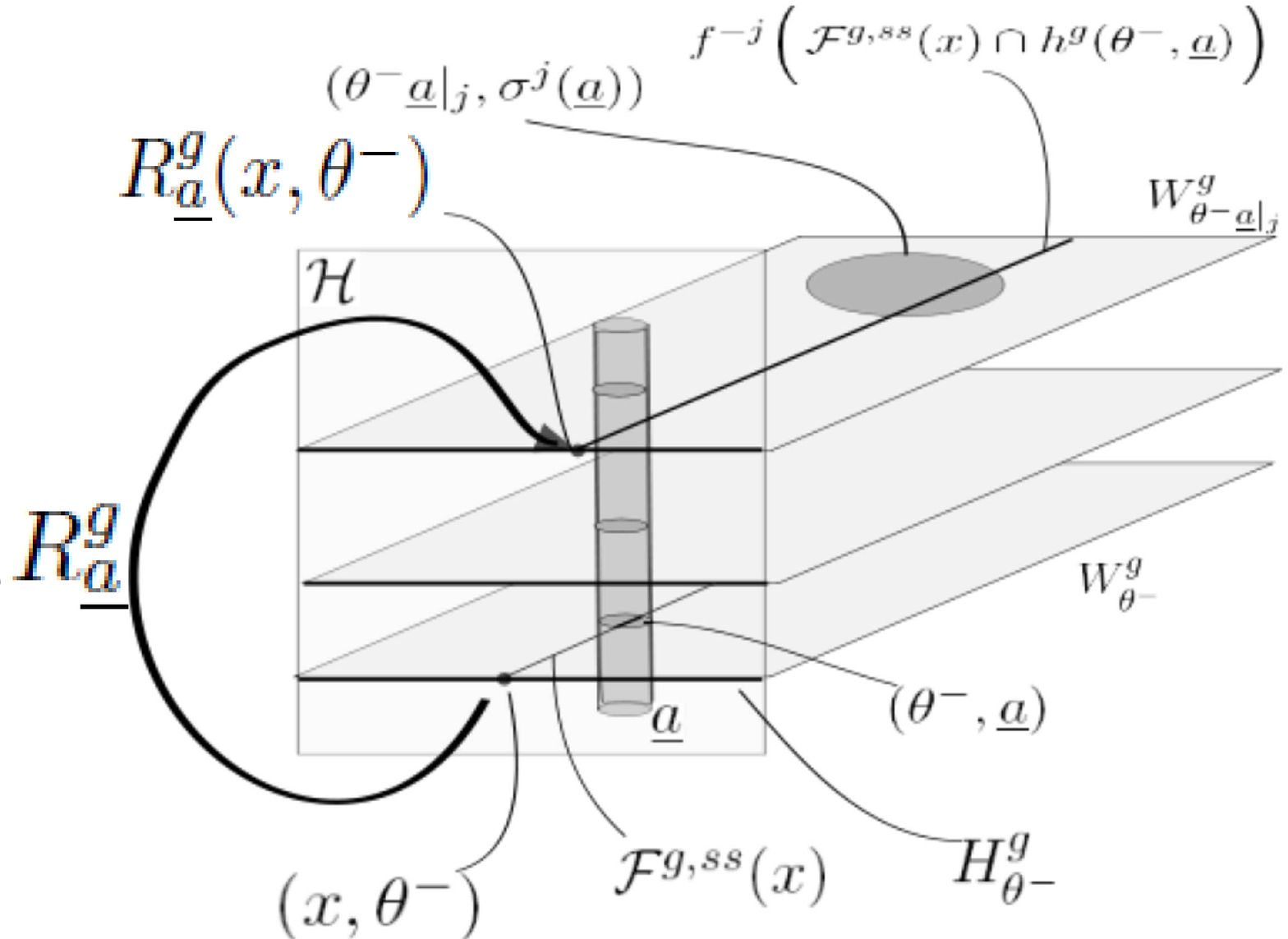}
\end{figure}

Now we introduce the criterion of the recurrent compact set.

\begin{defn}{Recurrent compact set}

A compact subset $K$ in $H$ is said \textbf{recurrent compact} for $g$ if for every $(x, \theta^-) \in K,$ there is $\underline{a} \in \Sigma^{*}_{\theta_0^-}$ such that $R^{g}_{\underline{a}}(x, \theta^-) \in int(K).$

\end{defn}

We use the notation $K_{\theta^-}$ to represent $K \cap H_{\theta^-}.$

\begin{remark}

We say that the horseshoe $(f,\Lambda)$ satisfies the criterion of the recurrent compact set if it has a recurrent compact set.

\end{remark}

\begin{prop}{Robustness of the criterion of the recurrent compact set}

The criterion of the recurrent compact set is robust, i.e., every horseshoe of class $C^{\infty}$ which is sufficiently $C^1-$close to the original horseshoe satisfies the criterion of the recurrent compact set with the same original recurrent compact set.

\end{prop}

\begin{proof}

For every $p \in K,$ there is a vertical cylinder $\underline{a} = \underline{a}(p)$ such that $R_{\underline{a}}(p)$ is inside $int(K).$ By continuity of $R_{\underline{a}},$ there is a neighbourhood $W(p)$ of $p$ such that $\overline{R_{\underline{a}}(W(p))} \subset int(K)$ and there is ${\delta}(p) > 0,$ such that if $\| f - g\|_{C^1} < {\delta}(p),$ then $R^g_{\underline{a}(p)}(x,\theta^-) \in int(K)$ for every $(x,\theta^-) \in W(p).$ As $K$ is compact, there is a finite covering, $\{ W(p_1),...,W(p_m) \},$ for $K$ and $\{ \delta(p_1),..., \delta(p_m) \}$ such that defining ${\delta} := min\{ {\delta}(p_i), 1 \leq i \leq m \},$ if $\| f-g \|_{C^1} < {\delta},$ then $R^g_{\underline{a}(p_i)}(x,\theta^-) \in int(K)$ for every $1 \leq i \leq m$ and $(x,\theta^-) \in W(p_i).$ This means, since $\{ W(p_j), 1 \leq j \leq m \}$ is a cover to $K,$ that $K$ is recurrent compact for every $g$ that is ${\delta}$ $C^1-$close to $f.$

\end{proof}

Now, we state some consequences of the criterion of the recurrent compact set.

\subsection{Blenders}
\label{sec_cr_blende}

The blenders were introduced in \cite{BD} to exhibit a new class of diffeomorphisms $C^1-$robustly transitive and non-hyperbolic. Since then, the blenders had been shown to be useful in order to obtain some ergodic and topologic consequences (see \cite{RHRHTU} for an ergodic one). We present in the sequel the definition of blender - we observe this enunciate is under the influence by the commentary which follows the topic 'The main local property of the cs-blender' which is in section 1 of \cite{BD}.

\begin{defn}{\textit{Blender}}

We consider a horseshoe, $(f,\Lambda),$ in such a way that its tangent bundle has sharp splitting, $T_{\Lambda}M = E_{\Lambda}^{ss} \oplus E_{\Lambda}^{ws} \oplus E_{\Lambda}^{u},$ and an open set, $U,$ in $M.$ We say $(f,\Lambda,U)$ is a blender if there is a cone field, $C^{ss},$ continuous in $U$ and a real number $r > 0,$ such that any tangent curve to $C^{ss}$ with size bigger than $r$ does intersect $W^{g,u}(\Lambda^g)$ for every horseshoe $(g,\Lambda^g)$ suffciently $C^1-$close to $(f,\Lambda).$

\end{defn}

The known Blenders were constructed through a kind of skew-horseshoe (see \cite{BD}) and they had been found only close to heterodimensional cycles (\cite{BLDiaz}). We stablish a criterion for the existence of blenders - the criterion of the recurrent compact set.

\begin{thm}{criterion for the existence of \textit{blender}}
\label{crit_compac_blender}

Let $(f,\Lambda)$ be a horseshoe in dimension higher than 2, of class $C^{\infty},$ with sharp splitting of its tangent bundle, $T_{\Lambda}M = E_{\Lambda}^{ss} \oplus E_{\Lambda}^{ws} \oplus E_{\Lambda}^u,$ and with a recurrent compact set $K.$

For every horseshoe, $(g,\Lambda^g),$ $C^{1}-$sufficiently close to $f$ and any $C^1$ curve, $\ell,$ sufficiently $C^1-$close to some leaf of $\mathcal{F}_{loc}^{ss}$ passing through some point in $K,$ we have $\ell \cap W^{g,u}(\Lambda^g) \ne \emptyset.$

In particular, there is an open set $U$ in $M$ in the neighbourhood of $\mathcal{F}_{loc}^{ss}(K)$ such that $(g,\Lambda^g,U)$ is blender, for any $g$ $C^1-$sufficiently close to $f.$

\end{thm}

One consequence of this result is that the recurrent compact set is contained in the projection along the strong stable foliation of the horseshoe on the wall, since if $x \in K_{\theta^-}$ and \\ $y \in \mathcal{F}_{loc}^{ss}(x) \cap W^{u}(\Lambda),$ then $y \in \Lambda,$ since $\mathcal{F}_{loc}^{ss} \subset W^s_{loc}(\Lambda),$ which implies $y \in W^{s}_{loc}(\Lambda) \cap W^u_{loc} (\Lambda).$

\subsection{More consequences of the criterion of the recurrent compact set}
\label{sec_out_conseque}

\begin{thm}
\label{teo_proj}

\

Let $(f,\Lambda)$ be a horseshoe in dimension higher than 2, of class $C^{\infty},$ with sharp splitting of its tangent bundle, $T_{\Lambda}M = E_{\Lambda}^{ss} \oplus E_{\Lambda}^{ws} \oplus E_{\Lambda}^u,$ and with a recurrent compact set, $K.$

Then, for every horseshoe, $(g,\Lambda^g),$ $C^{1}-$close to $f,$ for every $x \in \Lambda^g$ and any function $C^1,$ $P:W^{g,s}_{loc}(x) \rightarrow \mathbb{R},$ satisfying $P'(x)v \ne 0$ for $v \in E^{ws}(x) \backslash \{0\}$ ( i.e. $P'(x) \ne 0$ and the level curve of $p$ passing through $x$ is transversal to $E^{ws}(x)),$ the set $int( P(\Lambda^g \cap W^s_{\varepsilon}(x)) )$ is non-empty for every $\varepsilon > 0.$ ($W_{\varepsilon}^s(x)$ is a neighbourhood, in $W^s_{loc}(x),$ of size $\varepsilon$ around $x$.)

\end{thm}

Let $(f,\Lambda)$ be a horseshoe of class $C^1$ with sharp splitting, $T_{\Lambda}M = E_{\Lambda}^{ss} \oplus E_{\Lambda}^{ws} \oplus E_{\Lambda}^u,$ and $\mathcal{P}$ be a Markov partition sufficiently thin in such a way that there is a wall, $\mathcal{H}$. We say that a $C^0$ foliation which is $C^1-$continuous, $\mathcal{F},$ is ``transversal'' for $f$ if it is transversal to $E_{\Lambda}^{ws},$ $\mathcal{ F} \subset \mathcal{F}^{s}$ (each leaf in $\mathcal{F}$ is contained in some leaf in $\mathcal{F}^{s})$ and each leaf in $\mathcal{F}$ pass through $\mathcal{H}$ once.

The next corollary asserts that close (and in the same leaf) to the projection of any point in the horseshoe, there is an interval of projections of the horseshoe.

\begin{thm}
\label{qualq_folh_qualq_abert}

\

Let $(f,\Lambda)$ be a horseshoe in dimension higher than 2, of class $C^{\infty},$ with sharp splitting, $T_{\Lambda}M = E_{\Lambda}^{ss} \oplus E_{\Lambda}^{ws} \oplus E_{\Lambda}^u,$ and having a recurrent compact set, $K.$

Then, for every horseshoe, $(g,\Lambda^g),$ $C^{1}-$close to $(f,\Lambda),$ the projection on the wall along any foliation, $\mathcal{F},$ transversal for $g$ contains intervals densely in $\mathcal{F}_{loc}(\Lambda^g) \cap H_{\theta^-}$ for every $\theta^- \in \Sigma^-.$

\end{thm}

\subsection{Main theorem: typically, there is a recurrent compact set when the horseshoes have upper stable dimension bigger than 1}

\

Now we can state our main theorem. It implies, typically in the horseshoes with upper stable dimension bigger than 1, the results stated in sections \ref{sec_cr_blende} and \ref{sec_out_conseque}.

\begin{thm}{Main theorem}
\label{exis_compac_recorren}

Let $(f,\Lambda)$ be a horseshoe in dimension higher than 2, of class $C^{\infty},$ with sharp splitting, $T_{\Lambda}M = E_{\Lambda}^{ss} \oplus E_{\Lambda}^{ws} \oplus E_{\Lambda}^u,$ satisfying $\bar{d}_s(\Lambda) >1.$

Then, there is a horseshoe, $(g,\Lambda^g),$ $C^{\infty}-$close to $f$ having a non-empty recurrent compact set.

\end{thm}

We postpone the proof of theorem \ref{exis_compac_recorren} to the final sections of this work - sections \ref{dem_exis_compac_recorren_1}, \ref{dem_exis_compac_recorren_2}, \ref{Argumento_do_tipo_Marstrand} and \ref{Argumento_probabilistico}. We remark that theorem \ref{exis_compac_recorren} guarantees the following corollary for theorems \ref{crit_compac_blender}, \ref{teo_proj} and \ref{qualq_folh_qualq_abert}.

\begin{cor}

\

Let $(f,\Lambda)$ be a horseshoe in dimension higher than 2, of class $C^k$ ($k \in \mathbb{N}^* \cup \{ \infty \}$), with sharp splitting, $T_{\Lambda}M = E_{\Lambda}^{ss} \oplus E_{\Lambda}^{ws} \oplus E_{\Lambda}^u,$ and satisfying $\bar{d}_s(\Lambda) >1.$ Then, there is a $C^1-$open set, $C^{k}-$close to $f$ such that for any horseshoe $(g,\Lambda^g)$ inside this open set, the following happens:

\begin{itemize}

\item for every $x \in \Lambda^g$ and any $C^1$ function, $P:W^{g,s}_{loc}(x) \rightarrow \mathbb{R},$ satisfying $P'(x)v \ne 0$ for every $v \in E^{ws}(x) \backslash \{0\}$ (i.e. $P'(x) \ne 0$ and the level curve of $p$ passing through $x$ is transversal to $E^{ws}(x)),$ then $int( P(\Lambda^g \cap W^s_{\varepsilon}(x)) ) \ne \emptyset$ for every $\varepsilon > 0.$

\item the projection along any foliation $\mathcal{F}$ transversal for $g$ contains intervals in $H$ densely in $\mathcal{F}(\Lambda^g) \cap H_{\theta^-}$ for every $\theta^- \in \Sigma^-.$

\item for any $C^1$ curve, $\ell,$ sufficiently $C^1-$close to some leaf in $\mathcal{F}_{loc}^{ss}$ passing through some point in $K,$ $\ell \cap W^u(\Lambda) \ne \emptyset.$ In other words, there is a open set $U$ in $M$ in the neighbourhood of $\mathcal{F}_{loc}^{ss}(K)$ such that $(g,\Lambda^g,U)$ is blender.

\end{itemize}

\end{cor}

\begin{remark}

The arguments in section 4 of \cite{MPV} guarantee that if a horseshoe $(f,\Lambda)$ of class $C^{\infty}$ in dimension higher than 2 satisfies $\bar{d}_s(\Lambda) > 1,$ then there is some hyperbolic continuation of it, $(g,\Lambda^g),$ $C^{\infty}-$close to the original one owing a subhorseshoe $(g,\tilde{\Lambda}^g)$ satisfying $\bar{d}_s(\tilde{\Lambda}^g) > 1$ and having sharp splitting, $T_{\Lambda}M =  E_{\Lambda}^{ss} \oplus E_{\Lambda}^{ws} \oplus E_{\Lambda}^{u}.$ In this manner, the same conclusions stated in this section can be stablished to the hyperbolic continuations of the subhorseshoe $\tilde{\Lambda}$ in $\Lambda,$ and henceforth to the hyperbolic continuation of the horseshoe itself.

\end{remark}

We also observe that the case $\bar{d}_s(\Lambda) < 1$ is pretty different. In fact, in this case the Hausdorff dimension of the projection of $\Lambda$ along $\mathcal{F}^{ss}$ in each leaf would be less than 1, since the projection is a Lipschitz function and $HD(\Lambda \cap W^s_{loc}(x)) \leq \bar{d}_s(\Lambda) < 1,$ for every $x \in \Lambda.$ In particular, the projection does not contain intervals. Also, in this case the horseshoe has no blenders.

\subsection{Proof of the consequences of the criterion of the recurrent compact set}

\begin{proof}[Proof of theorem \ref{crit_compac_blender}]

Suppose that $\ell$ is $\varepsilon$ $C^1-$close to $\mathcal{F}_{loc}^{ss}(x,\theta^-),$ where $(x,\theta^-) \in K.$ Let's prove that for any $g,$ $\delta$ $C^1-$close to $f,$ $W_{loc}^{g,u}(\Lambda^g) \cap \ell \ne \emptyset.$

For any $p \in K,$ there is $\delta = \delta(p) > 0$ and a vertical cylinder $\underline{a} = \underline{a}(p)$ such that the ball with radius $\delta > 0$ around $R_{\underline{a}}(p)$ (denoted by $B_{\delta}(R_{\underline{a}}(p))),$ is contained in $K.$ By continuity of $R_{\underline{a}},$ there is a neighbourhood $W(p)$ of $p$ such that $B_{\frac{\delta}{2}}(R_{\underline{a}}(W(p))) \subset K.$ As $K$ is compact, there is a finite covering, $\{ W(p_1),...,W(p_m) \},$ for $K$ such that if $n_0 := \max\{ |\underline{a}(p_j)|, 1 \leq j \leq m \},$ $\eta := min \{ \frac{\delta(p_j)}{2}, 1 \leq j \leq m \}$ and $C:=\{ \underline{a}(p_j), 1 \leq j \leq m \}$ are vertical cylinders associated to points in each open set in the fixed open covering, then for every $p \in K,$ there is a vertical cylinder $\underline{a} \in C$ with $| \underline{a} | < n_0$ such that $B_{\eta}(R_{\underline{a}}(p)) \subset int(K).$

Henceforth, if $\varepsilon > 0$ is chosen sufficiently small, then any curve, $\tilde{\ell},$ $\varepsilon$ $C^1-$close to $\mathcal{F}_{loc}^{ss}(x,\theta^-)$ has non-empty intersection with the same cylinder, $\underline{a} \in C,$ corresponding to the open set $W(p_j)$ to which $(x,\theta^-)$ belongs.

Moreover, if $\delta > 0$ is chosen sufficiently small, then for any diffeomorphism, $g,$ $C^1$ $\delta-$close to $f$ and for any $(x,\theta^-) \in K,$ we assert that if $\underline{a} \in C$ is a vertical sylinder corresponding to a fixed open covering which owns $(x,\theta^-),$ then $\left(g^{-|\underline{a}|}(\ell)\right)_{loc}$ is $C^1$ $\varepsilon-$close to $\mathcal{F}_{loc}^{ss}(\tilde{x},\theta^-\underline{a})$ for some $\tilde{x} \in K_{\theta^- \underline{a}}.$

To see this it's enough to observe that a strong-stable foliation attracts any foliation transversal to the weak-stable direction, in such a way that there is a constant $0 < \lambda < 1$ such that $\left(f^{-|\underline{a}|}(\ell)\right)_{loc}$ is $C^1$ $\lambda \varepsilon-$close to some leaf in the strong stable foliation passing through $B_{\frac{\eta}{2}}(R_{\underline{a}}(p))$ if $\varepsilon$ is chosen sufficiently small, since $|\underline{a}| < n_0$ and, then, \\ $dist_{C^0} \left( \left(f^{-|\underline{a}|}(\ell)\right)_{loc}, f^{-|\underline{a}|}(\mathcal{F}^{ss}(x,\theta^-)) \right) < c \varepsilon$ for some constant $c > 0$ and $c\varepsilon < \frac{\eta}{2}$ if $\varepsilon$ is sufficiently small. Henceforth, $\left(g^{-|\underline{a}|}(\ell)\right)_{loc}$ is $C^1$ $(\lambda \varepsilon + \delta)-$close to some leaf in the strong stable foliation of $f$ passing through some point in $B_{\eta}(R_{\underline{a}}(p))$ if $\delta$ and $\varepsilon$ are chosen sufficiently small, since $|\underline{a}| < n_0$ and, then, $dist_{C^0}\left( \left(f^{-|\underline{a}|}(\ell)\right)_{loc}, g^{-|\underline{a}|}(\ell)\right)_{loc} < \tilde{c} \varepsilon$ for some constant $\tilde{c} > 0$ and $\tilde{c}\varepsilon < \frac{\eta}{2}$ if $\varepsilon$ is sufficiently small. Therefore, if $\delta >0$ is chosen sufficiently small in such a way that $\lambda \varepsilon + \delta < \varepsilon$, then we can conclude that $\left(g^{-|\underline{a}|}(\ell)\right)_{loc}$ is $C^1$ $\varepsilon-$close to $\mathcal{F}_{loc}^{ss}(\tilde{x},\theta^-\underline{a})$ for some $\tilde{x} \in K_{\theta^- \underline{a}},$ since $B_{\eta}(R_{\underline{a}}(p)) \subset int(K).$

As $(x,\theta^-)$ is in $K,$ there is $\underline{a}^0 \in C$ such that $x_1:=R_{\underline{a}^0}(x_0,\theta^-) \in int(K_{\theta^-\underline{a}^0}).$ In other words, $\mathcal{F}_{loc}^{ss}(x,\theta^-) \cap \underline{a}^0 \ne \emptyset$ and $f^{-|\underline{a}^0|}(\mathcal{F}_{loc}^{ss}(x,\theta^-) \cap \underline{a}^0)_{loc} = \mathcal{F}_{loc}^{ss}(x_1,\theta^- \underline{a}^0).$

This implies there is $\tilde{x}_1 \in K_{\theta^- \underline{a}^0}$ such that $\mathcal{F}_{loc}^{ss}(\tilde{x}_1,\theta^-\underline{a}^0)$ is $C^1$ $\varepsilon-$close to $g^{-|\underline{a}^0|}(\ell)_{loc}.$

As $\tilde{x}_1 \in K_{\theta^- \underline{a}^0},$ there is a cylinder $\underline{a}^1$ such that $x_2:=R_{\underline{a}^1}(\tilde{x}_1,\theta^- \underline{a}^1) \in int(K_{\theta^-\underline{a}^1\underline{a}^2}).$ In other words, $\mathcal{F}_{loc}^{ss}(\tilde{x}^1,\theta^- \underline{a}^0) \cap \underline{a}^1 \ne \emptyset.$ This implies that $g^{-|\underline{a}^0|}(\ell \cap \underline{a}^0)_{loc} \cap \underline{a}^1 \ne \emptyset.$

Argumenting recursively, there is a sequence of vertical cylinders in $\Sigma^{+*},$ $\left( \underline{a}^j \right)_{j \geq 0},$ satisfying $g^{-\sum_{j=0}^{l}|\underline{a}^j|}(\ell \cap \underline{a}^0 ... \underline{a}^{l-1})_{loc} \cap \underline{a}^j \ne \emptyset$ for every $\ell \in \mathbb{N},$ where $\underline{a}^0 ... \underline{a}^{l-1} \in \Sigma^{+*}.$  Therefore, $\bigcap_{j \geq 0} \underline{a}^0 ... \underline{a}^{j} \cap \ell \ne \emptyset,$ where $\underline{a}^0 ... \underline{a}^{j} \in \Sigma^{-*}.$ This implies $\ell \cap W_{loc}^{g,u}(\Lambda^g) \ne \emptyset.$

\end{proof}

\begin{proof}[Proof of theorem \ref{qualq_folh_qualq_abert}]

We observe that if the foliation $\mathcal{F}$ were $C^1,$ this result would be a corollary of theorem \ref{teo_proj}. It's enough to look at the foliation $\mathcal{F}$ locally as a foliation by level curves of a function which satisfies the hypothesis of theorem \ref{crit_compac_blender}. This means the projection of $\Lambda_{\theta^-}^g$ along $\mathcal{F}$ contains intervals densely in $\mathcal{F}_{loc}(\Lambda^g) \cap H_{\theta^-}$.

To prove the result in the case in which the foliation $\mathcal{F}$ is $C^0$ and $C^1-$continuous, we need use theorem \ref{crit_compac_blender}.

Let $y \in \mathcal{F}_{loc}(x) \cap H_{\theta^-},$ where $x \in \Lambda^g \cap W_{\theta^-}$ and let $V$ be an open set in $H_{\theta^-}$ around $y.$ We prove that $\mathcal{F}_{loc}(\Lambda^g)$ contains intervals in $V$ for any $g$ sufficiently $C^1-$close to $f.$ For this sake, we fix some neighbourhood $U$ around $x$ in $W_{\theta^-}$ such that $\mathcal{F}_{loc}(U) \cap H_{\theta^-} \subset V$ and we prove that $\mathcal{F}_{loc}(\Lambda^g \cap U)$ contains intervals.

Let $\tilde{x} \in \Lambda^g$ be such that $\{ g^n(\tilde{x}) \}_{n \in \mathbb{Z}}$ is dense in $\Lambda^g$ and let $n_0 \in \mathbb{Z}$ be such that $g^{n_0}(\tilde{x})$ is sufficiently close to $x$ such that $W^{g,u}_{loc}(\tilde{x}) \cap W^g_{\theta^-}$ is exactly one point, $\hat{x},$ in $\Lambda^g \cap U.$ We observe that $\{ g^{-n}(\hat{x}) \}_{n \in \mathbb{N}}$ is dense in $\Lambda^g,$ since $\hat{x} \in W^u(\tilde{x})$ and $\{ g^n(\tilde{x})\}_{n \in \mathbb{Z}}$ is dense in $\Lambda^g.$

If $g$ is sufficiently $C^1-$close to $f,$ then there is $n_1 \in \mathbb{N}$ sufficiently big such that $g^{-n_1}(\mathcal{F}(\hat{x}))_{loc}$ is sufficiently $C^1-$close to some leaf in $\mathcal{F}_{loc}^{ss}$ passing through $int(K)$ in such a way that we can apply theorem \ref{crit_compac_blender} for $g^{-n_1}(\mathcal{F}(\hat{x}))_{loc}$ (and so, for any open set in the foliation $g^{-n_1}(\mathcal{F})_{loc}$ around the leaf $g^{-n_1}(\mathcal{F}(\hat{x}))_{loc}$).

Therefore, any leaf in this open set in $g^{-n_1}(\mathcal{F})_{loc}$ does intersect $W^u(\Lambda^g)$. But, as any of these leaves is in $W^{g,s}_{loc}(g^{-n_1}(\hat{x}))$ (we remember $\hat{x} \in \Lambda^g$), then any of these leaves does intersect $\Lambda^g.$ This means any leaf in some open set of the foliation $\mathcal{F}_{loc}$ around the leaf $\mathcal{F}_{loc}(\hat{x})$ does intersect $\Lambda^g.$ That is, $\mathcal{F}_{loc}(U \cap \Lambda^g)$ contains intervals in $V,$ since $\hat{x} \in U$ and $\mathcal{F}_{loc}({U}) \cap H_{\theta^-} \subset V.$

\end{proof}

\begin{proof}[Proof of theorem \ref{teo_proj}]

We observe that the level curves, $\tilde{\mathcal{F}}$, of $P|_{W^s_{\varepsilon}(x)}$ in $W^{s}_{\varepsilon}(x)$ are transversal to the weak-stable direction for every $\varepsilon > 0$ sufficiently small. Observe that if $n \in \mathbb{N}$ is sufficiently big, then $\mathcal{F} := g^{-n}(\tilde{\mathcal{F}})_{loc} \cap W_{loc}^{g,s}(g^{-n}(x))$ is a $C^1$ foliation for $W_{loc}^{g,s}(g^{-n}(x))$ such that its leaves passing through $\Lambda$ are transversal to $E^{ws}.$ We can proceed, therefore, in the same manner we did to demonstrate theorem \ref{qualq_folh_qualq_abert} in order to prove that $\mathcal{F}(\Lambda^g)$ contains intervals and that, therefore, $P(W_{\varepsilon}^s(x) \cap \Lambda^g)$ contains intervals.

\end{proof}

\section{Sketch of the proof of the Main Theorem}
\label{dem_exis_compac_recorren_1}

Lets prove that for any $k \geq 2,$ there is a horseshoe, $(g,\Lambda^g),$ $C^k-$close to $(f,\Lambda)$ satisfying the criterion of the recurrent compact set.

For any stable leaf and for any $\rho > 0,$ we obtain approximately $\rho^{-\bar{d}_s}$ disjoint pieces with aproximate size $\rho > 0$ - pieces with approximate diameter $\rho$ - in each one of these leaves. Done this, we can choose a positive fraction of these pieces (and, therefore, approximately $\rho^{-\bar{d}_s}$ pieces of approximate size $\rho$) in such a way that every one is in some stacking - a stacking is a set of pieces which intercept the same strong stable leaf - containing at least $\rho^{-(\bar{d}_s - 1)}$ pieces. We still can consider, with no loss of generality - after possibily some $C^{k}-$small perturbation (Marstrand-like argument) - that for most stable leaves the projection of pieces in these stackings along the strong stable foliation has Lebesgue measure bounded by below for some positive constant. We define the candidate to recurrent compact set as the projection on the wall along the strong stable foliation of these pieces.

Now, we create a perturbation family, $\{ f^{\underline{\omega}} \}_{\underline{\omega} \in \Omega}$ ($\Omega := [-1,1]^{\Sigma_1}$), with $|\Sigma_1|$ parameters ($|\Sigma_1| \gg 1$), $C^k-$small and we adapt the probabilistic argument to this perturbation family in order to prove that the probability, in $\Omega,$ that $K=K(\rho)$ is a recurrent compact set for $f^{\underline{\omega}}$ converges to 1 as $\rho$ converges to 0.

This perturbation family will be such that for every $(x,\theta^-) \in K_{\theta^-}$ and $\underline{a} \in \Sigma^{+*}$ satisfying $(x,\theta^-) \in int \left( \Pi_{\theta^-}(\underline{a}) \right),$ the events $\left\{ \underline{\omega} \in \Omega \mbox{ such that } R_{\underline{a}}^{\underline{\omega}}(x,\theta^-) \in K_{-\rho^2} \right\},$ (where $K_{-\delta}$ is a set $\{ (x,\theta^-) \in K $ such that its neighbourhoods with radius $\delta \mbox{ in } H \mbox{ are contained in } K \}$ are essencially mutually independent for at least $\rho^{-\frac{c}{k}(\bar{d}_s - 1)}$ (for some $c>0$) of those $\rho^{-(\bar{d}_s-1)}$ pieces $\underline{a}$'s in the same stacking and that $P_{\Omega}\biggl( \underline{\omega} \in \Omega \mbox{ such that } R^{\underline{\omega}}_{\underline{a}}(x,\theta^-) \in K_{-\rho^2} \biggr) > P$ for each one of these $\underline{a}$'s, where $P > 0$ is fixed. This is possible by forcing that modifying a coordinate corresponds to moving with a displacement with approximate size $\rho$ and with approximate constant speed each the corresponding block in the Markov partition of $\Lambda$ which is formed by blocks with diameter with approximate size $\rho^{\frac{1}{k}}.$ These displacements need to be independents for every piece in the same stacking. Done this, as the renormalization operator sends each of these pieces in one leaf, then the preimage corresponding to the strong stable leaf of $(x,\theta^-)$ has a displacement with approximate constant speed along the entire stable leaf in which it falls. In this way, as the projection of the pieces with approximate size $\rho$ projecting on the candidate for recurrent compact set in these leaves which the renormalization operator falls has Lebesgue measure bounded by below by some positive constant, then the probability that the renormalization operator falls in the $\rho^{1+\alpha}-$relaxed interior  of this projection (for some $\alpha > 0$) - which is a bite of the set $K$ - is bigger than some positive number, $P>0.$

\begin{figure}[h]
    \centering
		\includegraphics[width=8cm]{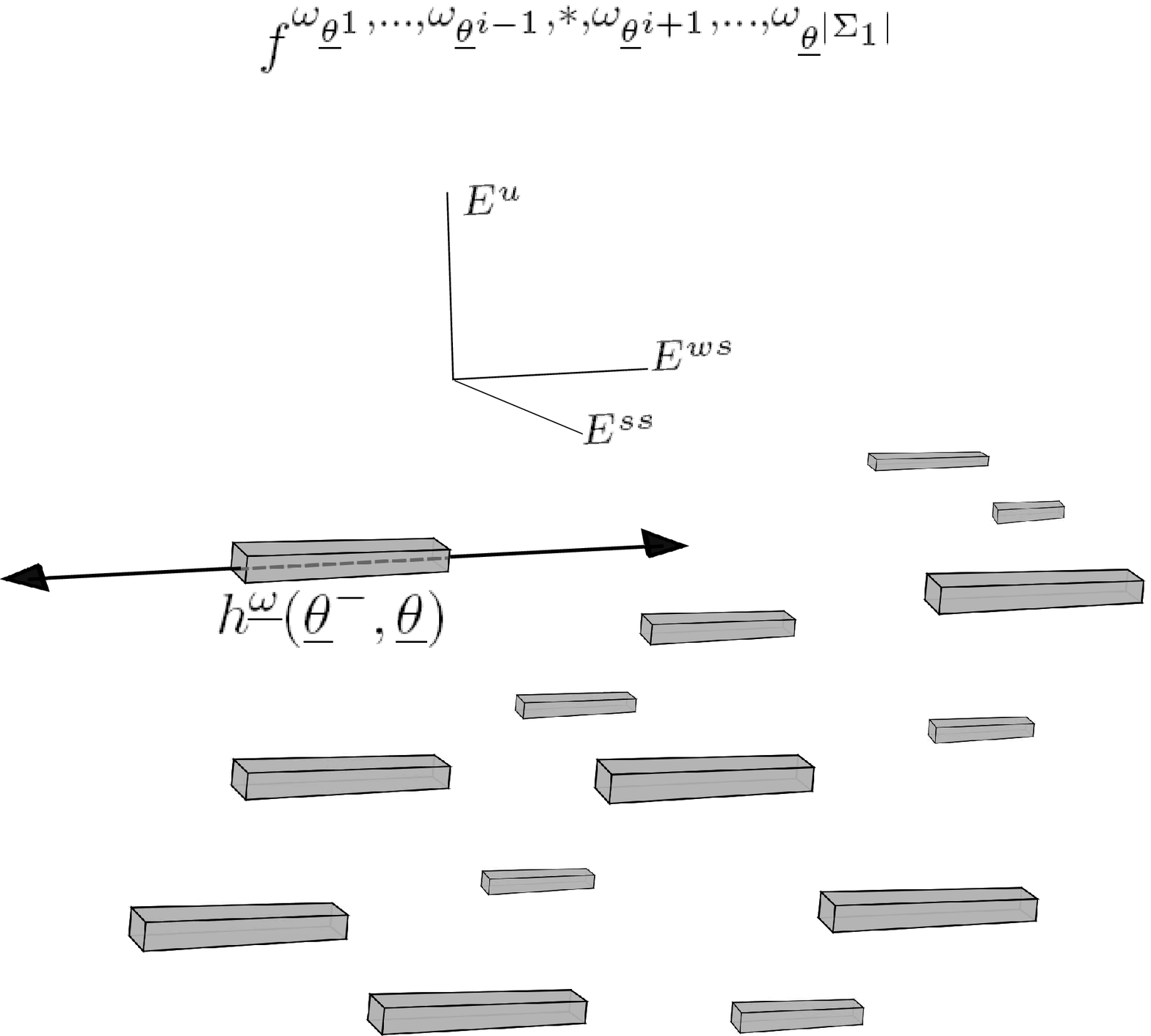}
		\caption{Perturbation of the probabilistic argument}
\end{figure}

Lets describe the probabilistic argument for these perturbation families. By independence, the probability that $R^{\omegab}_{\underline{a}}(x,\theta^-)$ does not falls in $K_{-\rho^{1+\alpha}}$ for all those $\rho^{-\frac{c}{k}(\bar{d}_s-1)}$ pieces $\underline{a}$'s above $(x,\theta^-)$ (pieces intersecting the strong stable leaf of $(x,\theta^-)$) is, roughly, smaller than \\ $(1-P)^{- \left( \rho^{-\frac{c}{k}(\bar{d}_s-1)} \right) }.$

We can decompose the set $K$ in approximately $\rho^{-4 + \alpha}$ rectangles with sides $\rho^{3+\alpha}$ by $\rho$ in such a a way that if $R_{\underline{a}}^{\underline{\omega}}(x,\theta^-) \in K_{-\rho^{1+\alpha}},$ then $R_{\underline{a}}^{\underline{\omega}}(\tilde{x},\tilde{\theta}^-) \in int(K),$ for every $(\tilde{x},\tilde{\theta}^-)$ in the corresponding rectangle in this decompostition containing $(x,\theta^-).$ Thus, the probability that there is some $\omegab \in \Omega$ such that for every $(x,\theta^-) \in K,$ there is some piece, $\underline{a}$, satisfying \\ $R^{\omegab}_{\underline{a}}(x) \in int(K)$ is bigger than $1 - \rho^{-4+\alpha} (1-P)^{- \left( \rho^{-\frac{c}{k}(\bar{d}_s-1)} \right) }.$ This probability congerves to 1 as the scale $\rho$ converges to zero. In this manner, for most $\omegab \in \Omega,$ $K$ is a recurrent compact set for ${f}^{\omegab}.$ This finishes the proof of the main theorem.

Now we discuss some shortcuts we used in this sketch. We need $C^k-$small perturbations ($\forall k \geq 2$) and the pieces in the same stacking moving essencially independently when modifying the coordinate associated to them. For this sake we need to assure a lot of space between these pieces to make the perturbations - a space with size roughly $\rho^{\frac{1}{k}}$ will be sufficient for our purposes. In order to obtain it we find, first of all, stackings with pieces with approximate size $\rho^{\frac{c}{k}}$ (for some $0 < c < 1$ depending only on the non-conformalities of $df|_{E^s}$) and, later, we create a stacking with pieces with approximate size $\rho,$ contained in the previous stackings in such a way that for each of these stackings - with approximate size $\rho$ - we can find approximately $\rho^{-\frac{c}{k}(\bar{d}_s - 1)}$ of its pieces well distributed - distributed in roughly $\rho^{-\frac{c}{k}(\bar{d}_s - 1)}$ disjoint pieces of the stackings formed with pieces of approximate size $\rho^{\frac{c}{k}}$ obtained in the first step.

\begin{figure}[h]
    \centering
		\includegraphics[width=10cm]{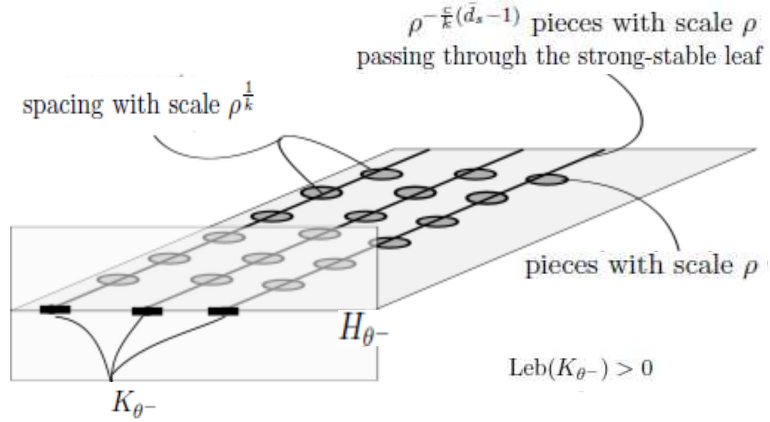}
		\label{fig:Marstrand1}
		\caption{Well-spaced stackings}
\end{figure}

Beyond it, as we must analysie the displacement of the pieces with respect to the strong stable leaves, we retire of our considerations the very recurrent stable leaves - those which in a short time interval return close to itself by forward iterates by the diffeomorphism. We observe in this way we eliminate few leaves. We still have another problem: we must avoid - to get independence of the displacements of pieces in the same stacking - very recurrent pieces in the same stacking since they can suffer double effect of the perturbation associated to the corresponding parameter or suffer effect of the perturbation of some coordinate associated to another piece in the same stacking. For this sake, we eliminate the pieces with approximate size $\rho$ whose preimages returns in a short time interval in some neighbourhood in which the piece lives. These pieces form a small fraction of those $\rho^{-\bar{d}_s}$ pieces which were in the leaf before we construct the stackings in such a way that we still can construct the mentioned stackings with these few recurrent pieces. We observe we can choose these few recurrent pieces in such a way that they still return, by the `renormalization' in leaves which had not been eliminated since there are still a positive proportion of these initial leaves.

\vspace{10mm}

Now we describe the Marstrand-like argument. We give this name to the first perturbation we will make in this work. It will be useful to find a diffeomorphism, $\tilde{g},$ $C^k-$close to the original, $f,$ satisfying the property we name \textbf{Marstrand-like}: For too many stable leaves, $\theta^- \in \Sigma^-,$ there is a measure $\nu^{\tilde{g}}_{\theta^-}$ in $H_{\theta^-}$ supporting the projection of $\Lambda$ such that its Radon-Nykodin derivative with respect to the Lebesgue measure, $Leb,$ in $H_{\theta^-}$ is $L_2$ and has $L_2$ norm bounded by above for all of these leaves, $\theta^-,$ uniformly (in particular, the projections of the horseshoe on the walls $H_{\theta^-}$ have non-zero Lebesgue measure).

In order to perform the first perturbation - Marstrand-like argument - we proceed according to the work \cite{SSU}. There, the authors describe sufficient conditions - transversality and distortion continuity (these concepts will be introduced later) - such that a multiparameter perturbation family of iterated function system (IFS) of contractions with bounded distortions in an interval or the line has invariant sets with positive Lebesgue measure for almost all multiparameters. We create a $N-$parameter perturbation family, $\{ f^{\underline{t}} \}_{\underline{t} \in I^N},$ on which to varies each coordinate in this multiparameter family corresponds to move, in the weak stable direction, each element in the Markov partition associated to the horseshoe, which will be chosen sufficiently thin if necessary.

\begin{figure}[h]
    \centering
		\includegraphics[width=7cm]{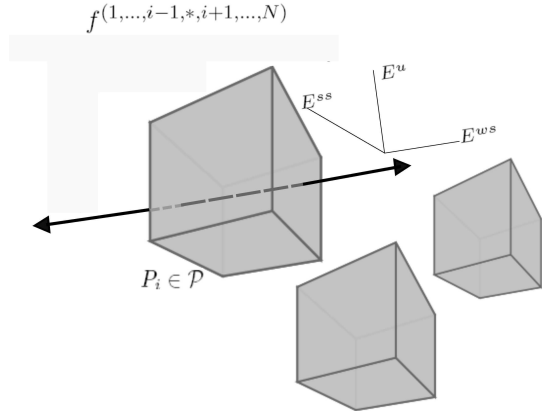}
		\caption{Perturbation for the Marstrand-like argument}
\end{figure}

From this family, and for each $\theta^- \in \Sigma^-,$ we create a $N-$parameter family of function system, $\{ \Phi_{\theta^-}^{\underline{t}} \}_{\underline{t} \in I^N},$ in which each $\Phi^{\underline{t}}_{\theta^-}$ is a function system $\left\{ \varphi_{(\theta^-,\underline{\theta})}^{\underline{t}} \right\}_{\underline{\theta} \in \Sigma^{*}_{\theta_0^-}}$ which consists, basically, that each function, $\varphi_{(\theta^-,\underline{\theta})}^{\underline{t}},$ is a contraction in $H_{\theta^-}$ whose image is the projection of the corresponding piece to the word $\underline{\theta}$ for $f^{\underline{t}}.$

\begin{figure}[h]
    \centering
		\includegraphics[width=7cm]{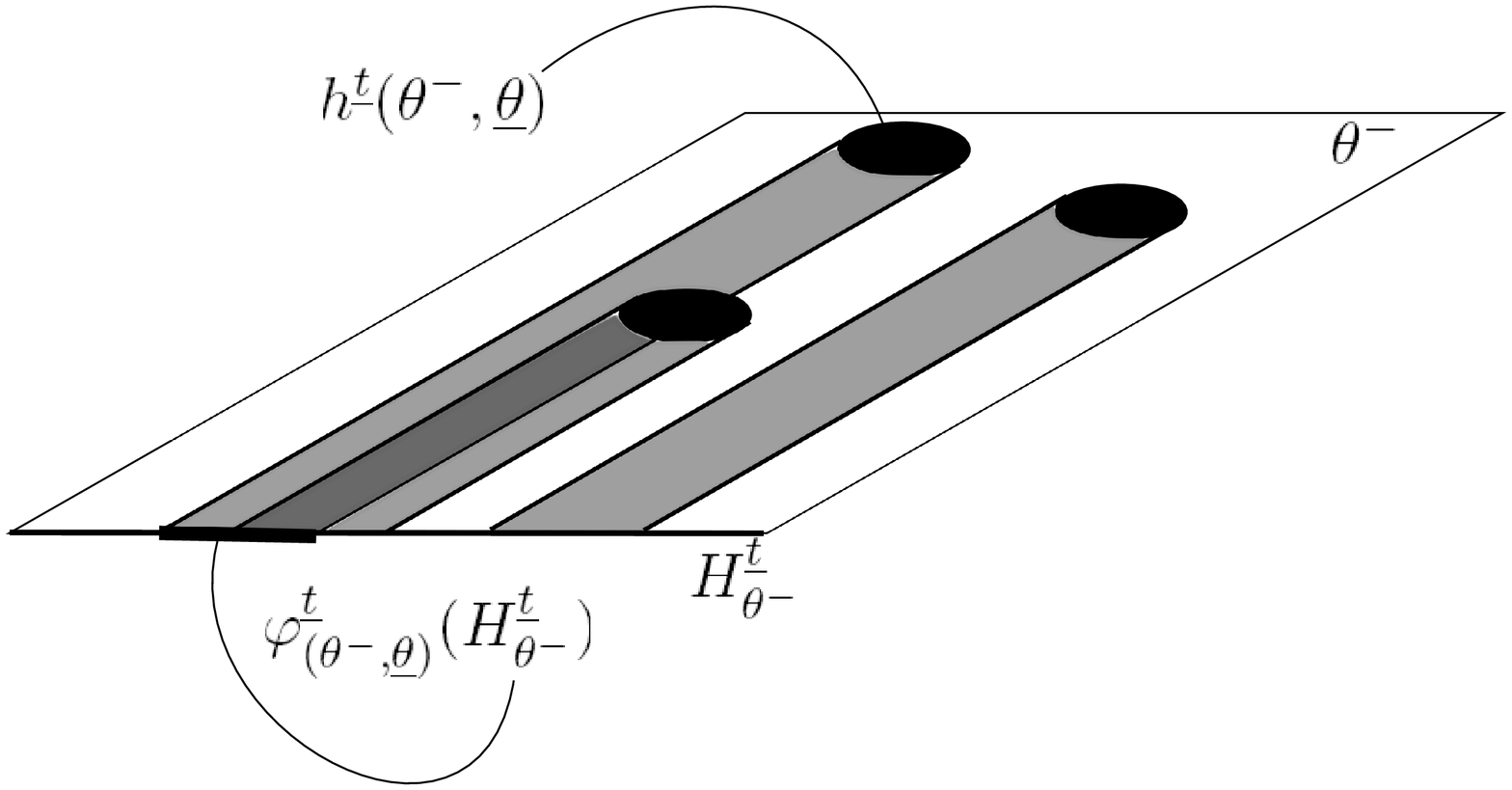}
\end{figure}

\vspace{60mm}

By adapting the arguments in \cite{SSU} we can conclude that for most leaves $\theta^- \in \Sigma^-$ and parameters $\underline{t} \in I^N,$ there are measures $\nu^{\underline{t}}_{\theta^-}$ supported in the projection of the horseshoe such that its Radon-Nykodin derivatives are $L_2$ with $L_2$ norms uniformly bounded by above. With this done, we choose one $\underline{t} \in I^N$ such that the norms $L_2$ of these Radon-Nykodin derivatives are bounded by above for most leaves $\theta^- \in \Sigma^-.$ For the $f^{\underline{t}}$ corresponding to this fixed $\underline{t}$ we begin the probabilistic argument.

\section{ Preparation for the Marstrand-like and probabilistic arguments }
\label{dem_exis_compac_recorren_2}

\subsection{ Some notations }

For each element $P \in \mathcal{P}$ we fix a local unstable manifold, $W_{loc}^u(x_P),$ of some point $x_P \in P.$

We define the distance between two leaves in the same partition by \\ $\mbox{dist}^-(\theta^-,\tilde{\theta}^-) := l(W_{\theta^-,\tilde{\theta}^-}^u(x_P)),$ where $l$ means length and $W_{\theta^-,\tilde{\theta}^-}^u(x_P)$ is the bite of $W_{loc}^u(x_P)$ connecting the leaves $\theta^-$ and $\tilde{\theta}^-.$ We observe the stable foliation is $C^1$ in such a way that we can guarantee that $l(\underline{\theta}^- \cap W^{u}_{loc}(\theta)) \asymp l(\underline{\theta}^- \cap W^{u}_{loc}(\hat{\theta}))$ for any $\theta,$ $\hat{\theta}$ in $\Sigma \cap \underline{\theta}^-,$ with $\underline{\theta}^- \in \Sigma^{-*}.$

We fix $c_1 > 1,$ and define the set of backward finite words with approximate size $\rho > 0$ (we can, also, name this set by blocks with approximate size $\rho$ of leaves) by $$\displaystyle{\Sigma^{-}(\rho) := \left\{ \underline{\theta}^- \in \Sigma^{*-} ; c_1^{-1} \rho \leq \mbox{diam}^-(\underline{\theta}^-) \leq c_1\rho \right\}}.$$

We denote by $\displaystyle{\Sigma^-\Sigma^{+*} := \bigcup_{\theta^- \in \Sigma^-} \bigcup_{\underline{\theta} \in \Sigma_{\theta^-_0}^* } (\theta^-,\underline{\theta})}$ the set of all the pieces.

We denote the projection by the diffeomorphism $g \in C^{\infty}$ of a piece $(\theta^-,\underline{\theta}) \in \Sigma^-\Sigma^{+*}$ by $\displaystyle{I^{g}_{(\theta^-,{\underline{\theta}})}:= \Pi^{g}_{\theta^-} ( {\underline{\theta}} ) },$ where $\Pi^{g}_{\theta^-}$ is the projection along the strong stable foliation of $g$ on the wall $H_{\theta^-}.$

Given a multiparameter family, $\bigl\{ f^{\underline{\gamma}} \bigr\}_{\underline{\gamma} \in \Gamma},$ we denote the pieces in this leaf $\theta^-$ with approximate size $\rho$ by $\Sigma_{\theta^-_0}(\rho)$ \index{$\Sigma_{\theta^-_0}(\rho)$}. These are the pieces in $\Sigma_{\theta^-_0}^{*}$ satisfying $c_1^{-1}\rho \leq | I^{{\underline{\gamma}}}_{(\tilde{\theta}^-,\underline{\theta})} | \leq c_1 \rho$ for every leaf $\tilde{\theta}^- \in \Sigma^-$ such that $\theta^-_0 = \tilde{\theta}^-_0$ and for each ${\underline{\gamma}} \in \Gamma.$

We say a cylinder has scale $\rho$ - the notation for these cylinders will be $\Sigma^+(\rho)$ - it its intersection with some stable leaf (and, therefore, with any of its intersecting leaves) has scale $\rho.$

We observe given a horseshoe $(f,\Lambda)$ of class $C^{k}$ ($k \geq 2$ or $k = \infty$), $d_s(h(\theta^-,\underline{\theta})) \asymp d_s(h^g(\theta^-,\underline{\theta}))$ for every $\underline{\theta}$ with diameter $\rho$ and any $g$ of class $C^k$ satisfying $\| g - f \|_k < \rho.$ Therefore, since $\mathcal{F}^{ss}$ is $C^1,$ $I_{(\theta^-,\underline{\theta})} \asymp I^{\underline{\gamma}}_{(\tilde{\theta}^-,\underline{\theta})}$ for any $\underline{\gamma} \in \Gamma,$ $\theta^- \in \Sigma^-$ and $\tilde{\theta}^- \in \Sigma^-$ with $\theta^-_0 = \tilde{\theta}^-_0$ for the two perturbation families we will create along the proof of the theorem.

\subsection{ Perturbations, non-recurrencies and the influence of these perturbations on non-recurrent pieces and leaves }

When developing the Marstrand-like and probabilistic arguments we use in two moments the lemmas which we enunciate in this section. They regard the effect on the movements of the pieces when the diffeomorphism is subject to a perturbation family. The main difficult we must overcome is that there can be so much recurrent pieces such that the influence a coordinate of the parameter in the perturbation family is unexpectable. The main reason for this problem is that the pieces can have preimages in the element of the Markov partition corresponding to the parameter we are make changes. We solve this problem by pulling out of our considerations such very recurrent pieces (which we name - and they are - recurrents). The same kind of problem occurs with the strong stable foliation - its unpredictability can be bigger than the size of the displacement we design for the pieces. As our objective is to control the displacements of the pieces relative to the leaves of the strong stable foliation, we also need to avoid this unpredictability. For this sake we eliminate from our considerations the very recurrent stable leaves.

In order to realize the perturbations we consider each element $P$ in the partition ${\cal P}$ is written, via some $C^k$ parametrization in such a way that $P$ is the box $[-1,1]^n$ and we denote by $e(P)$ a unitary vector in $E^{ws}(x_{P}) \backslash \{0\},$ where $x_P$ is some fixed point in $\Lambda \cap P,$ for every $P \in {\cal P}.$ We observe the Markov partition for $g$ is chosen to be the same as for $f$ if the perturbations are sufficiently small. To see this its enough to pick the partition formed by the compact neighbourhoods of the elements of some Markov partition.

Let $c_2>1$ and $c_3>0$ be fixed constants. We define in the sequel the model of perturbation families we will use in two moments along this work. \\ We consider $\Sigma^-(\alpha)\Sigma(\tilde{\alpha}) := \bigcup_{\underline{\theta}^- \in \Sigma^-(\alpha)} \bigcup_{\underline{\theta} \in \Sigma_{{\theta}_0^-}(\tilde{\alpha})} (\underline{\theta}^-, \underline{\theta})$

\begin{defn}{Model for the perturbation families}
\label{modelo_perturb}

Let $\alpha>0,$ $\tilde{\alpha} > 0,$ $\rho > 0$ and a partition, $\tilde{\Sigma},$ for $\Sigma$ formed by pieces, $(\underline{\theta}^-,\underline{\theta} )$ in $\Sigma^-(\alpha)\Sigma(\tilde{\alpha})$ be fixed. We say a perturbation family - $\{f^{\underline{\gamma}}\}_{\underline{\gamma} \in \Gamma},$ where $\Gamma:=[-1,1]^{\tilde{\Sigma}}$ - is of type $(\tilde{\Sigma}, \alpha, \tilde{\alpha}, \rho)$ if for each $\underline{\gamma}:=(\gamma_{\underline{a}})_{\underline{a} \in \tilde{\Sigma}},$ \index{$\underline{\omega}$}

$$f^{\underline{\gamma}} (x) = (id + \gamma_{\underline{a}} X_{\underline{a}}) \circ f(x), \mbox{ if } x \in f^{-1}(h(\underline{a})),$$

where $X_{\underline{a}} (x) = c_3 \rho \chi(T(x)) e(P(x))$ \index{$X_a$} and $P(x)$ is the element of the partition $\cal{ P }$ owning $x,$ $T$ is an affine transformation from $h(\underline{a})$ to $[-1,1]^3$ and $\chi$ \index{$\chi$} is a $C^{\infty}$ function satisfying

$$\chi(x)=\left\{
      \begin{array}{l}
       1,\,\,\mbox{se}\,\,\| x \| \leq c_2\\
       0,\,\,\mbox{se}\,\,\| x \| \geq c_2^2
       \end{array}
    \right.$$

\end{defn}

\begin{remark}
\label{obs_ck}

Assuming $0 < c < 1,$ if $\alpha$ and $\tilde{\alpha}$ are chosen with scale $\rho^{\frac{c}{k}},$ then this type of perturbation is $C^{k-1}-$small if the scale $\rho > 0$ is chosen sufficiently small since

$$\displaystyle{ \Bigl\| \frac{\partial^j}{\partial e_{i_1}...e_{i_j}} X_a \Bigr\| \leq \frac{\partial^j}{\partial e_{i_1}...e_{i_j}} \chi(T(x)) \rho^{1-\frac{j}{k}} } \mbox{ is small for every } j \leq k.$$

\end{remark}

The following definition will serve to control the dispersion of the displacements of the pieces we wish to perturb and the interferences due to the other pieces we do not wish to perturb. It will be useful in the Marstrand-like argument and in the probabilistic one. In the two cases, it will obstruct, considerably, the influence of pieces in some leaf on the other pieces in the same leaf or in itself again since it requires the backward iterates, by the diffeomorphism, of the piece $(\theta^-,\underline{\theta})$ do not return close to the leaf $\theta^-$ for a sufficient big interval time in such a way that the influeneces due to so propagated perturbations along so much time is small.

\begin{defn}{$(\alpha,\beta)-$non-recurrent word}
\label{palavra_nao_recorrente}

We say a word $\underline{\theta} \in \Sigma_{\theta^-_0}(\beta)$ is $(\alpha,\beta)-$non-recurrent in the leaf $\theta^-$ if any final word, $\underline{\theta}^-$ in $\theta^-,$ in  $\Sigma^{-}(\alpha),$ does not appear in $\underline{\theta}^- \underline{\theta} \in \Sigma^{-*}$ again.

We denote this set of non-recurrent words in the leaf $\theta^-$ by $\Sigma^{*}_{(\alpha,\beta),\theta^-}$ or $\Sigma_{(\alpha,\beta),\theta^-}(\beta).$ 
\end{defn}

By passing through the original diffeomorphism, $f,$ a perturbation family we can, eventually, observe some movement of the strong stable foliation. We do not want a generous movement at any place because we are interested in the displacements of the pieces relative to the leaves of this foliation. We will soon solve this problem.

The following proposition, \ref{prop_erranc_peca}, stablish the effects of the perturbations on the displacement of the pieces in the forward iterates of the perturbated blocks. We observe these few recurrent pieces present predictable displacements - essencially with scale with the size of the perturbation of the perturbation family if these pieces delay too much to arrive, through backward iterates, in pieces which are under perturbation, otherwise, they will present displacement at most by a small fraction with the size of the perturbations (we still assert its boundaries maintain imovable if this time is too long in such a way that the pieces turn themself leaves and along this route they had fallen inside the piece under perturbation). The few recurrent pieces are, therefore, predictable.

We need define the extremities of the pieces with respect to the strong stable foliation in order to enunciate the next proposition. We denote by $\partial^{\underline{\gamma}}(\theta^-,\underline{\theta})$ any point in the right or left extremity in $h^{\underline{\gamma}}(\theta^-,\underline{\theta})$ with respect to $\mathcal{F}^{\underline{\gamma},ss}.$ We define $\Sigma^*_{(\alpha,\beta),\theta^-} := \Sigma^*_{(\alpha,\beta),\theta^-} \cap \Sigma_{\theta^-}(\beta).$

\begin{lem}{Dispersion control of pieces's displacement velocities}
\label{prop_erranc_peca}

There are $\lambda$ and $\hat{\lambda}$ satisfying $0 < \lambda < \hat{\lambda} <1$ and $\tilde{c}_4 > 0$ such that for every $L \in \mathbb{N}$ and $\hat{c}_{5} > 1,$ there are $0 < \kappa < 1,$ $\beta_0 > 0$ and $\alpha_0 > 0$ such that if $0 < \alpha < \alpha_0,$ $0 < \tilde{\alpha} < \alpha_0$ and $0 < \beta < \beta_0,$ then for any perturbation family, $\{ f^{\underline{\gamma}} \}_{\underline{\gamma} \in \Gamma},$ of type $(\tilde{\Sigma},\kappa\alpha,\tilde{\alpha},\rho):$

(a) If $(\theta^-,\underline{\theta}) \in \Sigma^- \times \Sigma_{(\alpha,\beta),\theta^-}^*,$ then for any $\theta \in (\theta^-,\underline{\theta}),$ $\sigma^j(\theta)$ falls in the same element, $\underline{a},$ of the partition $\tilde{\Sigma},$ at most once for every $0 \leq j \leq |\underline{\theta}|.$

(b) If $(\theta^-,\underline{\theta}) \in \Sigma^- \times \Sigma^*_{(\alpha,\beta),\theta^-}$ is such that $\sigma^j(\theta^-,\underline{\theta}) \subset \underline{a} \in \tilde{\Sigma}$ with $0 \leq j \leq L,$ then $$c_5c_4(\underline{\gamma}^0) \lambda^j c_3 \rho < \displaystyle{ \biggl| \frac{\partial}{\partial \gamma_{\underline{a}}} \Pi_{\theta^-}^{\underline{\gamma}^0} \left( \partial^{\underline{\gamma}}(\theta^-,\underline{\theta}) \right) (\underline{\gamma}^0) \biggr| < \hat{c}_5c_4(\underline{\gamma}^0) \hat{\lambda}^j c_3 \rho }, \mbox{ for every } \underline{\gamma}^0 \in \Gamma,$$

$ \mbox{ where } c_5 = \hat{c}_5^{-1} \mbox{ and } \tilde{c}_4 < c_4(\underline{\gamma}^0)$

(c) If $(\theta^-,\underline{\theta}) \in \Sigma^- \times \Sigma^*_{(\alpha,\beta),\theta^-}$ is such that $\sigma^i(\theta^-,\underline{\theta}) \subsetneq \underline{a} \in \tilde{\Sigma}$ for every $0 \leq i \leq j,$ where $L+1 \leq j \leq |\underline{\theta}| - 1,$ then $$\displaystyle{ \biggl| \frac{\partial}{\partial \gamma_{\underline{a}}} \Pi_{\theta^-}^{\underline{\gamma}^0} \left( \partial^{\underline{\gamma}}(\theta^-,\underline{\theta}) \right) (\underline{\gamma}^0) \biggr| < \hat{c}_5c_4(\underline{\gamma}^0) \hat{\lambda}^j c_3 \rho }, \mbox{ for every } \underline{\gamma}^0 \in \Gamma,$$

$ \mbox{ where } c_5 = \hat{c}_5^{-1} \mbox{ and } \tilde{c}_4 < c_4(\underline{\gamma}^0)$

(d) If $(\theta^-,\underline{\theta}) \in \Sigma^- \times \Sigma^*_{(\alpha,\beta),\theta^-}$ is such that $\sigma^i(\theta^-,\underline{\theta}) \subsetneq \underline{a} \in \tilde{\Sigma}$ for every $0 \leq i \leq |\underline{\theta}|,$ then $$ \displaystyle{ \frac{\partial}{\partial \gamma_{\underline{a}}} \Pi_{\theta^-}^{\underline{\gamma}^0} \left( \partial^{\underline{\gamma}}(\theta^-,\underline{\theta}) \right) (\underline{\gamma}^0) = 0 }, \mbox{ for every } \underline{\gamma}^0 \in \Gamma, \mbox{ where } c_5 = \hat{c}_5^{-1}$$

\end{lem}

\begin{proof}

Lets prove part (a). Let $0 < \kappa < 1$ be such that $\sigma^i(\theta^-,\underline{\theta})$ is $(\kappa \alpha, \tilde{\beta}_i)-$non-recurrent for every $1 \leq i \leq L$, where $\tilde{\beta}_i$ is such that $\sigma^i(\underline{\theta}) \in \Sigma_{\theta_i}(\tilde{\beta}_i).$

Suppose $\theta$ is in a block, $\underline{a},$ in the partition $\tilde{\Sigma}.$ To prove $\sigma^j(\theta)$ does not returns in $\underline{a}$ for $1 \leq j \leq |\underline{\theta}|,$ it is enough to observe that the definition of a $(\alpha,\beta)-$non-recurrent word does imply that if $\theta$ is in a element, $\underline{a},$ in the partition $\tilde{\Sigma},$ then, for every $1 \leq j \leq |\underline{\theta}|$, $\sigma^j(\theta)$ is, at least, $\alpha-$away from the leaf $\theta^-,$ which intersects $\underline{a}$ since $\theta \in \underline{a}.$ As the block $\underline{a}$ is contained in a block of leaves with scale $\alpha$ which contain the leaf $\theta^-$ (since $\kappa \alpha < \alpha$), then $\sigma^j(\theta)$ does not returns in this block in the time interval $1 \leq j \leq |\underline{\theta}|.$

Now, suppose $\sigma^i(\theta)$ is in an element, $\underline{a},$ in the partition $\tilde{\Sigma}$ for some $1 \leq i \leq L.$ In order to prove that $\sigma^j(\theta)$ does not return in $\underline{a}$ for any $i < j \leq L,$ it is enough to observe that by the choice of $\kappa,$ $\sigma^k(\sigma^i(\theta))$ does not return in a block with scale $\kappa \alpha$ around the leaf containing $\sigma^i(\theta)$ and slicing $\underline{a}$ for any $1 \leq k \leq |\underline{\theta}| - i.$ This means $\sigma^k(\sigma^i(\theta))$ does not return in $\underline{a}$ for any $1 \leq k \leq |\underline{\theta}| - i.$

Lets prove part (b). First we prove that if $(\theta^-,\underline{\theta}) \in \Sigma^- \times \Sigma^*_{(\alpha,\beta),\theta^-}$ satisfies $(\theta^-,\underline{\theta}) \subset \underline{a} \in \tilde{\Sigma},$ then $c_5 c_4 c_3 \rho < \biggl| \frac{\partial}{\partial \gamma_{\underline{a}}} \Pi_{\theta^-}^{\underline{\gamma}^0} \left( \partial^{\underline{\gamma}}(\theta^-,\underline{\theta}) \right) (\underline{\gamma}^0) \biggr|$ and we observe that the other cases are analogous.

We prove that if $\theta \in (\theta^-,\underline{\theta}),$ then $c^{'}_5 c_4 c_3 \rho < \biggl| \frac{\partial}{\partial \gamma_{\underline{a}}} \left( \Pi^{\underline{\gamma}^0}_{\theta^-} \circ h^{\underline{\gamma}}(\theta) \right) (\underline{\gamma}^0) \biggr|,$ if $\beta_0$ is chosen sufficiently small, where $c_5 < c^{'}_5 < 1.$ This is sufficient because given $\varepsilon >0,$ $\displaystyle{ e^{-\varepsilon} \leq \frac{\left| \frac{\partial}{\partial \gamma_{\underline{a}}}\Pi^{\underline{\gamma}}_{\theta^-}(x) \right|}{\left| \frac{\partial}{\partial \gamma_{\underline{a}}}\Pi^{\underline{\gamma}}_{\theta^-}(y)\right| } \leq e^{\varepsilon}},$ if $x$ and $y$ are in the same piece in $\Sigma_{(\alpha,\beta),\theta^-}^*$ and $\alpha_0$ and $\beta_0$ are chosen sufficiently small.

Lets prove that $c^{'}_5 c_4 c_3 \rho < \biggl| \frac{\partial}{\partial \gamma_{\underline{a}}} \left( \Pi^{\underline{\gamma}^0}_{\theta^-} \circ h^{\underline{\gamma}}(\theta) \right) (\underline{\gamma}^0) \biggr|.$ Let $p^{\underline{\gamma}} := h^{\underline{\gamma}}(\theta)$ and $p^{\underline{\gamma},ws} := \Pi_{\theta^-}^{\underline{\gamma}^0} ( p^{\underline{\gamma}} ).$ We must prove that $\displaystyle{ \left| \frac{\partial}{\partial \gamma_{\underline{a}}} p^{\underline{\gamma}^0,ws} \right| > c^{'}_5 c_4 c_3 \rho }.$

Let $p_{-n}^{\underline{\gamma}} := (f^{\underline{\gamma}})^{-n}(p^{\underline{\gamma}})$ and $p_{-n}^{\underline{\gamma}} := \Pi_{\theta^-\theta|_n}^{\underline{\gamma}^0}(p_{-n}^{\underline{\gamma}}).$  By invariance of the tangent bundle splitting, $T_{\Lambda}M = E_{\Lambda}^{ss} \oplus E_{\Lambda}^{ws} \oplus E_{\Lambda}^u,$

\begin{tabular}{rl}

$ \displaystyle{  \left| \frac{\partial}{\partial \gamma_{\underline{a}}} p^{\underline{\gamma}^0,ws} \right| = }$ & $\displaystyle{ \bigg| \frac{\partial}{\partial \gamma_{\underline{a}}} \left( \Pi^{\underline{\gamma}^0}_{\theta^-_{\theta^-\theta_1}} \circ f^{\underline{\gamma}^0} (p_{-1}^{\underline{\gamma}^0})  \right) }$ 
 \\

$ { } $ & $\displaystyle{ +  \sum_{k=1}^{|\underline{\theta}|-1}    \frac{\partial}{\partial \gamma_{\underline{a}}}  p_{k+1}^{\underline{\gamma}^0,ws}       \prod_{j=1}^k    \frac{\partial}{\partial ws} \left( \Pi^{\underline{\gamma}^0}_{\theta^-_{\theta^-\theta|_j}} \circ f^{\underline{\gamma}^0} (p_{-j}^{\underline{\gamma}^0})  \right)   } $ \\

$ { } $ & $\displaystyle{ +  \frac{\partial}{\partial \gamma_{\underline{a}}}  p_{|\underline{\theta}|}^{\underline{\gamma}^0,ws}       \prod_{j=1}^{|\underline{\theta}|}    \frac{\partial}{\partial ws} \left( \Pi^{\underline{\gamma}^0}_{\theta^-_{\theta^-\underline{\theta}}} \circ f^{\underline{\gamma}^0} (p_{-j}^{\underline{\gamma}^0})  \right) \biggr|  }$ \\

\\

$=$ & $\displaystyle{ \biggl| \frac{\partial}{\partial \gamma_{\underline{a}}} \left( \Pi^{{\underline{\gamma}}^0}_{\theta^-_{\theta^-\theta_1}} \circ f^{\underline{\gamma}^0} (p_{-1}^{\underline{\gamma}^0}) \right) }$ \\

$ { } $ & $\displaystyle{ +  \frac{\partial}{\partial \gamma_{\underline{a}}}  p_{|\underline{\theta}|}^{\underline{\gamma}^0,ws}       \prod_{j=1}^{|\underline{\theta}|}    \frac{\partial}{\partial ws} \left( \Pi^{\underline{\gamma}^0}_{\theta^-_{\theta^-\underline{\theta}}} \circ f^{\underline{\gamma}^0} (p_{-j}^{\underline{\gamma}^0})  \right) \biggr|  },$ by non-recurrence (see item (a)). \\

$\geq$ & $\displaystyle{ \tilde{c}_4 c_3\rho - \left| \frac{\partial}{\partial \gamma_{\underline{a}}}  p_{|\underline{\theta}|}^{\underline{\gamma}^0,ws}       \prod_{j=1}^{|\underline{\theta}|}    \frac{\partial}{\partial ws} \left( \Pi^{\underline{\gamma}^0}_{\theta^-_{\theta^-\underline{\theta}}} \circ f^{\underline{\gamma}^0} (p_{-j}^{\underline{\gamma}^0})  \right) \right |},$\\

$ { } $ &  if $\alpha_0$ is chose suffiently small, for some constant $c_4 > 0$ \\

$\geq$ & $\displaystyle{ c_4 c_3\rho - c_4 \frac{1-c_5}{2} c_3 \rho  },$ if $\beta_0$ is chosen sufficiently small\\

$\geq$ & $\displaystyle{ c_4 c^{'}_5 c_3\rho },$ for some $c^{'}_5$ in $(c_5, 1).$

\end{tabular}

The proof or part (c) is analogous. In order to prove part (d) its enough to observe that $\sigma^j(\theta^-,\underline{\theta}) \subsetneq \underline{a}$ for every $0 \leq j \leq | \underline{\theta} |$ since it is non-recurrent and $(\theta^-,\underline{\theta}) \subsetneq \underline{a}.$ Therefore, the boundary of $\sigma^j(\theta^-,\underline{\theta})$ is not contained in $\underline{a}$ for every $j \geq 0.$

\end{proof}

The following definition will be useful to control the velocity dispersion of the strong stable leaves when perturbing the original diffeomorphism through the perturbation families. We require, in order to control these dispersions, that the stable leaves with which we work are few recurrents, that is, they return close to themselves, by forward iterates, only after a certain time delay (the closeness and this time delay will specify how non-recurrent these leaves will be.) Ahead - in the probabilistic argument - it will be convenient to work with never-recurrent leaves (leaves that does never return close to themselves, being the notion of closeness, in this case, given by how close to themselves they never return). The idea is that the strong stable foliation inside the few recurrent stable leaves move themselves very few, as the ones in the never-recurrents does not move themselves.

\begin{defn}{$(\alpha,\beta)-$non-recurrent leaves}
\label{folha_nao_recorrente}

We say a leaf, $\theta^- \in \Sigma^-$ is $(\alpha,\beta)-$non-recurrent if any finite final word, $\underline{\theta}^-,$ in $\theta^-$ in $\Sigma^-(\alpha)$ does not repeat itself in any finite final subword with scale $\beta$ in $\theta^-.$ We denote this set of leaves by $\Sigma^-_{(\alpha,\beta)}.$

\end{defn}

We denote by $\Sigma^-_{(\alpha,\beta)}(\rho)$ the set of blocks of leaves in $\Sigma^-_{(\alpha,\beta)}$ with scale $\rho.$

The following proposition controls the strong stable foliation dispersion while perturbing the diffeomorphism along the fixed perturbation family. For this sake its necessary to know how are the forward iterates of the stable leaves in which the strong stable leaves we want to control lives since the strong stable foliation depends on the forward iterates of the stable leaves containing them. We want, therefore, assert the stable foliations containing those strong stable leaves are sufficiently non-recurrent in order to get the desired error control of those strong-stable leaves as we perturb the parameters associated to pieces intersecting the stable foliation. We observe these parameters (and some others intersecting few backward iterates of the stable leaves) will make the non-recurrent pieces displace in a predicted way and also the bite of the strong stable foliation contained in that stable leaves does not displace too much. In this way, we get control of the displacement of the pieces relative to the strong-stable foliation.

\begin{lem}{Strong-stable foliation dispersion control}
\label{prop_erranc_folha}

For every $c_6 > 0,$ there is $\beta_0 > 0$ such that for any perturbation family of type $\bigl( \tilde{\Sigma}, \alpha, \tilde{\alpha}, \rho \bigr)$ and any $(\alpha,\beta)-$non-recurrent stable leaf, $\theta^- \in \Sigma^-_{\alpha,\beta}$ with $\beta < \beta_0,$ then

$$ \left \| \left( \Pi^{\underline{\gamma}}_{\theta^-} \right)'(z) \underline{\bar{\gamma}} \right\| < c_6 \rho,$$ for every $z \in W_{\theta^-},$ $\underline{\gamma} \in \Gamma$ e $\underline{\bar{\gamma}} \in \Gamma_{\theta^-,\underline{\gamma}},$ where $\Gamma_{\theta^-,\underline{\gamma}}$ are the parameters $\underline{\bar{\gamma}}$ such that its coordinate values corresponding to pieces not intersecting $W_{\theta^-}$ are fixed in the corresponding coordinate values of $\underline{\gamma}.$

\end{lem}

\begin{proof}

Let $n \in \mathbb{N}$ be such that $(\sigma^{n}(\theta^-),\theta^-|_{n})$ has scale $\beta.$ As the strong-stable foliation varies in a $C^1-$way with the parameters, then $dist_{C^1} ( \mathcal{F}^{\underline{\gamma},ss}(f^n(z)) , \mathcal{F}^{\underline{\bar{\gamma}},ss}(f^n(z)) < C \rho |\underline{\gamma} - \underline{\bar{\gamma}}|$ for some $C > 0.$ Beyond that, given any two manifolds, $S$ and $\bar{S},$ passing through the same point and transversally to $E^{\underline{\gamma},ws},$ then $dist_{C^1}( S , \bar{S} ) < \lambda dist_{C^1}(f^{\underline{\gamma}}(S) , f^{\underline{\gamma}}(\bar{S}))$ for some $0 < \lambda < 1,$ since the weak-stable direction contracts with less force than the strong-stable one.

Therefore, since $f^{\underline{\bar{\gamma}}}|_{(\sigma^{j}(\theta^-),\theta^-|_{j})} = f^{\underline{\gamma}}|_{(\sigma^{j}(\theta^-),\theta^-|_{j})}$ for every $0 \leq j \leq n,$ $\theta^- \in \Sigma^-_{\alpha,\beta}$ if $\underline{\gamma}$ and $\underline{\bar{\gamma}}$ are the parameters whose coordinates have some difference in the values for only those corresponding to the blocks intersecting the leaf $\theta^-,$ then $dist_{C^1} ( \mathcal{F}^{\underline{\gamma},ss}(z) , \mathcal{F}^{\underline{\bar{\gamma}},ss}(z) ) < C \rho |\underline{\gamma} - \underline{\bar{\gamma}}| \lambda^n.$

Then, if $\beta > 0$ is chosen sufficiently small, $n$ will be sufficiently big in such a way that $dist_{C^1} ( \mathcal{F}^{\underline{\gamma},ss}(z) , \mathcal{F}^{\underline{\bar{\gamma}},ss}(z) ) < c_6 \rho |\underline{\gamma} - \underline{\bar{\gamma}}|.$

This implies $$\displaystyle{ \lim_{\varepsilon \rightarrow 0} \sup_{ \substack{ 0 < |\underline{\gamma} - \underline{\bar{\gamma}}  | < \varepsilon  \\ \underline{\gamma},\underline{\bar{\gamma}} \in W_{\theta^-} } } \frac{ \left| \Pi^{\underline{\gamma}}_{\theta^-}(z) - \Pi^{\underline{\bar{\gamma}}}_{\theta^-}(z) \right| }{|\underline{{\gamma}} - \underline{\bar{\gamma}}|} < c_6 \rho}.$$ As the strong stable foliation is $C^1$ along the parameters, then $$ \left \| \left( \Pi^{\underline{\gamma}}_{\theta^-} \right)'(z) \underline{\bar{\gamma}} \right\| < c_6 \rho.$$

\end{proof}

\section{ Marstrand-like argument }
\label{Argumento_do_tipo_Marstrand}

Fixed a constant $A > 0,$ we consider the sets $\Sigma^-_{(A,A^2)}$ and $\Sigma^*_{(A,A^2),\theta^-},$ denoted from now on by $\Sigma^-_{A}$ and $\Sigma^*_{A,\theta^-}.$

Fixed a measure, $\mu,$ in $\Sigma,$ we define the measure, $\mu^-$ in $\Sigma^-$ by $$\mu^-(X) := \mu \left( \theta \in \Sigma \mbox{ such that } h(\theta) \in h(X)\right)$$ for any borelian, $X,$ in $\Sigma^-.$

We define the measure, $\mu_{\theta_0^-}$, in $\Sigma_{\theta^-_0}$ by $$\mu_{\theta_0^-}(X) := \mu \left( \hat{\theta} \in \Sigma \mbox{ such that } h(\hat{\theta}) \in h(X) \right).$$

We also need the measure, $\nu_{\theta^-}^g,$ in $H_{\theta^-}$ defined by $$\nu_{\theta^-}^{g} := \mu_{\theta^-} \circ \left( h^g \circ \Pi_{\theta^-}^{g} \Bigl|_{\Sigma_{A,\theta^-}} \right)^{-1}.$$

\begin{prop}
\label{lema_integral_limitada}
\label{Marstrand}

\

Given a horseshoe $(f,\Lambda)$ with sharp splitting there is a sufficiently sharp partition $\mathcal{P} := \{ P_1 ,..., P_N\},$ a $C^k-$continuous family of diffeomorphisms with $N$ parameters, $\bigl\{ f^{\underline{t}} \bigr\}_{{\underline{t}} \in I^N},$ an invariant probability measure, $\mu,$ in $\Sigma,$ $c_7 > 1,$ $\tilde{K}_1 > 0,$ $\delta>0$ and an open ball with radius $\delta$ around $\underline{0},$ $B_{\delta} \subset I^N,$ such that

(i) $f^{\underline{0}} = f$

(ii) $\displaystyle{ c_7^{-1} \rho^{\bar{d}_s} \leq \mu(\underline{\theta}) \leq c_7 \rho^{\bar{d}_s} },$ for every cylinder, $\underline{\theta} \in \Sigma^+(\rho),$ with scale $\rho.$

(iii) For Lebesgue almost every $\underline{t} \in B_{\delta}$ and every leaf $\theta^- \in \Sigma^-_{A},$ $\nu_{\theta^-}^{\underline{t}} \ll Leb$ and $$\displaystyle{ \int_{B_{\delta}(\underline{0})} \left\| \frac{d\nu_{\theta^-}^{\underline{t}}}{dLeb} \right\|^2_{L_2}  d \underline{t}  < \tilde{K}_1}.$$
\end{prop}

In particular, as a consequence of this proposition, we already obtain a Marstrand-like result - the one we are looking for is stronger, and in order to get it we need the probabilistic argument and the criterion of the recurrent compact set. In order to adapt these two techniques we are interesting in the following result which is a consequence of the following proposition.

\begin{prop}{Marstrand-like property}
\label{cor_prop_tip_Mars}

For every $\xi >0$, there is a parameter $\underline{t} \in I^N,$ a constant $K_1 > 0$ and a subset $\Sigma_{MB}^-$ in $\Sigma_A^-$ with $\mu^-(\Sigma_A^- \backslash \Sigma_{MB}^-) < \frac{\xi}{2}$ such that for every $\theta^- \in \Sigma_{MB}^-,$ $$\displaystyle{ \left\| \frac{d \nu^{\underline{t}}_{\theta^-}}{d Leb} \right\|^2_{L_2} \leq K_1}.$$

\end{prop}

\begin{proof}

By proposition \ref{lema_integral_limitada}, $\displaystyle{ \int_{\Sigma^-} \int_{B_{\delta}(\underline{0})} \left\| \frac{d \nu^{\underline{t}}_{\theta^-}}{d Leb} \right\|^2_{L_2} d \underline{t} d \mu^-(\theta^-)  < \tilde{K}_1}.$ Therefore, there is some constant $K_1 > 0$ such that $\displaystyle{ \left\| \frac{d \nu^{\underline{t}}_{\theta^-}}{d Leb} \right\|^2_{L_2} \leq K_1 }$ for some ${\underline{t}} \in B_\delta(0)$ and for every $\theta^- \in \Sigma_{MB}^-,$ where $\Sigma^-_{MB} \subset \Sigma_A^-$ and satisfying $\mu \bigl( \Sigma^-_{A} \backslash \Sigma^-_{MB} \bigr) < \frac{\xi}{2}.$

\end{proof}

We can find in \cite{SSU} the definition of \textbf{iterated function system (IFS)}: it's a collection of functions, $\Phi:=\{ \varphi_1, ..., \varphi_k \}$ , of a closed interval $I$ or the line in itself. We observe that $\varphi_l(I)$ and $\varphi_j(I)$ can overlap themselves. In that work the authors observe that under certain circunstancies there is an unique invariant set $I_{\Phi}$ of $\Phi$ ($\displaystyle{I_{\Phi} = \bigcup_{j=1}^k \varphi_j(I_{\Phi})}$), compact and non-empty.

Beyond it, they prove that when considering many parameters families of IFS's, \\ $\{ \Phi^{\underline{\gamma}}\}_{\underline{\gamma} \in \Gamma}:=\{ \varphi^{\underline{\gamma}}_1, ..., \varphi^{\underline{\gamma}}_k \}_{\underline{\gamma} \in \Gamma},$ satisfying some restriction relative to the fractal geometry of $\Phi^{\underline{0}}$ (among others), then almost every IFS's in this family display invariant sets with positive Lebesgue measure. An interesting problem enunciated in that work follows transcribed.

\begin{prob}
``It is a open problem the fact that the limit set is in fact a fat Cantor set or does it contains, necessarily, intervals."
\end{prob}

We think one possible solution for this problem follows with analogous arguments we developed in this work. Another interesting problem would be the following - this shall be related to the extension of our results to the case when the central direction has dimension bigger than 1.

\begin{prob}
Describe, tipically, the metric and topological properties of the invariant sets of IFS's with ambiant dimensions higher than $1.$
\end{prob}

\begin{figure}[h]
    \centering
		\includegraphics[width=5cm]{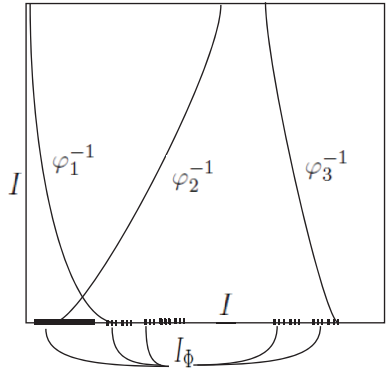}
		\caption{Iterated function system (IFS)}
\end{figure}

An IFS can be seen as the projection along the strong stable foliation of an IFS in higher dimension. This point of view will be useful for the desired result. The projection of $\Lambda$ along the strong stable foliation can be seen as a FS which we will define later.

\begin{figure}[h]
    \centering
		\includegraphics[width=5cm]{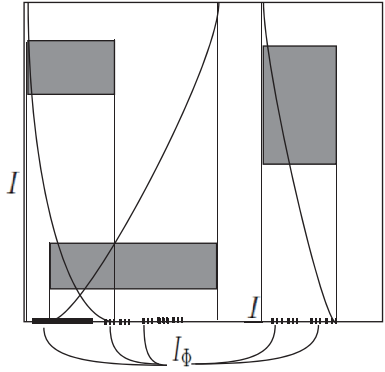}
		\caption{A IFS in dimension 2 could be represented by the IFS in dimension 1 of the previous figure}
\end{figure}

In order to prove the Marstrand-like property we construct a perturbation family with many parameters where changing each parameter means displace the corresponding element in the Markov partition transversally to the strong-stable and unstable foliations. In this way we will need to solve a problem similar to the one solved by \cite{SSU} and, therefore, we restrict ourselves, basically, to adapt the proof of theorem 3.1 in \cite{SSU} in order to obtain the Marstrand-like property for some perturbation of $f$ - as a consequence, by curiosity, we already obtain the fact that the projection of a horseshoe restricted to a stable leaf contains a subset with positive Lebesgue measure.

Since we made some modifications in the arguments in \cite{SSU}, it should be reasonable to extend those results to FS's satisfying some restrictions weaker than just being IFS's.

\subsection{Proof of proposition \ref{lema_integral_limitada}}

\subsubsection{Preparation}

We construct, following the model of definition \ref{modelo_perturb}, a $C^k-$continuous family of diffeomorphisms with $N$ parameters - $\{ f^{\underline{t}} \}_{{\underline{t}} \in I^N},$ where $I$ is an interval in the line - in such a way that changing the $i-$th coordinate of the family of diffeomorphisms means displacing the image of the component $P_i$ in the partition, $\mathcal{P},$ transversally to $E^{ss} \oplus E^u.$

Following the notation in \ref{modelo_perturb} we make a parturbation of type $\Bigl( \mathcal{P} , A , A , A \Bigr),$ where $A > 0$ will be chosen sufficiently small in such a way that this perturbation family is $C^k-$small.

Its worthwhile to observe that $f^{\underline{t}}$ is $C^k,$ that this is a $C^k-$continuous family and that \\ $\| f - f_{\underline{t}} \|_{C^k}$ can be done sufficiently small for every $\underline{t} \in I^N,$ being suffice, for this sake, choosing $A > 0$ sufficiently small.

Let $\Sigma_{A,\theta^-}$ be the elements in $\Sigma^{+}$ beginning with some element in $\Sigma^{*}_{A,\theta^-}.$ For each $\theta^- \in \Sigma_A^-$ and for any $\underline{\theta}$ in $\Sigma^{*}_{A,\theta^-},$ we define a \textbf{function system} family, $\Phi_{\theta^-},$ in such a way that $\Phi_{\theta^-}$ is composed by function families representing, each family, the projection of a piece, $\underline{\theta} \in \Sigma^{*}_{A,\theta^-},$ on $H_{\theta^-}$ (identified with $\mathcal{I} = [-1,1]$) along the strong-stable foliation for $f^{\underline{t}}.$

\begin{defn}{Function system}

Fixed $\theta^- \in \Sigma_A^-,$ the \textbf{function system} family, $\Phi_{\theta^-},$ is $\displaystyle{ \Phi_{\theta^-} := \{\Phi_{(\theta^-,\underline{\theta})} \}_{\underline{\theta} \in \Sigma^*_{A,\theta^-}}},$ where each family of functions, $\Phi_{(\theta^-,\underline{\theta})} := \Bigl\{ \varphi_{(\theta^-, \underline{\theta})}^{\underline{t}} \Bigr \}_{\underline{t} \in I^N}$ is composed by functions defined as follows:

If $\underline{\theta}:=(\theta_1,\theta_2,...,\theta_k) \in \Sigma^*_{A,\theta^-},$ then $\varphi^{\underline{t}}_{(\theta^-, \underline{\theta})} : [-1,1] \rightarrow [-1,1]$ is the increasing affine transformation sending $[-1,1]$ into $\Pi_{\theta^-}^{\underline{t}}(\underline{\theta}).$

\end{defn}

\begin{figure}[h]
    \centering
		\includegraphics[width=8cm]{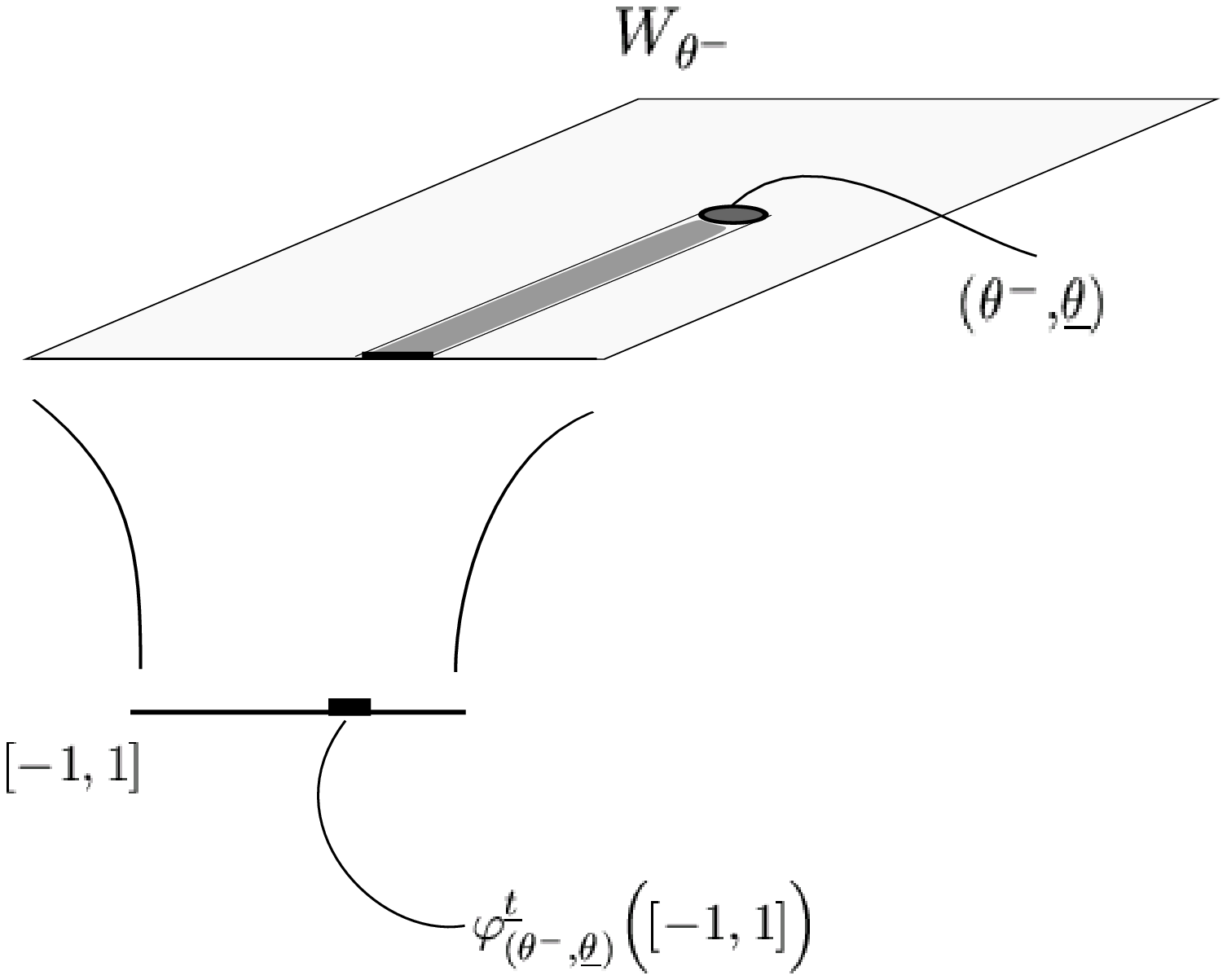}
\end{figure}

This is not an iterated function system as in \cite{SSU}. For this reason we will remake the proof of theorem 3.2 (ii) in \cite{SSU}, with the necessary modifications.

In the next section - \ref{dist_limt} - we will guarantee some conditions - continuity and boundedness of distortion and transversality - for the perturbation family refering to theorem 3.2 (ii) in \cite{SSU}. These conditions will be useful in the proof of propostition \ref{Marstrand}. Precisely, propositions \ref{prop_distortion_continuity}, \ref{prop_transversality_condition} and \ref{Gibbs}.

In the sequel, section \ref{secao_end_proof}, we write the proof of theorem 3.2 (ii) in \cite{SSU}, taking in account the required modifications.

\subsubsection{Hypothesis of theorem 3.2 (ii) in \cite{SSU}}
\label{dist_limt}
\label{section_distortion_continuity}
\label{section_transversality_condition}
\label{secao_pressao}

In \cite{SSU}, the concept of distortion continuity is necessary. The authors used it in order to transfer the fractal information of the original IFS to their neighbourhoods in the perturbation family of IFS. We observe $\varphi_{\theta^-}^{\underline{t}}$ has no distortion at all (in particular, it has bounded distortion) in such a way that the term `distortion continuity' loose its meaning. But the notion the term refers means that, beyond the distortions varying continuously, the contration rates of the IFS's functions also do varies continuously - our family of function systems does satisfies this property. Although the term seams void in our case we still use it unashamedly in order to follow closely the work of the authors in \cite{SSU}.

\begin{defn}{Distortion continuity}

The system function family with $N$ parameters, $\{ \Phi_{\theta^-} \}_{\theta^- \in \Sigma^-_A},$ has uniform distortion continuity if for every $\eta > 0,$ there is $\delta > 0$ such that for every $\theta^- \in \Sigma_A^-$ and ${\underline{t}}^1, {\underline{t}}^2 \in I^N$ with $|{\underline{t}}^1 - {\underline{t}}^2 | < \delta,$ then, for every
$\underline{\theta} \in \Sigma_{A,\theta^-}^{*},$ $$e^{-\ | \underline{\theta} \ | \eta
} \leq \frac{ \left \| \left( \varphi_{ (\theta^-,\underline{\theta})}^{{\underline{t}}^1}
\right)^{'} \right \| }{ \left \| \left( \varphi_{ (\theta^-,\underline{\theta})}^{{\underline{t}}^2}
\right)^{'} \right \| } \leq e^{\ | \underline{\theta} \ | \eta }$$

\end{defn}

The continuity distortion just follows from the fact that $\varphi_{(\theta^-,\underline{\theta})}^{\underline{t}}$ are linear, that $f$ is $C^1$ and that $\mathcal{F}^{ss}$ is $C^{1}$ and varies in Lipschitz way with $\underline{t} \in I^N.$

\begin{prop}
\label{prop_distortion_continuity}

$\{ \Phi \}_{\theta^- \in \Sigma^-_{A}}$ satisfies the uniform distortion continuity property.

\end{prop}

The next condition - transversality - guarantees, typically, in terms of the Lebesgue measure in $I^N$, that the pieces in the construction of the stable Cantor sets (intersection of the horseshoe with the local stable manifolds) do not accumulate excessively along the strong-stable direction for a long period of time when varying the multiparameter $\underline{t} \in I^N.$ This will be useful in order to obtain a lower bound for the Lebesgue measure of the projection of the stable horseshoe along the strong-stable foliation for most parameters in the perturbation family. In order to guarantee the transversality condition we need the pieces defining the stable Cantor sets present relative moviment with speed bounded by below as ${\underline{t}}$ varies. For this reason we restrict ourselves to the leaves in $\Sigma^-_A$ and pieces in $\Sigma_{A,\theta^-}^*$ - they present predictable behaviour.

Lets introduce the following notation in \cite{SSU} in order to make easy the exposition: $$\displaystyle{ \Bigl \{ \pi_{\theta^-}^{\underline{t}}(\theta) \Bigr\} := \bigcap_{i=1}^{\infty} \varphi^{\underline{t}}_{(\theta^-,{\theta}^+|_i)} \Bigl( [-1,1] \Bigr)}, \mbox{ for every } \theta = (\theta^-,\theta^+) \in \Sigma.$$

\begin{defn}{Transversality condition}

$\{ \Phi_{\theta^-} \}_{\theta^- \in \Sigma^-_A}$ satisfies the transversality condition uniformly if there is a constant $C > 0,$ such that for every $\theta^- \in \Sigma_A^-$ and for all ${\theta},$ ${\tau}$
$\in \Sigma_{A,\theta^-}$ with $\theta_1 \ne \tau_1,$ then $$Leb \left\{ {\underline{t}} \in I^N ; \ | \pi_{\theta^-}^{\underline{t}}(\theta) - \pi_{\theta^-}^{\underline{t}}(\tau) \ | \leq r
\right\} \leq Cr, \mbox{ for every } r > 0.$$ 

\end{defn}

\begin{prop}{Transversality condition}
\label{prop_transversality_condition}

If $A > 0$ is chosen sufficiently small, then $\{ \Phi_{\theta^-} \}_{\theta^- \in \Sigma^-_A}$ satisfies the transversality condition uniformly.

\end{prop}

\begin{proof}

Let $\theta^- \in \Sigma_A^-$ and $\theta$ and $\tau$ in $\Sigma_{A,\theta^-}$ with $\theta_1 \ne \tau_1.$ According to lemmas \ref{prop_erranc_peca} and \ref{prop_erranc_folha}, there is a constant ${c}_{8} > 0$ such that for every $\underline{t} \in [-1,1]^N,$ $\displaystyle{ \biggl| \frac{d(\pi^{\underline{t}}_{\theta^-}(\theta) - \pi^{\underline{t}}_{\theta^-}(\tau))}{dt_{\theta_1}} \biggr| \geq {c}_{8} A },$ if $A > 0$ is sufficiently small (enough to choose $A$ such that $A < \alpha_0$ and $A^2 < \beta_0$ given by lemma \ref{prop_erranc_peca} for the constants $1 > c_5 > 0$ chosen sufficiently close to $1$ and satisfyig $c_5 > \hat{c}_5 \lambda$).

Therefore, $  \left\{ {\underline{t}} \in I^N ; \ | \pi^{\underline{t}}_{\theta^-}(\theta) - \pi^{\underline{t}}_{\theta^-}(\tau) \ | \leq r \right\} $ is the product of $I^{N-1}$ by intervals $J_r$'s satisfying $Leb(J_r) < Cr.$ Thus, $Leb \left\{ t \in I ; \ | \pi^t_{\theta^-}(\theta) - \pi^t_{\theta^-}(\tau) \ | \leq r \right\} \leq C r.$
\end{proof}

In order to guarantee the existence of the measure, $\mu,$ enunciated in proposition \ref{lema_integral_limitada} we use a theorem which can be found in \cite{B3}. Let $\Sigma^n_{A,\theta^-} := \{ \underline{\theta} \in \Sigma^{*}_{A,\theta^-} ; |\underline{\theta}|=n \}.$

\begin{thm}{Existence of Gibbs state \cite{B3}}
\label{teo_B3}

Let $(\sigma,\Sigma)$ be topologically mixing and $\psi: \Sigma \rightarrow \mathbb{R}$ a h\"older continuous potential. Then, there is a unique Borel probability measure, $\mu,$ $\sigma-$invariant in $\Sigma$ satisfying the Gibbs condition: there are positive constants $d_1 > 0$ and $d_2 > 0,$ such that $$\displaystyle{ d_1 \leq \frac{\mu \left\{ \theta ; \theta_i = \tau_i \mbox{ for every } i \in [1,m] \right\} }{ \exp{ \left(-P(\psi)m + S_m(\psi,\tau) \right) } } \leq d_2 \mbox{ for every } \tau \in \Sigma^+ \mbox{ and } m \geq 1,}$$

where $\displaystyle{ P(\phi) := \lim_{m \rightarrow \infty} \frac{1}{m} \log{ \Biggl( \sum_{\underline{\theta} \in \Sigma^{m}} \text{exp}\left(\sup_{\theta \in \Sigma ; \theta|_{m} = \underline{\theta}} \{ S_m(\theta) \} \right) \Biggr) } }$ and $S_m(\psi,\theta) := \sum_{k=1}^{m} \psi(\sigma^k(\theta)).$

\end{thm}

We consider, along the Marstrand-like argumentation, the following H\"older continuous potential for the subshift $\sigma: \Sigma \rightarrow \Sigma$ given by $\phi: \Sigma \rightarrow \mathbb{R},$ where $$\displaystyle{ \phi(\theta) := \bar{d}_s \log \left( \left| \lambda^{ws}(h^{\underline{0}} \circ \sigma )(\theta) \right| \right) }.$$

\begin{lem}
\label{quase_igual}

\

There is a constant $c_9 > 1$ such that for every $\theta^- \in \Sigma_A^-$ e $\underline{\theta} \in \Sigma^*_{A,\theta^-},$ then

$\displaystyle{c_9^{-1} \leq \frac{ D_s(V) }{ \| (\varphi_{(\theta^-,\underline{\theta})})' \| } \leq c_9}.$

\end{lem}

\begin{proof}

By bounded distortion of $df$ along the transversals to $E^{ss},$ there is a positive constant $0 < \tilde{K} < 1 $ such that for every $x \in \Lambda \cap V_{\underline{\theta}},$ $\mbox{diam}(W_{loc}^s(x) \cap V_{\underline{\theta}}) > \tilde{K} D_s(V_{\underline{\theta}}).$ In particular, $D_s(V_{\underline{\theta}}) > \mbox{diam}(W_{loc}^s \cap V_{\underline{\theta}}) > \tilde{K} D_s(V_{\underline{\theta}}).$

By the mean value theorem, by the fact that $\varphi_{(\theta^-,\underline{\theta})}$ is linear and that $diam \bigl( \varphi_{(\theta^-,\underline{\theta})}([-1,1]) \bigr)$ is different of $\mbox{diam}(W_{\theta^-} \cap V_{\underline{\theta}})$ by an independent multiplicative constant of $\underline{\theta},$ we conclude that $\displaystyle{ \hat{K}^{-1} \leq \frac{ \mbox{diam}(W_{\theta^-} \cap V_{\underline{\theta}}) }{ \| (\varphi_{(\theta^-,\underline{\theta})})' \| } \leq \hat{K} }$ for some constant $\hat{K} > 1.$

Therefore, $\displaystyle{ \hat{K}^{-1} \leq \frac{ D_s(V) }{ \| (\varphi_{(\theta^-,\underline{\theta})})' \| } \leq \tilde{K}^{-1}\hat{K} }.$ Its enough to choose, $c_9 := \tilde{K}^{-1}\hat{K}.$

\end{proof}

\begin{lem}

$\displaystyle{ P(\phi) = \lim_{n \rightarrow \infty} \frac{1}{n} \log W_{n}(\bar{d}_s)} ,$ where $\displaystyle{ W_{n}(d) := \sum_{\underline{\theta} \in \Sigma^m} D_s(V_{\underline{\theta}})^d }.$

\end{lem}

\begin{proof}

For any set $\Theta$ \index{$\Theta$} containing exactly one symbol in $\Sigma^-$ finishing with each $i \in \{ 1,...,N\},$ then:

\begin{tabular}{rl}

$\displaystyle{ W_{n}(\bar{d}_s) := }$ & $\displaystyle{ \sum_{\underline{\theta} \in \Sigma^{n}} D_s(V)^{\bar{d}_s} }$ \\

$ \asymp$ & $\displaystyle{ \sum_{(\theta^-,\underline{\theta}) \in \Theta \Sigma_{\theta^-_0}^{n}} \| ( \varphi^{\underline{0}}_{{\theta}^-,\underline{\theta}} )' \|^{\bar{d}_s}} ,$ by lemma \ref{quase_igual}, 

\end{tabular}

where $\Sigma_{\theta^-_0}^n := \Sigma^n \cap \Sigma^*_{\theta^-_0}$ and $\Theta \Sigma_{\theta^-_0}^n:= \bigcup_{\theta^- \in \Theta} \bigcup_{\underline{\theta} \in \Sigma_{\theta^-_0}^n}(\theta^-,\underline{\theta}).$

On the other side, given $\displaystyle{ \phi(\theta) := \bar{d}_s \log \left( \left| \lambda^{ws}(h^{\underline{0}} \circ \sigma )(\theta) \right| \right) },$ we have

\begin{tabular}{rl}

$\displaystyle{ e^{S_m(\phi,\theta)} = }$ & $\displaystyle{ \Biggl( \prod_{i=1}^m \left| \lambda^{ws}(h^{\underline{0}} \circ \sigma^i )(\theta) \right| \Biggr)^{\bar{d}_s} }$ \\

$ \asymp $ & $\displaystyle{ \| (\varphi_{(\theta^-,\theta|_n)})' \|^{\bar{d}_s} },$

\end{tabular}

by bounded distortion of the derivatives of $f$ on the transversal directions to the strong stable one.

Hence, for any $\underline{\theta} \in \Sigma^m_{\theta^-_0},$ $\displaystyle{\text{exp} \left( \sup_{\theta ; \theta|_{m} = \underline{\theta}} \{ S_m(\theta) \} \right) \asymp \| (\varphi_{(\theta^-,\theta|_m)})' \|^{\bar{d}_s} }.$

Now,

\begin{tabular}{rl}

$\displaystyle{ { \sum_{\underline{\theta} \in \Sigma^m} \text{exp}\left( \sup_{\theta ; \theta|_{| \underline{\theta}|} = \underline{\theta}} \{ S_n(\theta) \} \right) } = }$ & $\displaystyle{ \sum_{(\theta^-,\underline{\theta}) \in \Theta \Sigma^{m}_{\theta^-}} \| ( \varphi^{\underline{0}}_{{\theta}^-,\underline{\theta}} )' \|^{\bar{d}_s}}$ \\

$\asymp$ & $\displaystyle{ W_{n}(\bar{d}_s) }.$

\end{tabular}

Therefore, $\displaystyle{ P(\phi) = \lim_{n \rightarrow \infty} \frac{1}{n} \log W_{n}(\bar{d}_s)} .$

\end{proof}

\begin{lem}
\label{pressao_zero}

$P(\phi) = 0.$

\end{lem}

\begin{proof}

As $\bar{d}_s$ is the upper stable dimension of $\Lambda$, then $\displaystyle{ \sum_{\underline{\theta} \in \Sigma^n} D_s(V_{\underline{\theta}})^{{\lambda}_n} = 1 }$ for every $n >0,$ and then, $W_{n}(\bar{d}_s) \asymp 1$ for every $n > 0.$ This implies $P(\phi) = 0.$

\end{proof}

\begin{lem}
\label{Gibbs}

\

There is a constant $c_{10} > 1$ and a borel probability measure, $\mu,$ $\sigma-$invariant in $\Sigma,$ such that for every leaf $\theta^- \in \Sigma_A^-$ and $\underline{\theta} \in \Sigma^{*}_{A,\theta^-},$ $$\ds{c_{10}^{-1} \leq \frac{ \mu_{\theta^-_0}(\underline{\theta}) }{ \| (\varphi_{(\theta^-,\underline{\theta})})' \|^{{\bar{d}_s}} } \leq c_{10}}.$$

\end{lem}

\begin{proof}

By theorem \ref{teo_B3} and by lemma \ref{pressao_zero}, there is a probability measure, $\mu,$ satisfying the desired conclusions, since $$\displaystyle{ \exp{ \Bigl( -P(\phi)m + \sum_{k=1}^{m} \phi(\sigma^k(x)) \Bigr)} = \exp{(-P(\phi)m). \| (\varphi_{(\theta^-,\underline{\theta})})' \|^{\bar{d}_s} = \| (\varphi_{(\theta^-,\underline{\theta})})' \|^{\bar{d}_s} }}.$$

\end{proof}

As a consequence of this lemma we can enunciate the following result which will be also useful in the recurrent compact argument in which we need to count the pieces on the leaves.

\begin{lem}
\label{cor_gibbs_rho}

\

If $c_1 > 0$ is chosen sufficiently big, there is a constant $c_{11} > 1$ such that for every leaf $\theta^- \in \Sigma_A^-$ and $\underline{\theta} \in \Sigma^{*}_{A,\theta^-}$ with scale $\rho,$ $$\displaystyle{c_{11}^{-1} \leq \frac{ \mu_{\theta^-_0}(\underline{\theta}) }{ \rho^{\bar{d}_s} } \leq c_{11}}.$$

\end{lem}

\begin{proof}

It is enough to observe that if $c_1 > 0$ is sufficiently big, this means if $\underline{\theta} \in \Sigma_{\theta^-}(\rho),$ then $d_s(\theta^-,\underline{\theta}) \asymp \rho$ for any $\theta^- \in \Sigma^-$ such that $\underline{\theta} \in \Sigma^*_{\theta^-_0}.$ Then, $\mu(\underline{\theta}) \asymp d_s(\theta^-,\underline{\theta})^{\bar{d}_s} \asymp \rho^{\bar{d}_s}.$

\end{proof}

\subsubsection{End of the proof - A version of Simon-Solomyak-Urba\' nski's theorem}
\label{secao_end_proof}

\begin{lem}
\label{lema_SSU}

There are constants $\delta > 0$ and $K_3 >0,$ such that for every $\theta^- \in \Sigma_A^-,$ $$\displaystyle{ X := \int_{B_\delta(0)} \int_{ \Pi^{\underline{t}}_{\theta^-} \bigl( \Sigma_{A,\theta^-} \bigr) } \underline{D}(\nu_{\theta^-}^{\underline{t}},x) d \nu^{\underline{t}}_{\theta^-}(x) d{\underline{t}} < K_3 },$$

where $\displaystyle{\underline{D}(\nu_{\theta^-}^{\underline{t}},x)= \liminf_{r \rightarrow 0} \frac{\nu^{\underline{t}}_{\theta^-}(x-r,x+r)}{2r}}$ is the lower density of $\nu^{\underline{t}}$ in $x$ and $B_\delta(\underline{0})$ is the ball, in $I^N,$ centered in $\underline{0},$ with radius $\delta.$
\end{lem}

Now we can prove proposition \ref{lema_integral_limitada}.

\begin{proof}[Proof of proposition \ref{lema_integral_limitada}]

By lemma \ref{lema_SSU}, $\displaystyle{ \int_{ \Pi^{\underline{t}}_{\theta^-} \bigl( \Sigma_{A,\theta^-} \bigr) } \underline{D}(\nu_{\theta^-}^{\underline{t}}, x) d \nu_{\theta^-}^{\underline{t}} < \infty }$ for almost all \\ ${\underline{t}} \in B_\delta(0).$

Therefore, by theorem 2.12 (1) in \cite{M}, for almost all $\underline{t} \in B_{\delta}(0),$ the density of $\nu^{\underline{t}}_{\theta^-},$

$\displaystyle{ D(\nu^{\underline{t}}_{\theta^-}, x) := \lim_{r \downarrow 0} \frac{\nu^{\underline{t}}_{\theta^-}( B_{r}(x)) }{2r} }$ does exist, for Lebesgue almost all point $x$ in $\Pi^{\underline{t}}_{\theta^-} \bigl( \Sigma_{A,\theta^-} \bigr) .$

Beyond it, by theorem 2.12 (2), $D(\nu^{\underline{t}}_{\theta^-}, .)$ is the Radon-Nykodin derivative $\displaystyle{ \frac{d\nu^{\underline{t}}_{\theta^-}}{d Leb}(.) }$ for Lebesgue almost all point $x$ in ${\Pi^{\underline{t}}_{\theta^-} \bigl( \Sigma_{A,\theta^-} \bigr) }.$

Therefore,

\begin{tabular}{rl}

$\displaystyle{ \int_{\Pi^{\underline{t}}_{\theta^-} \bigl( \Sigma_{A,\theta^-} \bigr) } \biggl| \frac{d\nu^{\underline{t}}_{\theta^-}}{d Leb}(x) \biggr|^2 d Leb(x) =}$ & $\displaystyle{ \int_{\Pi^{\underline{t}}_{\theta^-} \bigl( \Sigma_{A,\theta^-} \bigr) } D(\nu^{\underline{t}}_{\theta^-},x) \frac{d \nu^{\underline{t}}_{\theta^-}}{d Leb}(x) d Leb(x) }$ \\

$=$ & $\displaystyle{ \int_{\Pi^{\underline{t}}_{\theta^-} \bigl( \Sigma_{A,\theta^-} \bigr) } \underline{D}(\nu^{\underline{t}}_{\theta^-},x) d \nu^{\underline{t}}_{\theta^-}(x) }$

\end{tabular}

That is,

$\displaystyle{ \int_{B_{\delta}(\underline{0})} \int_{\Pi^{\underline{t}}_{\theta^-} \bigl( \Sigma_{A,\theta^-} \bigr) } \biggl| \frac{d\nu^{\underline{t}}_{\theta^-}}{d Leb}(x) \biggr|^2 d Leb(x) d\underline{t} \leq K_3 =: \tilde{K}_1.}$

\end{proof}

\begin{proof}[Proof of lemma \ref{lema_SSU}]

Lets follow the lines, with the necessary modifications of the proof of theorem 3.2 (ii) in \cite{SSU}.

By Fatou lemma, $\displaystyle{ X \leq \liminf_{r \rightarrow 0} \int_{B_\delta(0)} \int_{ \Pi^{\underline{t}}_{\theta^-} \bigl( \Sigma_{A,\theta^-} \bigr) } \frac{\nu_{\theta^-}^{\underline{t}}(x-r,x+r)}{2r} d\nu_{\theta^-}^{\underline{t}} d{\underline{t}} }.$

Now,

\begin{tabular}{rl}

$\displaystyle{ \int_{ \Pi^{\underline{t}}_{\theta^-} ( \Sigma_{A,\theta^-} ) } \nu_{\theta^-}^{\underline{t}}(B_r(x)) d\nu_{\theta^-}^{\underline{t}}(x) = }$ & $\displaystyle{ \int_{ \Pi^{\underline{t}}_{\theta^-} ( \Sigma_{A,\theta^-} ) } \int_{ \Sigma_{A,\theta^-} } \chi^{\underline{t}}_r(\theta) d\mu_{\theta^-_0}(\theta) d\nu_{\theta^-}^{\underline{t}}(x) }$ \\

$=$ & $\displaystyle{ \int_{\Sigma_{A,\theta^-} } \int_{\Sigma_{A,\theta^-}} \chi^{\underline{t}}_r(\theta,\tau) d\mu_{\theta^-_0}(\theta) d\mu_{\theta^-_0}(\tau),}$

\end{tabular}

\vspace{10mm}

where $\chi_A(.)$ denotes the characteristic function of $A,$ $\chi^{\underline{t}}_r(\theta) := \chi_{ \{ \theta \in \Sigma_{A,\theta^-}; \ | \pi_{\theta^-}^{\underline{t}}(\theta) - x \ | \leq r \} }(\theta),$ and $\chi^{\underline{t}}_r(\theta,\tau) := \chi_{ \{ (\theta,\tau) \in \Sigma_{A,\theta^-} \times \Sigma_{A,\theta^-}; \ | \pi_{\theta^-}^{\underline{t}}(\theta) - \pi_{\theta^-}^{\underline{t}}(\tau) \ | \leq r \} }(\theta,\tau).$

Therefore, by Fubini lemma,

\begin{tabular}{rl}

$\displaystyle{ X_r := }$ & $\displaystyle{ \int_{B_\delta(t_0)} \int_{\Sigma_{A,\theta^-}} \int_{\Sigma_{A,\theta^-}} \chi^{\underline{t}}_{r}(\theta,\tau) d\mu_{\theta^-_0}(\theta) d\mu_{\theta^-_0}(\tau) d{\underline{t}} }$ \\

$=$ & $\displaystyle{ \int_{\Sigma_{A,\theta^-}} \int_{\Sigma_{A,\theta^-}} \int_{B_\delta(t_0)} \chi^{\underline{t}}_{r}(\theta,\tau) d{\underline{t}} d\mu_{\theta^-_0}(\theta) d\mu_{\theta^-_0}(\tau) }$ \\

$=$ & $\displaystyle{ \int_{\Sigma_{A,\theta^-}} \int_{\Sigma_{A,\theta^-}} Leb \{ {\underline{t}} \in B_\delta(t_0); \ | \pi_{\theta^-}^{\underline{t}}(\theta) - \pi_{\theta^-}^{\underline{t}}(\tau) \ | \leq r \} d\mu_{\theta^-_0}(\theta) d\mu_{\theta^-_0}(\tau) }.$

\end{tabular}

\vspace{10mm}

Now we consider the following partition of $\Sigma_{A,\theta^-} \times \Sigma_{A,\theta^-}$:

$A_\beta = \Bigl\{ (\theta, \tau) \in \Sigma_{A,\theta^-} \times \Sigma_{A,\theta^-}; \theta|_{| \beta |} = \tau|_{| \beta |} = \beta \mbox{ e } \theta_{|\beta| + 1 } \ne \tau_{|\beta| + 1 } \Bigr\}.$

Then,

$\displaystyle{ X_r = \sum_{n \geq 0} \sum_{\beta \in \Sigma^n_{A,\theta^-}}  \int \int_{A_\beta} Leb \{ \underline{t} \in B_\delta(0); \ | \pi_{\theta^-}^{\underline{t}}(\theta) - \pi_{\theta^-}^{\underline{t}}(\tau) \ | \leq r \} d\mu_{\theta^-_0}(\theta) d\mu_{\theta^-_0}(\tau) }$

Now, fixed $n > 0,$

\begin{tabular}{rl}

$\displaystyle{ \Bigl| \pi_{\theta^-}^{\underline{t}}(\theta) - \pi_{\theta^-}^{\underline{t}}(\tau)  \Bigr| = }$ & $\displaystyle{ \lim_{i \rightarrow \infty} \Bigl| \varphi_{(\theta^-,\theta|_i)}^{\underline{t}}(1) - \varphi_{(\theta^-,\tau|_i)}^{\underline{t}}(1) \Bigr|}$ \\

$ = $ & $\displaystyle{ \lim_{i \rightarrow \infty} \Bigl| \varphi_{(\theta^-,\theta|_n)}^{\underline{t}}(\varphi_{(\theta^-\theta|_n,\sigma^n(\theta|_i))}^{\underline{t}}(1)) - \varphi_{(\theta^-,\theta|_n)}^{\underline{t}}(\varphi_{(\theta^-\theta|_n,\sigma^n(\tau|_i))}^{\underline{t}}(1)) \Bigr| }$ \\

$ = $ & $\displaystyle{ \Bigl| \varphi_{(\theta^-,\theta|_n)}^{\underline{t}}(\pi_{\theta^-\theta|_n}^{\underline{t}}(\sigma^n(\theta))) - \varphi_{(\theta^-,\theta|_n)}^{\underline{t}}(\pi_{\theta^-\theta|_n}^{\underline{t}}(\sigma^n(\tau))) \Bigr| } $ \\

$ = $ & $\displaystyle{ \Bigl| \bigl( \varphi_{ (\theta^-,\theta|_n )}^{\underline{t}} \bigr)'(c) \Bigr| \Bigl| \pi_{\theta^-\theta|_n}^{\underline{t}}(\sigma^n(\theta)) - \pi_{\theta^-\theta|_n}^{\underline{t}}(\sigma^n(\tau)) \Bigr|,}$

\end{tabular}

for some $c \in [-1,1],$ by the mean value theorem.

Therefore, by linearity of $\varphi_{(\theta^-,\underline{\theta}|_n)}^{\underline{t}},$

$\displaystyle{ \ | \pi_{\theta^-}^{\underline{t}}(\theta) - \pi_{\theta^-}^{\underline{t}}(\tau) \ | \geq \| (\varphi^{\underline{t}}_{(\theta^-, \underline{\theta}|_{n})})' \| \ | \pi_{\theta^-\theta|_n}^{\underline{t}}(\sigma^n(\theta)) - \pi_{\theta^-\theta|_n}^{\underline{t}}(\sigma^n(\tau)) \ |},$ for every $\theta^- \in \Sigma^-_A.$

Hence,

\begin{tabular}{rl}

$X_r \leq $ & $\displaystyle{ \sum_{n \geq 0} \sum_{\beta \in \Sigma^n_{A,\theta^-}}  \int \int_{A_\beta} \tilde{L}_r(\theta,\tau) d\mu_{\theta^-_0}(\theta) d\mu_{\theta^-_0}(\tau) }$ \\

$ = $ & $\displaystyle{ \sum_{n \geq 0} \sum_{\beta \in \Sigma^n_{A,\theta^-}}  \int \int_{A_\beta} \hat{L}_r(\theta,\tau) d\mu_{\theta^-_0}(\theta) d\mu_{\theta^-_0}(\tau) },$

\end{tabular}

where $\displaystyle{ \tilde{L}_r(\theta,\tau) := Leb \bigl\{ {\underline{t}} \in B_\delta(0); \|(\varphi_{(\theta^-,\theta|_n)}^{\underline{t}})' \| \ | \pi_{\theta^-\theta|_n}^{\underline{t}}(\sigma^n(\theta)) - \pi_{\theta^-\theta|_n}^{\underline{t}}(\sigma^n(\tau)) \ | \leq r \bigr\} }$

and $\displaystyle{ \hat{L}_r(\theta,\tau) := Leb \biggl\{ {\underline{t}} \in B_\delta(0); \ | \pi_{\theta^-\theta|_n}^{\underline{t}}(\sigma^n(\theta)) - \pi_{\theta^-\theta|_n}^{\underline{t}}(\sigma^n(\tau)) \ | \leq \frac{r}{\|(\varphi_{(\theta^-,\theta|_n)}^{\underline{t}})' \|} \biggr\} }$

Now, lets use the distortion continuity to transfer the fractal property of $\varphi^{\underline{0}}$ to $\varphi^{\underline{t}}$ for $\underline{t}$ close to $\underline{0}.$

Let $ \epsilon_0 > 0 $ be such that $\epsilon_0 < \bar{d}_s - 1.$ By distortion continuity, for any $\eta > 0$ there is a constant $\delta >0$ such that if ${\underline{t}}$ satisfies $\ | {\underline{t}} \ | \leq \delta,$

\begin{tabular}{rl}

$\displaystyle{\|(\varphi_{(\theta^-,\theta|_n)}^{})' \|^{(1 + \frac{\epsilon_0}{4})} \leq }$ & $\displaystyle{ (e^{\ | \theta|_n \ | \eta } \| (\varphi_{(\theta^-,\theta|_n)}^{{\underline{t}}})' \|)^{(1 + \frac{\epsilon_0}{4})} }$ \\

$\leq$ & $\displaystyle{ (e^{\ | \theta|_n \ | \eta })^{(1 + \frac{\epsilon_0}{4})} \lambda^{\frac{\epsilon_0}{4} |\theta|_{n} | } \| (\varphi_{(\theta^-,\theta|_n)}^{{\underline{t}}})' \| },$ for some $0 < \lambda < 1.$

\end{tabular}

Then, by choosing $\eta >0$ such that $(1 + \frac{\epsilon_0}{4})\eta + \frac{\epsilon_0}{4} \log(\lambda)=0,$

\begin{tabular}{rl}

$\displaystyle{ \|(\varphi_{(\theta^-,\theta|_n)}^{})' \|^{(1 + \frac{\epsilon_0}{4})} \leq }$ & $\displaystyle{ (e^{\ | \theta|_n \ | })^{( (1 + \frac{\epsilon_0}{4})\eta + \frac{\epsilon_0}{4} \log(\lambda))} \| (\varphi_{(\theta^-,\theta|_n)}^{{\underline{t}}})' \| }$ \\

$ = $ & $\displaystyle{ \| (\varphi_{(\theta^-,\theta|_n)}^{{\underline{t}}})' \|. }$

\end{tabular}

Denoting $\displaystyle{ L_r(\theta,\tau) := Leb \biggl\{ {\underline{t}} \in B_\delta(0); \ | \pi_{\theta^-\theta|_n}^{\underline{t}}(\sigma^n(\theta)) - \pi_{\theta^-\theta|_n}^{\underline{t}}(\sigma^n(\tau)) \ |  \leq \frac{r}{\|(\varphi_{(\theta^-,\theta|_n)}^{})' \|^{(1 + \frac{\epsilon_0}{4})}} \biggr\} }$

and choosing such a $\eta > 0,$ there is $\delta > 0,$ such that

\begin{tabular}{rl}

$X_r \leq$ & $\displaystyle{ \sum_{n \geq 0} \sum_{\beta \in \Sigma^n_{A,\theta^-}}  \int \int_{A_\beta} L_r(\theta,\tau) d\mu_{\theta^-_0}(\theta) d\mu_{\theta^-_0}(\tau) }$ \\

$\leq$ & $\displaystyle{ \sum_{n \geq 0} \sum_{\beta \in \Sigma^n_{A,\theta^-}}  \int \int_{A_\beta} \frac{Cr}{\|(\phi_{(\theta^-,\theta|_n)}^{})' \|^{(1 + \frac{\epsilon_0}{4})}} d\mu_{\theta^-_0}(\theta) d\mu_{\theta^-_0}(\tau)},$

\end{tabular}

by transversality condition.

Now, by lemma \ref{Gibbs}, $\| (\varphi_{(\theta^-,\theta|_n)})' \|^{\bar{d}_s} \geq c_{10}^{-1} \mu(\theta|_n).$

Therefore, $\displaystyle{ \frac{1}{\| (\varphi_{(\theta^-,\theta|_n)})' \|^{1 + \frac{\epsilon_0}{4}}} \leq {(c_{10}^{-1} \mu(\theta|_n))^{-\frac{1 + \frac{\epsilon_0}{4}}{\bar{d}_s}}} }.$

Now, as ${\bar{d}_s}>1$ e $\epsilon_0 < {\bar{d}_s}-1,$ then

$\displaystyle{ \frac{1 + \frac{\epsilon_0}{4}}{{\bar{d}_s}} < \frac{1 + \frac{{\bar{d}_s}-1}{4}}{{\bar{d}_s}} = \frac{1}{{\bar{d}_s}} + \frac{1}{4} - \frac{1}{4{\bar{d}_s}} = \frac{1}{4} + ( \frac{1}{{\bar{d}_s}} - \frac{1}{4{\bar{d}_s}} ) = \frac{1}{4} + \frac{3}{4{\bar{d}_s}} = \frac{1}{4} + \frac{3}{4} = 1 }.$

Hence, there is a number $x >0$ such that $\displaystyle{ \frac{1 + \frac{\epsilon_0}{4}}{{\bar{d}_s}} < 1 - x }.$

Then,

\begin{tabular}{rl}

$\displaystyle{\frac{1}{\| (\varphi_{(\theta^-,\theta|_n)})' \|^{1 + \frac{\epsilon_0}{4}}} \leq }$ & $\displaystyle{ \frac{1}{(c_{10}^{-1} \mu(\theta|_n))^{1-x}} }$ \\

$ = $ & $\displaystyle{ \frac{1}{(c_{10}^{-1} \mu(\theta|_n))}(c_{10}^{-1} \mu(\theta|_n))^x }$ \\

$\leq$ & $\displaystyle{ \frac{1}{(c_{10}^{-1} \mu(\theta|_n))} (c_{10}\lambda^n)^x}$ \\

$=$ & $\displaystyle{ c_{10}^{x-1} \lambda^{nx} \frac{1}{\mu(\theta|_n)} }.$

\end{tabular}

Therefore,

\begin{tabular}{rl}

$\displaystyle{ X = }$ & $\displaystyle{ \liminf_{r \rightarrow 0} \frac{X_r}{2r} }$ \\

$\leq$ & $\displaystyle{ \sum_{n \geq 0} \sum_{\beta \in \Sigma^n_{A,\theta^-}} \frac{ \int \int_{A_\beta} Cr c_{10}^{x-1} \lambda^{nx} \frac{1}{\mu(\theta|_n)} d\mu_{\theta^-_0}(\theta) d\mu_{\theta^-_0}(\tau) }{2r} }$ \\

$=$ & $\displaystyle{ \frac{C}{2} c_{10}^{x-1} \sum_{n \geq 0} \lambda^{nx} \sum_{\beta \in \Sigma^n_{A,\theta^-}}  \frac{1}{\mu_{\theta^-_0}(\beta)} \mu_{\theta^-_0} \times \mu_{\theta^-_0} (\beta \times \beta) }$ \\

$ = $ & $\displaystyle{ \frac{C}{2} c_{10}^{x-1} \sum_{n \geq 0} \lambda^{nx} \sum_{\beta \in \Sigma^n_{A,\theta^-}}  \mu_{\theta^-_0}(\beta) }$ \\

$\leq$ & $\displaystyle{ \frac{C}{2} c_{10}^{x-1} \sum_{n \geq 0} \lambda^{nx} =: K_3 < \infty }.$

\end{tabular}

\end{proof}

\newpage

\section{Probabilistic argument}
\label{Argumento_probabilistico}

\subsection{The family of perturbations}
\label{sec_perturba}

From now on we fix a $N-$parameter, $\underline{t},$ given by proposition \ref{cor_prop_tip_Mars} and work with $f^{\underline{t}}.$ We consider, with no loss of generality $f^{\underline{t}} = f.$ In this section we create a new perturbation family passing through $f.$

Let $c > 0$ be sufficiently small in such a way that it guarantees that the pieces with scale $\rho^{\frac{c}{k}}$ are far a way one to each other with approximate distance $\rho^{\frac{1}{k}}.$ This relaxation is due to the lack of conformality of $df|_{E^s},$ measured by $c.$

The perturbation families used in this second perturbation, refering to the probabilistic argument, will be done under a family $\Omega := [-1,1]^{\Sigma_1}$ where its coordinates are indexed by blocks (advanced iterates of the fixed Markov partition) forming a partition in $\Sigma.$ These blocks have diameter with scale $\rho^{\frac{1}{k}}$ (the distance between its leaves have approximate size $\rho^{\frac{1}{k}}$ and with pieces with scale $\rho^{\frac{c}{k}}$). More specifically, $\Sigma_1 \subset \Bigl\{ \underline{\theta} = (\underline{\theta}^-, \underline{\theta}^+) \in \Sigma^{*}; \underline{\theta}^- \in \Sigma^{-}(\rho^{\frac{1}{k}}), \underline{\theta}^+ \in \Sigma^{+*}_{\theta^-}(\rho^{\frac{c}{k}})  \Bigr\}.$

This perturbation family, $\bigl\{ f^{\underline{\omega}} \bigr\}_{\underline{\omega} \in \Omega},$ will be of type $( \Sigma_1, \kappa \rho^{\frac{1}{k}}, \rho^{\frac{c}{k}}, \rho ),$ for some $0 < \kappa < 1$ depending only on $f.$ This means we will perform perturbations with scale $\rho$ in blocks with diameter $\kappa \rho^{\frac{1}{k}}$ - and therefore distant with approximate size $\kappa \rho^{\frac{1}{k}}.$ Such perturbations are transversals to $E^{ss} \oplus E^u$ (the idea is that they are moving in the direction of $E^{ws}).$

We have already observed in \ref{obs_ck} that these perturbations are $C^{k-1}$ small if the scales $\rho > 0$ are chosen sufficiently small.

\subsection{The probabilistic argument}
\label{arg_prob}

We assume the candidate for recurrent compact set, $K,$ has already been constructed and that it satisfies property \ref{prob} which will be enunciated in this section. Here we prove this set is, in fact, recurrent compact for some, in fact many, $f^{\underline{\omega}}.$

We define, for each $(x,\theta^-) \in H,$ the parameters for which there is some piece with scale $\rho$ such that the renormalization of $(x, \theta^-)$ refering to such parameter, falls in the candidate for recurrent compact set, $K,$ $$\Omega_0(x,\theta^-) = \left\{ {\omegab} \in \Omega \text{ such that there is }  \underline{a} \in \Sigma_{\theta^-_0}(\rho) \text{ satisfying } R^{\omegab}_{\underline{a}}(x,\theta^-) \in int(K) \right\}.$$

We say $(x,\theta^-)$ is recurrent for $f^{\underline{\omega}}$ if $\underline{\omega} \in \Omega_0(x,\theta^-).$

\begin{defn}

\

\begin{itemize}

\item For each $\delta > 0,$ $K_{-\delta}$ is said \textbf{$\delta-$relaxed interior of $K$} if the neighbourhoods with radius $\delta$ of $K_{-\delta}$ are contained in $K.$

\item We define, for each $(x,\theta^-) \in H,$ $$\Omega_{0,\rho^2}(x,\theta^-) = \left\{ {\omegab} \in \Omega \text{ such that there is }  \underline{a} \in \Sigma_{\theta^-_0}(\rho) \text{ satisfying } R^{\omegab}_{\underline{a}}(x,\theta^-) \in K_{-\rho^2} \right\}.$$

\end{itemize}

\end{defn}

The set $K,$ candidate for recurrent compact, which will be constructed in another section, satisfies the following property. We will be concerned in the construction of the set $K$ which satisfies these properties from the next section on.

\begin{prop}{Main property}
\label{prob}

There are positive constants $c_{13} > 0$ and $c_{14} > 0$ such that for every $\rho>0$ sufficiently small, there is $K \subset H$ and a subset of leaves, $\mathcal{W}^-,$ $c_{14}\rho-$dense in the leaves whose projections contain $K$ such that if $(x,\theta^-) \in K \cap \mathcal{W}^-,$ then $$\mathbb{P} \Bigl( \Omega \setminus \Omega_{0,\rho^2}(x,\theta^-) \Bigr) < exp(- c_{13} \rho^{-\frac{c}{k}(\bar{d}_s-1)} ).$$

\end{prop}

This means there are too many parameters $\omegab \in \Omega$ such that $\omegab \in \Omega_{0,\rho^2}(x,\theta^-).$ The following proposition manifests that it is enough to prove that there is a set, $K,$ satisfying the main property - \ref{prob} - in order to guarantee the existence of a recurrent compact set $C^k-$close to a horseshoe, $(f,\Lambda),$ with sharp splitting.

\begin{prop}
\label{teor_exit_omega_K_compacto_recorrent}

\

If $\rho > 0$ is chose sufficiently small then if $K$ satisfies the main property, \ref{prob}, there is $\omegab \in \Omega$ such that $K$ is recurrent compact set for $f^{\omegab}.$

\end{prop}

\begin{proof}

We denote the set of points in $H$ for which there is a piece with scale $\rho$ such that the renormalization refering to the parameter $\underline{\omega}$ of such points falls in $K_{-\rho^2},$ by $$\Omega^{-1}_{0,\rho^2}(\underline{\omega}) := \left\{ (x,\theta^-) \in H; \exists \underline{a} \in \Sigma_{\theta^-_0}(\rho), R_{\underline{a}}^{\underline{\omega}}(x,\theta^-) \in K_{-\rho^2} \right\}.$$

Let $B_x(r)$ be the ball with center $x$ and radius $r$ inside $\mathbb{R}$ and $\mathcal{B}_{\theta^-}(\rho)$ a cylinder of leaves, $\mathcal{B}_{\theta^-}(\rho) \in \Sigma^-(\rho),$ with scale $\rho$ containing $\theta^-.$

We perform the following decomposition in $\Omega_{0,\rho^2}^{-1}(\omegab)$: There is a positive constant $c_{12} > 0$ such that for every $\omegab \in \Omega$ and $\rho > 0,$ $\Omega_{0,\rho^2}^{-1}(\omegab)$ is empty or it is $\displaystyle{ \bigcup_{\alpha \in \mathcal{A}} U_{\alpha} \cap \Omega^{-1}_{0,\rho^2}(\underline{\omega}) },$ where $\mathcal{A}$ is a set with, at most, $c_{12}\rho^{-5}$ indexes and $U_{\alpha} =  B_{x_{\alpha}}(\rho^{4} ) \times \mathcal{B}_{\theta^-_{\alpha}}(c_{14}\rho).$

Fixed $(x,\theta^-) \in K,$ by property \ref{prob}, $\mathbb{P}_{\Omega} \Bigl( \Omega \setminus \Omega_{0, \rho^2}(x,\theta^-) \Bigr) < exp(- c_{13} \rho^{-\frac{c}{k}(\bar{d}_s-1)}).$ That is, $\mathbb{P}_{\Omega} \Bigl( \Omega_{0, \rho^2}(x,\theta^-) \Bigr) \geq 1 - exp(- c_{13} \rho^{-\frac{c}{k}(\bar{d}_s-1)}).$

Beyond that, by property \ref{prob}, if $\omegab \in \Omega_{0, \rho^2}(x,\theta^-)$ and $(x,\theta^-) \in U_{\alpha},$ then if $c_{14}>0$ is chosen sufficiently small, $\omegab \in \Omega_{0}(\tilde{x},\tilde{\theta}^-)$ for every $(\tilde{x},\tilde{\theta}^-) \in U_{\alpha},$ since there is $\underline{a} \in \Sigma_{\theta_0}^-(\rho)$ such that the renormalization operator corresponding to $\underline{a}$ sends $(x,\theta^-)$ into $K_{-\rho^2}.$ This is due to the fact that the renormalization operator associated to a vertical cylinder with scale $\rho > 0$ expands with scale $\rho^{-1}$ in the direction $\mathcal{I}$ and contracts with scale $\rho$ in the direction $\mathcal{K}$ ($H = \mathcal{I} \times \mathcal{K}$). This means the renormalization operator corresponding to $\underline{a}$ sends $(\tilde{x},\tilde{\theta}^-)$ in $int(K)$ if $\rho$ and $c_{14}$ are chosen sufficiently small.

Therefore, $\displaystyle{ \mathbb{P}_{\Omega} \biggl( \bigcap_{(\tilde{x},\tilde{\theta}^-) \in U_{\alpha}}  \Omega_{0}(\tilde{x},\tilde{\theta}^-) \biggr) \geq 1 - exp(- c_{13} \rho^{-\frac{c}{k}(\bar{d}_s-1)}) }$ if $(x,\theta^-) \in K \cap \mathcal{W}^- \cap \mathcal{U}_{\alpha}.$

Hence, $\displaystyle{ \mathbb{P}_{\Omega} \biggl( \bigcup_{(\tilde{x},\tilde{\theta}^-) \in U_{\alpha}} (\Omega \setminus \Omega_{0}(\tilde{x},\tilde{\theta}^-) ) \biggr) < exp(- c_{13} \rho^{-\frac{c}{k}(\bar{d}_s-1)}) }.$

And then it follows that, since $\mathcal{W}^-$ is $c_{14}\rho-$dense in the leaves whose projection contains $K,$ then

$\displaystyle{\mathbb{P}_{\Omega} \left( \bigcup_{{\begin{array}{c}
\alpha \in \mathcal{A} \\
(\tilde{x},\tilde{\theta}^-) \in U_{\alpha}
\end{array}}}   \Omega \setminus \Omega_{0}(\tilde{x},\tilde{\theta}^-)  \right) < c_{14}^{-1}c_{12} \rho^{-5} exp(- c_{13} \rho^{-\frac{c}{k}(\bar{d}_s-1)})}, $ that is,

$\displaystyle{\mathbb{P}_{\Omega} \biggl( \bigcup_{(\tilde{x},\tilde{\theta}^-) \in K } \Omega \setminus \Omega_{0}(\tilde{x},\tilde{\theta}^-)  \biggr) < c_{14}^{-1}c_{12} \rho^{-5} exp(- c_{13} \rho^{-\frac{c}{k}(\bar{d}_s-1)}) \xrightarrow{\rho \rightarrow 0} 0 }.$

In this way, there is (in abundance) $\omegab \in \Omega$ such that $K$ is recurrent compact for $f^{\omegab}.$

\end{proof}

\subsection{Construction of the candidate for recurrent compact set $K$ }

In order to finish it, we need to construct a candidate for recurrent compact set, $K,$ and guarantee it does satisfy the main property - \ref{prob}. We will make it in section \ref{demonstr_teore_prob}.

Let $\theta^-$ be a leaf. We construct a subset $K_{\theta^-}$ in $H_{\theta^-}$ having a stacking property for $f$ - this will be precised later: we mean essentially that above each point in $K$ (i.e., along its strong-stable leaf) there are a lot of pieces with scale $\rho.$

We will be concerned, from now on, in preparing the basis for the proof of lemma \ref{lema_acavalamentos_bem_distribuidos}, which asserts, essentially, for each $(x,\theta^-),$ there are approximately $\rho^{-(\bar{d}_s - 1)\frac{c}{k}}$ pieces with scale $\rho$ whose projections along the strong-stable foliation of $f$ contain $(x,\theta^-)$, in such a way that each of these pieces are contained in different pieces with scale $\rho^{\frac{c}{k}}$ - which means they are well separated, with approximate distance $\rho^{\frac{1}{k}}.$ Beyond that, the Lebesgue measure of the projection of these pieces along the strong stable foliation on $H_{\theta^-}$ is bounded below by some positive constant independent of $\theta^-.$ Concerning the separation between these pieces - with scale $\rho^{\frac{1}{k}}$ - we observe this is necessary in order that the parameters of the perturbation family corresponding to a piece does not exhert influence in the others pieces - we need independence of the displacements of pieces in the same strong-stable leaf (the fact is that we guarantee independence of pieces which are in the same stable leaf). Remember our comantary on non-recurrent pieces and leaves - they are predictable.

\subsubsection{ Selection of good leaves and pieces }

Now we define the leaves we will be dealing with. They are in the neighbourhood of some very good and never-recurrent leaf (this one does never return close to itself). In the sequel, we define the pieces we will be dealing with.

\begin{defn}{Very good leaves}

Lets fix a constant $0 < \xi < 1$ and name \textbf{very good leaves} the ones in the set $\Sigma^-_{MB} \in \Sigma^-,$ given by proposition \ref{cor_prop_tip_Mars} applied for $\xi$. We remember this means if $\theta^- \in \Sigma_{MB},$ then $$\displaystyle{ \Bigl\| \frac{d\nu_{\theta^-}}{dLeb} \Bigr\|^2_{L_2} < K_1 }$$

\end{defn}

\begin{defn}{Never recurrent leaf}

We will consider the leaves in $\Sigma^-$ never-recurrent for words with scale $\rho^{\frac{c}{k}},$ that is, the leaves in $\displaystyle{ \bigcap_{\beta > 0} \Sigma^-_{(\rho^{\frac{c}{k}}, \beta)} .}$ We denote these leaves by the symbol $\mathcal{W}^-$ and we call them never-recurrent leaves.

\end{defn}

The leaves we will define in short will be those in which the candidate for recurrent compact set will be constructed.

\begin{defn}{Non-recurrent good leaf}

Fixed $c_{14} > 0,$ we say $\theta^- \in \Sigma^-$ is a \textbf{non-recurrent good leaf} if $\theta^-$ is in some block, $\underline{\theta}^-,$ in $\Sigma_{(\rho^{\frac{c}{k}},c_{14}\rho)}^{-}$ contained in some block in $\Sigma_A^-$ such that $\underline{\theta}^- \cap \Sigma^-_{MB} \ne \emptyset.$ We denote the set of blocks with scale $c_{14}\rho$ of non-recurrent good leaves by $\mathcal{W}^-(c_{14}\rho).$

\end{defn}

We observe each block of non-recurrent good leaves, $\underline{\theta}^- \in \mathcal{W}^-(c_{14}\rho),$ contains never-recuurrent good leaves in $\mathcal{W}^-.$

\begin{lem}
\label{todas_folhas_boas}

\

For each $c_{14} > 0,$ there is $\rho_0 > 0$ such that if $0 < \rho < \rho_0,$ then $\mu( \Sigma^- \setminus \mathcal{W}^-(\frac{1}{2}c_{14}\rho) ) < \xi.$

\end{lem}

In order to prove this lemma, we need the following lemma.

\begin{lem}
\label{lem_mu^-}

\

For each $\varepsilon > 0,$ $\mu^- \left( \Sigma^- \backslash \Sigma^-_{(\rho^{\frac{c}{k}}, \frac{1}{2}c_{14}\rho )} \right) < \varepsilon,$ if $\rho > 0$ is chosen sufficiently small.

\end{lem}

\begin{proof}

Let $\tilde{C} > 0$ and $m := -\left\lceil \tilde{C} \log\left(\frac{1}{2}c_{14}\rho\right) \right\rceil.$ We observe that if $\tilde{C}$ is chosen sufficiently big, then any leaf with scale $\frac{1}{2}c_{14}\rho,$ $\underline{\theta}^- \in \Sigma^-\left(\frac{1}{2}c_{14}\rho\right),$ satisfies $|\underline{\theta}^-| \leq m.$

For each $(i,j,k,l) \in \mathbb{N}^4,$ we define $$T_{(i,j,k,l)} := \left\{ \underline{a}\underline{b}\underline{c}\underline{b}\underline{e} \in \Sigma^-\left(\frac{1}{2}c_{14}\rho\right) \mbox{ such that } |\underline{a}| = i, |\underline{b}| = j, |\underline{c}| = k, |\underline{e}| = l \mbox{ e } \underline{b} \in \Sigma^-(\rho^{\frac{c}{k}}) \right\}.$$

We observe $\displaystyle{ \bigcup_{ \substack{ (i,j,k,l) \in \mathbb{N}^4 \\ i+2j+k+l \leq m } } T_{(i,j,k,l)} }$ is the set of all words with scale $\frac{1}{2}c_{14}\rho$ containing a subword with scale $\rho^{\frac{c}{k}}$ repeating in two disjoint intervals of indexes. We prove these sets have small $\mu^-$ measure. Then, we also prove the set of words containing subword with scale $\rho^{\frac{c}{k}}$ reapeting in index intervals intersecting have small measure. After these two steps we have our lemma.

We note there is a constant $\hat{C} > 1$ such that $\displaystyle{ \hat{C}^{-1} \leq \frac{ \mu^-(\underline{a}\underline{b}\underline{c}\underline{b}\underline{e}) }{ \mu^-(\underline{a})\mu^-(\underline{b})\mu^-(\underline{c})\mu^-(\underline{b})\mu^-(\underline{e}) } \leq \hat{C} },$ because $\mu^-(\underline{\theta}) \asymp \rho^{\bar{d}_s}$ if $\underline{\theta} \in \Sigma^-(\rho)$ and $d_s(\underline{a}\underline{b}) \asymp d_s(\underline{a})d_s(\underline{b}).$

Therefore,

\begin{tabular}{rl}

$\displaystyle{ \mu^-(T_{(i,j,k,l)}) }$ & $\displaystyle{ \leq \sum_{\underline{a} \in \Sigma^{i-}} \sum_{\underline{b} \in \Sigma^{j-}} \sum_{\underline{c} \in \Sigma^{k-}} \sum_{\underline{e} \in \Sigma^{l-}} \hat{C} \mu^-(\underline{a})\mu^-(\underline{b})\mu^-(\underline{c})\mu^-(\underline{b})\mu^-(\underline{e}) }$ \\

$$ & $\displaystyle{ = \hat{C} \mu^-(\underline{b}) \sum_{\underline{a} \in \Sigma^{i-}} \mu^-(\underline{a}) \sum_{\underline{b} \in \Sigma^{j-}} \mu^-(\underline{b}) \sum_{\underline{c} \in \Sigma^{k-}} \mu^-(\underline{c}) \sum_{\underline{e} \in \Sigma^{l-}} \mu^-(\underline{e})}$ \\

$$ & $\displaystyle{ = \hat{C} \mu^-(\underline{b}) }$ \\

$$ & $\displaystyle{ \leq \bar{C} \rho^{\frac{c}{k} \bar{d}_s } },$ for some constant $\bar{C} > 0.$

\end{tabular}

Therefore,

\begin{tabular}{rl}

$\displaystyle{ \mu^- \left( \bigcup_{ \substack{ (i,j,k,l) \in \mathbb{N}^4 \\ i+2j+k+l \leq m } } T_{(i,j,k,l)} \right) }$ & $\displaystyle{ \leq \sum_{ \substack{ (i,j,k,l) \in \mathbb{N}^4 \\ i+2j+k+l \leq m } } \mu^-(T_{(i,j,k,l)}) }$ \\

$$ & $\leq (m+1)^4 \bar{C} \rho^{\frac{c}{k} \bar{d}_s }$ \\

$$ & $= (\tilde{C} \left\lceil \log\left(\frac{1}{2}c_{14}\rho\right)\right\rceil)^4 \bar{C} \rho^{\frac{c}{k} \bar{d}_s }$ \\

$$ & $= \tilde{C}^4 \bar{C} \left\lceil \log\left(\frac{1}{2}c_{14}\rho\right)\right\rceil^4 \rho^{\frac{c}{k} \bar{d}_s } \rightarrow 0,$

\end{tabular}

as $\rho$ convergers to zero.
 
\vspace{10mm}

Now we prove the measure of words containing subword with scale $\rho^{\frac{c}{k}}$ reapeting in intersecting index intervals is small. This ends the proof of the lemma.

First we observe these are the words of type $\underline{a}\underline{c}\underline{b}\underline{e}$ such that there is some $\tilde{c}$ with $\underline{c}\underline{b} = \underline{b}\underline{\tilde{c}}$ with $|c| < |b|.$ This implies there is $l \geq 2$ and $\underline{\hat{c}}$ such that $\underline{a}\underline{c}\underline{b}\underline{e} = \underline{a}\underline{c}^l\underline{\hat{c}}\underline{e},$ where $\displaystyle{ \underline{c}^l \in \bigcup_{ 1 \leq \beta \leq 2} \Sigma^-(\rho^{\beta \frac{c}{k}}).}$

For each $(i,j,k,r,l) \in \mathbb{N}^4 \times (\mathbb{N}^* \backslash \{1\}),$ we define $$\hat{T}_{(i,j,k,r,l)} := \left\{ \underline{a}\underline{c}^l\underline{\hat{c}}\underline{e} \in \Sigma^-\left(\frac{1}{2}c_{14}\rho\right) \mbox{ such that } |\underline{a}| = i, |\underline{c}| = j, |\underline{e}| = k, |\underline{\hat{c}}|=r \mbox{ e } \underline{c}^l \in \bigcup_{ 1 \leq \beta \leq 2} \Sigma^-(\rho^{\beta \frac{c}{k}}) \right\}.$$

We observe $\displaystyle{ \bigcup_{ \substack{ (i,j,k,l,r) \in \mathbb{N}^4 \times (\mathbb{N}^* \backslash \{1\}) \\ i+jl+k+r \leq m } } \hat{T}_{(i,j,k,l,r)} }$ is the set of all words with scale $\frac{1}{2}c_{14}\rho$ containing some subword with scale $\rho^{\frac{c}{k}}$ repeating itself in intersecting intervals.

We note there is a constant $\hat{D} > 1$ such that $\displaystyle{ \hat{D}^{-1} \leq \frac{ \mu^-(\underline{a}\underline{c}^l\underline{\hat{c}}\underline{e}) }{ \mu^-(\underline{a})\mu^-(\underline{c}^{l})\mu^-(\underline{\hat{c}})\mu^-(\underline{e}) } \leq \hat{D} }.$

Hence, 

\begin{tabular}{rl}

$\displaystyle{ \mu^-(\hat{T}_{(i,j,k,l,r)}) }$ & $\displaystyle{ \leq \sum_{\underline{a} \in \Sigma^{i-}} \sum_{\underline{c} \in \Sigma^{j-}} \sum_{\underline{\hat{c}} \in \Sigma^{r-}} \sum_{\underline{e} \in \Sigma^{k-}} \hat{D} \mu^-(\underline{a})\mu^-(\underline{c}^{l})\mu^-(\underline{\hat{c}})\mu^-(\underline{e}) }$ \\

$$ & $\displaystyle{ = \hat{D} \mu^-(\underline{c}^{l-1}) \sum_{\underline{a} \in \Sigma^{i-}} \mu^-(\underline{a}) \sum_{\underline{c} \in \Sigma^{j-}} \mu^-(\underline{c}) \sum_{\underline{\hat{c}} \in \Sigma^{r-}} \mu^-(\underline{\hat{c}}) \sum_{\underline{e} \in \Sigma^{k-}} \mu^-(\underline{e})}$ \\

$$ & $\displaystyle{ = \hat{D} \mu^-(\underline{c}^{l-1}) }$ \\

$$ & $\displaystyle{ \leq \bar{D} \rho^{\frac{1}{4}\frac{c}{k} \bar{d}_s } },$ for some $\bar{D} > 0.$

\end{tabular}

Therefore,

\begin{tabular}{rl}

$\displaystyle{ \mu^- \left( \bigcup_{ \substack{ (i,j,k,l) \in \mathbb{N}^4 \times (\mathbb{N}^* \backslash \{1\}) \\ i+jl+k+r \leq m } } T_{(i,j,k,l)} \right) }$ & $\displaystyle{ \leq \sum_{ \substack{ (i,j,k,l) \in \mathbb{N}^4 \times (\mathbb{N}^* \backslash \{1\})	 \\ i+jl+k+r \leq m } } \mu^-(T_{(i,j,k,l)}) }$ \\

$$ & $\leq (m+1)^4 \bar{D} \rho^{\frac{c}{k} \bar{d}_s }$ \\

$$ & $= (\tilde{C} \left\lceil \log\left(\frac{1}{2}c_{14}\rho\right)\right\rceil)^4 \bar{D} \rho^{\frac{1}{4}\frac{c}{k} \bar{d}_s }$ \\

$$ & $= \tilde{C}^4 \bar{D} \left\lceil \log\left(\frac{1}{2}c_{14}\rho^2\right)\right\rceil^4 \rho^{\frac{1}{4}\frac{c}{k} \bar{d}_s } \rightarrow 0,$

\end{tabular}

as $\rho$ converges to zero.

\end{proof}

Now we back to the proof of lemma \ref{todas_folhas_boas}.

\begin{proof}[Proof of lemma \ref{todas_folhas_boas}]

We observe, analogously to the previous lemma, we can prove that for every $\varepsilon > 0,$ $\mu^- \left( \Sigma^- \backslash \Sigma^-_{A} \right) < \varepsilon,$ if $A > 0$ has been chosen sufficiently small.

Then it is enough to observe that $\mu^-(\Sigma^- \backslash \Sigma^-_{MB}) \leq \frac{\xi}{2}.$ In other words, $\mu^-(\Sigma^- \backslash \mathcal{W}^-(c_{14}\rho)) < \xi$ if $\rho$ is chosen sufficiently small.

\end{proof}

\begin{defn}{Good piece}
\label{defini_peca_boa_a}

We say $\displaystyle{ (\theta^-, \underline{\theta} ) \in \mathcal{W}^-(c_{14}\rho) \times \Sigma^*_{\theta^-} }$ is a \textbf{good piece} if $\theta^- \underline{\theta} \in \mathcal{W}^-\left(\frac{1}{2}c_{14}\rho\right).$

We denote these pieces by $\Sigma_{B,\theta^-}(c_{14}\rho).$

\end{defn}

Now lets define the pieces we will be dealing with to construct the candidate to recurrent compact set in the projection of $\mathcal{W}^-(c_{14}\rho).$ The candidate for recurrent compact will be essencially the projection of some of these pieces. They will have two distinct properties: `avoid certain recurrencies' and `become itself a non-recurrent good leaf, with some leisure'.

\begin{defn}{Non-recurrent good piece}
\label{defn_peca_muito_boa}

We say $\displaystyle{ (\theta^-, \underline{\theta} ) \in \mathcal{W}^-(c_{14}\rho) \times (\Sigma^*_{(\rho^{\frac{c}{k}},\rho),\theta^-} } \cap \Sigma^*_{A,\theta^-})$ is a \textbf{non-recurrent good piece} if $\theta^- \underline{\theta} \in \mathcal{W}^-(\frac{1}{2}c_{14}\rho).$

We denote this set of non-recurrent good pieces in the leaf $\theta^- \in \Sigma^-$ by $\Theta_{\theta^-}(c_{14}\rho).$

\end{defn}

\begin{lem}
\label{quantidade}

\

Let $X \subset W_{\theta^-}$ be the disjoint union of pieces with scale $\rho$ satisfying $\mu_{\theta_0^-}(X) >0.$ Then $X$ has between $c_{11}^{-1} \mu_{\theta_0^-}(X) \rho^{-\bar{d}_s}$ and $c_{11} \mu_{\theta_0^-}(X) \rho^{-\bar{d}_s}$ with scale $\rho.$

\end{lem}

\begin{proof}

We denote by $X = \bigcup_{\underline{\theta} \in \tilde{X} } \underline{\theta}$ the disjoint union enunciated in the lemma. \\ As $\mu_{\theta_0^-}(X) = \sum_{ \underline{\theta} \in \tilde{X}} \mu_{\theta_0^-}( \underline{\theta} ),$ then, by corollary \ref{cor_gibbs_rho}, $c_{11}^{-1} \rho^{\bar{d}_s} \leq \mu_{\theta_0^-}(\underline{\theta}) \leq c_{11} \rho^{\bar{d}_s}$ for every piece with scale $\rho.$. Then, there is between $c_{11}^{-1} \mu_{\theta_0^-}(X) \rho^{-\bar{d}_s}$ and $c_{11} \mu_{\theta_0^-}(X) \rho^{-\bar{d}_s}$ pieces with scale $\rho.$

\end{proof}

\begin{lem}
\label{Prop_8_15}

\

There is a constant $c_{15} > 0$ such that for every non-recurrent good leaf $\theta^- \in \mathcal{W}^-(c_{14}\rho),$ $\# \Theta_{\theta^-}(\rho) \geq c_{15} \rho^{-\bar{d}_s},$ if $\rho > 0$ is chosen sufficiently small.

\end{lem}

Before we prove this result we need the next lemma which has analogous proof of lemma \ref{lem_mu^-}.

\begin{lem}
\label{lem_muadf}

\

(i) $\mu_{{\theta}_0^-} \left( \Sigma_{{\theta}^-} \backslash \Sigma_{(\beta^{\frac{c}{k}}, \beta ),\theta}^- \right)$ is as small as we want if $\beta > 0$ is chosen sufficiently small.

\vspace{5mm}

(ii) $\mu_{{\theta}_0^-} \left( \Sigma_{{\theta}^-} \backslash \Sigma^-_{A,\theta^-} \right)$ is as small as we want if $A > 0$ is chosen sufficiently small.

\end{lem}

\begin{proof}[Proof of lemma \ref{Prop_8_15}]

Denoting $A_k:= \left\{ \underline{\theta} \in \Sigma_{ ( \rho^{\frac{c}{k}},\rho ), {\theta}^- } ; | \underline{\theta}| = k \right\},$ we get $\displaystyle{ \Sigma_{ ( \rho^{\frac{c}{k}},\rho ), {\theta}^- } = \bigcup_{k \geq 0} A_k.}$

As $\sigma^{-k} \left( \sigma^k(A_k) \backslash \mathcal{W}^-\left(\frac{1}{2}c_{14}\rho\right) \right) = A_k \backslash \Theta_{{\theta}^-}(\rho)$ and $\mu$ is $\sigma-$invariant, then \\ $\mu^- \left( \bigcup_{k \geq 0} \sigma^k(A_k) \backslash \mathcal{W}^-(\frac{1}{2}c_{14}\rho) \right) = \mu_{{\theta}_0^-} \left( \bigcup_{k \geq 0} A_k \backslash \Theta_{{\theta}^-}(\rho) \right).$

But, by lemma \ref{todas_folhas_boas}, $\mu^- \left( \bigcup_{k \geq 0} \sigma^k(A_k) \backslash \mathcal{W}^-\left(\frac{1}{2}c_{14}\rho\right) \right)$ is smaller then $\xi$, and hence, \\ $\mu_{{\theta}_0^-} \left( \bigcup_{k \geq 0} A_k \backslash \Theta_{{\theta}^-}(\rho) \right)$ is less than $\xi$. Therefore, $\displaystyle{ \mu_{{\theta}_0^-} \left( \Sigma_{ ( \rho^{\frac{c}{k}},\rho ),{\theta}^-}  \backslash \Theta_{{\theta}^-}(\rho) \right)}$ is less than $\xi$.

In this way we use assertion \ref{lem_muadf}, $\displaystyle{ \mu_{{\theta}_0^-} \left( \Sigma_{{\theta}^-} \backslash \Theta_{{\theta}^-}(\rho) \right)}$ is, still, less than $\xi,$ if $\rho$ is sufficiently small. Therefore $\mu_{\theta^-_0}\left(\Theta_{\theta^-}(\rho)\right) > 1 - \xi.$

Now, by lemma \ref{quantidade}, $\Theta_{{\theta}^-}(\rho)$ has, at least, $c_{11}^{-1} \mu_{{\theta}_0^-}(\Theta_{{\theta}^-}(\rho)) \rho^{-\bar{d}_s}$ pieces with scale $\rho,$ and hence, at least, $(1-\xi)c_{11}^{-1} \rho^{-\bar{d}_s}$ pieces with scale $\rho.$ It is enough to choose, therefore, $c_{15} := (1-\xi) c_{11}^{-1}.$

\end{proof}

\subsubsection{ Set of stackings and the first lemma on stackings }
\label{contando_pecas}

A stacking is a set of pieces whose projections along the strong-stable foliation are essentially the same.

Given a scale $\rho > 0,$ we denote the fundamental intervals for scale $\rho > 0$ by $I_i := [ (i-1){\rho}, i\rho]$ for $i \in \{1,..., \lceil \rho^{-1} \rceil \}$ and $I_{\lceil \rho^{-1} \rceil + 1} := [ (\lceil \rho^{-1} \rceil){\rho}, 1].$

\begin{defn}{Stacking}

Let $X$ be a set of disjoint pieces in some leaf $\theta^-.$ We say $A \subset X$ is a stacking with fundamental intervals for scale $\rho > 0$ if there is some $i \in \{ 1,...,\lceil \rho^{-1} \rceil + 1 \}$ such that if $\underline{\theta} \in A,$ then \\ $I^{\underline{0}}_{(\theta^-,\underline{\theta})} \cap I_i \ne \emptyset.$

\end{defn}

\begin{defn}{Set of stackings}

We say $ S := \{ A_{j} \}_{j=1}^J $ is a \textbf{set of stackings for $f^{\underline{0}}$ in $X,$ with fundamental intervals for scale $\rho > 0$ and with, at least, $x$ pieces contained in leaf $\theta^-,$ - denoted by $CA(X,\rho,x,\theta^-,\underline{0}$) - } if $A_k$ are stackings with, at least, $x$ pieces in $X$ for every $k \in \{ 1,...,J \}$ and $\displaystyle{\bigcup_{j=1}^{J} A_j}$ is formed by disjoint pieces.

\end{defn}

The following lemma will be useful to separate the pieces with scale $\rho,$ we are going to make move later, by distancies with scale $\rho^{\frac{1}{k}}.$ This separation will be important in order to get some independence of their movements exerted by the perturbation family we will construct.

\begin{lem}{First stacking (with scale $\rho^{\frac{c}{k}}$)}
\label{acavalamento 1}

There are constants $c_{16} > 0$ and $c_{17} > 0,$ independents on $\rho$ such that for each leaf \\ $\theta^- \in \mathcal{W}^-(c_{14}\rho)$ $f,$ there is a $CA\left( \Theta_{\theta^-}(\rho^{\frac{c}{k}}), \rho^{\frac{c}{k}}, c_{16} \rho^{-\frac{c}{k}(\bar{d}_s-1)}, \theta^-, \underline{0} \right),$ $S = \{ A_i \}_{i=1}^{J},$ such that $$ \sum_{i=1}^J \# A_i \geq c_{17} \rho^{-\frac{c}{k}\bar{d}_s}.$$

\end{lem}

\begin{proof}

In order to prove it, we throw away the small stackings and we consider only the big ones. By doing it, there will be no more stackings with few pieces. In this way we still have around $\rho^{-\frac{c}{k}\bar{d}_s}$ pieces with scale $\rho^{\frac{c}{k}}$ in the big stackings. The details follow.

By proposition \ref{Prop_8_15}, there is, at least, $c_{15}^{-1} \rho^{-\frac{c}{k}\bar{d}_s}$ disjoint pieces in $\Theta_{\theta^-}(\rho^{\frac{c}{k}})$ with scale $\rho^\frac{c}{k}$.

We fix a constant $\tilde{Q} > 0$ to be chosen sufficiently small and we consider

$\hat{S}:= \left\{ \underline{\theta} \in \Theta_{\theta^-}(\rho^{\frac{c}{k}}) \text{ such that } Y(\hat{S}) < \tilde{Q} \rho^{-\frac{c}{k}(\bar{d}_s - 1 )} \mbox{ for every } 1 \leq i \leq \lceil \rho^{-\frac{c}{k}} \rceil +1 \right\},$ \\ where $Y(\hat{S}) := \# \left\{ \underline{\tilde{\theta}} \in \Theta_{\theta^-}(\rho^{\frac{c}{k}}) \text{ such that } I_{(\theta^-,\underline{\tilde{\theta}})} \cap I_{i} \neq \emptyset \right\}.$

Now, $\displaystyle{ \# \hat{S} \leq \sum_{i=1}^{\rho^{-\frac{c}{k}}} \tilde{Q} \rho^{-\frac{c}{k}(\bar{d}_s-1)} \leq \rho^{-\frac{c}{k}} \tilde{Q} \rho^{-\frac{c}{k}(\bar{d}_s-1)} = \tilde{Q} \rho^{-\frac{c}{k} \bar{d}_s } }.$

Let $\tilde{S} \subset \Theta_{\theta^-}(\rho^{\frac{c}{k}}) \backslash \hat{S}$ be a partition, formed by elements in $\Theta_{\theta^-}(\rho^{\frac{c}{k}}) \backslash \hat{S}$ for $\Theta_{\theta^-}(\rho^{\frac{c}{k}}) \backslash \hat{S}.$ By definition of $\tilde{S},$ there are, at least, $(c_{15}^{-1} - \tilde{Q}) \rho^{-\frac{c}{k}\bar{d}_s}$ pieces in $\tilde{S}.$

We define the set of stackings $S := \left\{ A_i \right\}_{i=1}^J,$ in such a way that each stacking is composed by $c_{16} \rho^{-\frac{c}{k}(\bar{d}_s-1)}$ pieces in $\Theta_{\theta^-}(\rho^{\frac{c}{k}})$ in such a way tha its stackings do not share pieces, where $c_{16}$ is some positive fraction of $c_{15} - \tilde{Q}.$

Choosing $\tilde{Q} > 0$ sufficiently small, e.g. $\tilde{Q} < c_{15}^{-1},$ $\displaystyle{ \sum_{i=1}^J \# A_i \geq (c_{15}^{-1} - \tilde{Q}) \rho^{-\frac{c}{k}\bar{d}_s}}.$

Enough to choose $c_{17} := c_{15}^{-1} - \tilde{Q}$ and observe that $S$ is a set of stacking we are looking for.

\end{proof}

\subsubsection{Well-spaced stacking}

In this section we construct a set of stacking with, at least, $\rho^{-(\bar{d}_s - 1)}$ pieces with scale $\rho$ in each stacking such that $\rho^{-\frac{c}{k}(\bar{d}_s - 1)}$ of them are well-spaced (they are contained in different non-intersecting pieces with scale $\rho^{\frac{c}{k}}$) and whose projections have Lebesgue measure bounded below by some positive constant independent of $\rho.$

\begin{lem}{Using the Marstrand-like argument}
\label{lema_medida_de_Lebesgue_folha_muito_boa}

There is a constant $c_{21} > 0$ such that if $\theta^- \in \mathcal{W}^-(c_{14}\rho)$ is a leaf and $X$ is a set composed by, at least, $z \rho^{-\bar{d}_s}$ pieces in $\Sigma^*_{(\rho^{\frac{c}{k}},\rho),\theta^-}$ contained in disjoint pieces in $\Sigma^*_{A,\theta^-},$ then $Leb( \pi^{\underline{\omega}}_{\theta^-}(X) ) \geq c_{21} z^{-2}$ for every $\underline{\omega} \in \Omega.$

\end{lem}

\begin{proof}

Given $\theta^- \in \mathcal{W}^-(c_{14}\rho),$ then $\displaystyle{ \mu_{{\theta}_0^-}(X) \leq \nu^{\underline{0}}_{{\theta}^-}(\pi^{\underline{0}}_{{\theta}^-}(X)) \leq \int_{\pi^{\underline{0}}_{{\theta}^-}( X )} \frac{d\nu^{\underline{0}}_{{\theta}^-}}{dLeb} dLeb } .$

\vspace{10mm}

Therefore, by Cauchy-Schwarz theorem and due to the fact tha the pieces in consideration are contained in pieces in $\Sigma^{*}_{A,{\theta}^-},$ then if $\rho$ is chosen sufficiently small, $\displaystyle{ \mu_{\theta^-_0}(X) \leq Leb( \pi^{\underline{0}}_{{\theta}^-}(X) )^{\frac{1}{2}} . \left\| \frac{d\nu^{\underline{0}}_{{\theta}^-}}{dLeb} \right\|_{L_2} },$ and, hence,

\begin{tabular}{rl}

$\displaystyle{Leb( \pi_{{\theta}^-}( X ) ) \geq}$ & $\displaystyle{ \mu(X)^2 . \Bigl\| \frac{d\nu^{\underline{0}}_{{\theta}^-}}{dLeb} \Bigr\|^{-2}_{L_2} }$\\

$\geq$ & $\displaystyle{ ( (z \rho^{-\bar{d}_s}) . c_{11}^{-1} \rho^{\bar{d}_s})^{2} .K_1^{-1}}, \mbox{ by lemma \ref{cor_gibbs_rho}}$\\

$\geq$ & $\displaystyle{ (z \rho^{-\bar{d}_s} c_{11}^{-1} \rho^{\bar{d}_s})^{2} .K_1^{-1}}$\\

$=$ & $\displaystyle{ (z c_{11}^{-1})^{2}  K_1^{-1} }.$

\end{tabular}

We define $\tilde{c}_{21} := {c}_{11}^{2} K_1^{-1}.$ Therefore, there is a positive fraction of $\tilde{c}_{21},$ say $c_{21},$ such that $Leb( \pi^{\underline{\omega}}_{\tilde{\theta}^-}( X ) ) \geq c_{21}$ for every $\theta^- \in \mathcal{W}^-(c_{14}\rho)$ and $\underline{\omega} \in \Omega.$

\end{proof}

We denote the pieces in $\Theta_{\theta^-}(\rho)$ which begins with $\underline{\theta}$ by $\Theta_{(\theta^-,\underline{\theta})}(\rho).$

\begin{lem}{Substackings}
\label{subacavalamentos}

There are constants $c_{19} > c_{18} > 0,$ and $c_{20} > 0$ such that if $\rho > 0$ is chosen sufficiently small then for every leaf $\theta^- \in \mathcal{W}^-(c_{14}\rho)$ and any piece $\underline{\theta} \in \Sigma_{\theta_0^-}^*(\rho^{\frac{c}{k}}),$ there is a $$CA\left( \Theta_{(\theta^-,\underline{\theta})}(\rho), \rho, c_{18} \rho^{-(1-\frac{c}{k})(\bar{d}_s-1)}, \theta^-, \underline{0} \right),$$ $S = \{ A_i \}_{i=1}^{J},$ satisfying:

(i) $\sum_{i=1}^J \#A_i \geq c_{20} \rho^{-(1-\frac{c}{k})\bar{d}_s}$

(ii) $\#A_i \leq c_{19} \rho^{-(1-\frac{c}{k})(\bar{d}_s-1)}.$

\end{lem}

\begin{proof}

Analogously to the proof of lemma \ref{acavalamento 1}, we can consider $\Theta_{(\theta^-,\underline{\theta})}(\rho)$ has a \\ $CA\left( \Theta_{(\theta^-,\underline{\theta})}(\rho), \rho, c_{18} \rho^{-(1-\frac{c}{k})(\bar{d}_s-1)}, \theta^-, \underline{0} \right),$ $S = \{ A_i \}_{i=1}^{J},$ satisfying $\sum_{i=1}^J \#A_i \geq \tilde{c}_{20} \rho^{-(1-\frac{c}{k})\bar{d}_s},$ for some constant $\tilde{c}_{20} > 0.$

Lets prove we can consider these stackings satisfyig, also, property (ii).

If we iterate these stackings backward, by the diffeomorphism, $f,$ till the piece $(\theta^-,\underline{\theta})$ turns itself a leaf $\theta^-\underline{\theta},$ we obtain a $CA\left( \tilde{\Sigma}_{\theta^-\underline{\theta}}(\rho^{1-\frac{c}{k}}), \rho^{1-\frac{c}{k}}, c_{18} \rho^{-(1-\frac{c}{k})(\bar{d}_s-1)}, \theta^-\underline{\theta}, \underline{0} \right),$ $\tilde{S} = \{ \tilde{A}_i \}_{i=1}^{J},$ satisfying $\sum_{i=1}^J \#\tilde{A}_i \geq c_{20} \rho^{-(1-\frac{c}{k})\bar{d}_s},$ where $\tilde{\Sigma}_{\theta^-}(\rho)$ are pieces with relaxed scale $\rho,$ that is, the same as $\Sigma_{\theta^-}(\rho)$ with a constant $C > c_1$ sufficiently big replaced in the constant $c_1$ defining $\Sigma_{\theta^-}(\rho).$ This constant is chosen in such a way that any word with scale $\rho$ is transformed, by those backward iterates by the diffeomorphism $f,$ in pieces with relaxed scale $\rho^{1-\frac{c}{k}}.$ This is possible because the diameter of the piece is, essencially, the diameter of the pieces which are its preimages by $f$ multiplied by the derivative of $f$ in the weak-stable direction of some point in its preimages. As the weak stable direction is H\"older continuous and $f$ is $C^{\infty},$ then these derivatives satisfy the bounded distortion property. This means the diameters of the pieces distort a little by $f.$

Given this stacking, $\tilde{S},$ lemma \ref{lema_medida_de_Lebesgue_folha_muito_boa} guarantees the Lebesgue measure of the projection of these stackings is bounded below by some positive constant, say $\tilde{C} > 0.$ Hence, if we admit these stackings, $\tilde{A}_i,$ have less than $c_{19} \rho^{-(1-\frac{c}{k})(\bar{d}_s-1)}$ pieces - being sufficient for this sake withdraw the exceding ones - the remaind pieces will form a new stacking, $\{ \tilde{B}_i \}_{i=1}^{j},$ with pieces contained in the previous stackings and in such a way that the projection of the new set of stackings still has Lebesgue measure bounded below by a positive fraction of the constant $\tilde{C}.$ Therefore, this stacking will have $c_{20} \rho^{-1} \rho^{-(1-\frac{c}{k})(\bar{d}_s-1)} = c_{20} \rho^{-\frac{c}{k}(\bar{d}_s -1)}$ pieces, where $c_{20}$ is some positive constant.

Now, it is enough to iterate forward, by the diffeomorphism $f,$ these stackings $\tilde{B}_j$ till its pieces back to the leaf $\theta^-$, in order to obtain a set of stackings, whose stackings are subsets of the stackings $A_{j}$ and such that they satisfy property (ii).

\end{proof}

The next lemma is the objective of this section. In the next section we extend this lemma for perturbations of $f$ in the fiexed perturbation family.

\begin{lem}{Well-distributed stacking with scale $\rho$}
\label{lema_acavalamentos_bem_distribuidos}

There are positive constants $c_{22} > 0,$ $c_{23} > 0$ and $c_{24} > 0$ such that for every $\theta^- \in \mathcal{W}^-(c_{14}\rho),$ there is a $CA( \Theta_{\theta^-}(\rho), \rho, c_{22} \rho^{-(\bar{d}_s-1)}, \theta^-, {\underline{0}} ),$ $T,$ in such a way that in each stacking, the pieces with scale $\rho$ are well-distributed, i.e., they are distributed in at least $c_{24} \rho^{-\frac{c}{k}(\bar{d}_s-1)}$ distinct pieces with scale $\rho^\frac{c}{k}.$

Beyond that, the Lebesgue measure of the projection of pieces in the stackings in this set of stackings, along the strong-stable foliation, is bounded below by $c_{23}.$

\end{lem}

\begin{proof}

We denote by $S_{\theta^-}$ and by $S_{\theta^-,\underline{\theta}}$ the stackings refering to lemmas \ref{acavalamento 1} and \ref{subacavalamentos} respectively. Let $Z_{\theta^-} := \{ \underline{\beta} \in A_{\underline{\theta}} ; A_{\underline{\theta}} \in S_{(\theta^-,\underline{\theta})}, \underline{\theta} \in A, A \in S_{\theta^-} \}.$

By lemmas \ref{acavalamento 1} and \ref{subacavalamentos}, there are, at least, $( c_{20} \rho^{-\bar{d}_s(1-\frac{c}{k})} ) \# S_{\theta^-} \geq  c_{20} \rho^{-\bar{d}_s(1-\frac{c}{k})} c_{17} \rho^{-\frac{c}{k}\bar{d}_s}$ pieces in $Z_{\theta^-}.$ In this way, there are, at least, $c_{20} c_{17} \rho^{-\bar{d}_s}$ pieces in $Z_{\theta^-}.$

The set of stackings enunciated in this lemma will the a set, $T,$ of words inside $Z_{\theta^-}.$

With analogous arguments to the proof of lemma \ref{acavalamento 1}, we show there is a

$CA(\Theta_{\theta^-}(\rho), \rho, c_{22} \rho^{-(\bar{d}_s-1)} , \theta^-, {\underline{0}} ),$ $T,$ with pieces in $Z_{\theta^-}$ for some $c_{22} >0$ satisfying \\ $\# T \geq \tilde{c}_{23} \rho^{-\bar{d}_s}$ for some $\tilde{c}_{23} > 0.$

In order to prove that $Leb(\Pi^{\underline{\omega}}_{\theta^-} (T)) \geq c_{23} > 0$ for some $c_{23} >0,$ it is enough to observe that by lemma \ref{lema_medida_de_Lebesgue_folha_muito_boa}, since there are, at least, ${\tilde{c}_{23}} \rho^{-\bar{d}_s}$ pieces with scale $\rho$ (see lemma \ref{acavalamento 1}), for some constant $\tilde{c}_{23} > 0,$ then $Leb(T) \geq \tilde{c}_{23}^{-2} c_{21}.$ Now, we choose $c_{23} := c_{21} \tilde{c}_{23}^{-2} > 0.$

Beyond that, by item (ii) of lemma \ref{subacavalamentos}, they are well-distributed, i.e., for each stacking there is, at least, one piece with scale $\rho$ among the $c_{24} \rho^{-(\bar{d}_s-1)\frac{c}{k}}$ different pieces with scale $\rho^{\frac{c}{k}},$ since $\displaystyle{ \frac{ c_{22} \rho^{-(\bar{d}_s-1)} }{ c_{19} \rho^{-(\bar{d}_s-1) (1 - \frac{c}{k})} } = \frac{c_{22}}{c_{19}} \rho^{-\frac{c}{k}(\bar{d}_s-1)} = c_{24} \rho^{-\frac{c}{k}(\bar{d}_s-1)}},$ being enough to choose $c_{24}:=\frac{c_{22}}{c_{19}}.$

\end{proof}

\subsubsection{ Well-spaced relaxed stackings for the perturbed diffeomorphisms }

We fix a constant $c_{25} > 0.$ Given the interval $I,$ we call the the one with the same center as $I$ and length $c$ times the one of $I$ by $cI.$

\begin{defn}{Relaxed stackings}

Let $X$ be the set of disjoint pieces in some leaf $\theta^-.$ We say $A \subset X$ is a relaxed stacking with scale $\rho$ for $f^{\underline{\omega}}$ if there is some $i \in \{ 1,...,\lceil \rho^{-1} \rceil + 1 \}$ such that if $\underline{\theta} \in A,$ then $I^{\underline{\omega}}_{\underline{\theta}} \cap c_{25}I_i \ne \emptyset.$

\end{defn}

\begin{defn}{Set of relaxed stackings}

We say $ S := \{ A_{j} \}_{j=1}^J $ is a \textbf{set of relaxed stackings for $f^{{\underline{\omega}}}$ in $X$ with fundamental intervals with scale $\rho$ and with, at least, $x$ pieces contained in the leaf $\theta^-,$ - denoted by $CAf(X,\rho,x,\theta^-,\underline{\omega}$) - } if $A_k$ are relaxed stackings with, at least, $x$ pieces in $X$ for every $k \in \{ 1,...,J \}$ and $\displaystyle{ \bigcup_{k=1}^J A_k}$ is formed by disjoint pieces.

\end{defn}

In order to prove property \ref{prob} for the candidate of recurrent compact set we will construct it is not sufficient to perturb only the first preimage of the piece with scale $\rho^{\frac{c}{k}}$ containing the piece with scale $\rho$ we desire to be moving since the contributions to the displacements of this piece due to the other pieces in its preimages - this displacement can be a consequence of the first preimage or of the perturbations in other pieces in other leaves - can have the same approximate size as the one due to the first preimage. This could cancel the effects we desire - displacement with scale of a big fixed multiple of $\rho$ are sufficient in order that the renormalization operator applied in $(x,\theta^-)$ returns in the relaxed interior of $K$ and traverses it (we remember we need to back in the relaxed interior of $K$ in order to guarantee the run of the probabilistic argument).

As it is necessary to displace the pieces in a stacking, independently, we perturb the first (the first $L \in \mathbb{N}$) preimages of the piece we want to displace. If this constant, $L,$ is sufficiently big (depending only on $f$), the contributions due to others preimages will have scale of a very small factor of $\rho,$ in such a way that the effect of these perturbations with scale $\rho$ in the first $L$ preimages of $(x,\theta^-)$ are sufficiently strong in order to $R^{\underline{\omega}}_{\underline{a}}(x,\theta^-)$ backs in a relaxed iterior of $K$ and traverses it.

Lets, from now on, consider $L \in \mathbb{N}$ to be fixed in short, depending only on $f$ - specifically on the derivative of $f$ in the weak stable direction. After it, we fix $0 < \kappa < 1$ given by lemma $\ref{prop_erranc_peca}.$

For each $(\theta^-, \theta^+) \in \Sigma,$ we define, in the sequel, the coordinates of the multiparameter $\underline{\omega}$ which will exert considerably influence on the movement of $(\theta^-,{\theta}^+)$ when perturbing the pieces with scale $\rho^{\frac{c}{k}}$ whose first $L$ forward trajectories contains $(\theta^-,{\theta}^+).$

\begin{defn}
\label{Sigma_1_nr}

\

$\Sigma_{1}(\theta^-, \theta^+):=\Bigl\{ \underline{a}^i:=(\underline{a}^{i-},\underline{a}^{i+}) \in \Sigma_{1} ; f^{-i}(h(\theta^-,\theta^+)) \in h(\underline{a}^i) \Bigr\}$ (see definition of $\Sigma_1$ in section \ref{sec_perturba}).

\end{defn}

These are the coordinates in the multiparameters $\omegab \in \Omega$ exerting considerably influence on the displacements of the associated pieces to the iterates, that is, the ones inside the set $\bigcap_{i=0}^L f^i( h(\underline{a}^i) ).$ A historic, backward, of a piece having considerably influence but only till iterate, say $L,$ independent of $\rho.$

Now we define the parameters which will be perturbed in order to exert influence on the movement of $(\theta^-,{\theta}^+).$

\begin{defn}
\label{Omega_omega}

\

$\Omega_1^{\omegab}(\theta^-, \theta^+) := \{ \tilde{\underline{\omega}} \in \Omega \mbox{ such that } \tilde{{\omega}}_{\underline{a}} = {\omega}_{\underline{a}} \mbox{ for every } \underline{a} \in \Sigma_{1} \backslash \Sigma_{1}(\theta^-,\theta^+) \}.$ 

\end{defn}

\begin{remark}

\

$\Sigma_{1}(\theta^-,\theta^+) = \Sigma_{1}(\tilde{\theta}^-,\tilde{\theta}^+)$ and, hence, $\Omega_1^{\omegab}(\theta^-,\theta^+) = \Omega_1^{\omegab}(\tilde{\theta}^-,\tilde{\theta}^+)$ if $\theta^+$ and $\tilde{\theta}^+$ share a comum begining with scale $\rho,$ $\theta^-$ and $\tilde{\theta}^-$ share a comum final with scale $\rho$ and if $\rho$ is sufficiently small. \index{$\Sigma_{2}(\theta^-,\theta^+)$}

\end{remark}

\begin{defn}
\label{def_discretizada_Sigma_Omega}

\

For each $\underline{\theta}^+ \in \Sigma(\rho),$ we define $\Sigma_{1}(\theta^-,\underline{\theta}^+) :=  \Sigma_{1}(\theta^-,\theta^+)$ and $\Omega_1^{\omegab}(\theta^-,\underline{\theta}^+) :=  \Omega_1^{\omegab}(\theta^-,\theta^+),$ where $\theta^+$ begins with $\underline{\theta}^+.$

\end{defn}

Now we define the parameters having zero in the coordinates exerting considerably influence on the pieces in a same set of stackings, $S.$

\begin{defn}

\

$\Omega_0(S) := \{ \tilde{\omega} \in \Omega \mbox{ such that } \tilde{\omega}_{\underline{a}} = 0, \mbox{ } \forall \underline{a} \in \Sigma_1(S) \}.$

\end{defn}

\begin{lem}
\label{sobre os cruzamentos dos tubos}

\

If $S$ is a $CA(\Theta_{\theta^-}(\rho),\rho,x,\theta^-,\underline{0})$ for $f^{\underline{0}}$ of a leaf $\theta^- \in \mathcal{W}^-,$ then $S$ is a $CAf(\Theta_{\theta^-}(\rho),\rho,x,\theta^-,\underline{\omega})$ for every ${\omegab} \in \Omega_0(S),$ if $L$ is chosen sufficiently big.

\end{lem}

\begin{proof}

It is enough to apply parts (c) and (d) of lemma \ref{prop_erranc_peca} and lemma \ref{folha_nao_recorrente} in order to conclude these pieces in a same stacking of a leaf of $\mathcal{W}^-$ moves at most with approximate size $c_{25} \rho$ ones with respect to the others. For this sake,  we need to choose $L$ sufficiently big, depending on $c_3$ (the size of the perturbation) and $c_{25}$ (the stacking relaxing).

\end{proof}

The next proposition is a consequence of lemmas \ref{sobre os cruzamentos dos tubos} an \ref{lema_acavalamentos_bem_distribuidos}.

\begin{prop}{Relaxed stackings for $f^{\underline{\omega}},$ with scale $\rho$ and well-distributed}
\label{lema_acav_bem_distr_omega}

There are positive constants $c_{26} > 0,$ $c_{27} > 0$ and $c_{28} > 0$ such that for any $\underline{\omega} \in \Omega$ and \\ $\theta^- \in \mathcal{W}^-(c_{14}\rho),$ the set of stackings, $S := \bigcup_{i=1}^J A_i,$ given by lemma \ref{lema_acavalamentos_bem_distribuidos}, is a \\ $CAf( \Theta_{\theta^-}(\rho), \rho, c_{26} \rho^{-(\bar{d}_s-1)}, \theta^-, {\underline{{\omega}}} )$ for $f^{\underline{{\omega}}},$ for every ${\omegab} \in \Omega_0(S),$ such that in each stacking, the pieces with scale $\rho$ are well-distributed, i.e., they are distributed in, at least, $c_{28} \rho^{-\frac{c}{k}(\bar{d}_s-1)}$ distinct pieces,
with scale $\rho^\frac{c}{k}.$

Beyond that, if $\rho$ is chosen sufficiently small and if $\theta^-$ and $\tilde{\theta}^-$ are in the same block in $\mathcal{W}^-(c_{14}\rho),$ then the sets of relaxed stackings for them can be considered formed by the same cylinders. Also, the Lebesgue measure of the projections of the pieces in the stacking of these sets of relaxed stackings, along the strong-stable foliation, is bounded below by $c_{27} > 0.$

\end{prop}

For each $\theta^- \in \mathcal{W}^-(c_{14}\rho),$ we denote the stackings given by lemma \ref{lema_acav_bem_distr_omega} by $$S_{\theta^-} := \left\{ A_{\theta^-}(i) \right\}_{i=1}^{J_{\theta^-}}.$$

\begin{defn}{Candidate for recurrent compact, $K$}

Let $I_{a_{\theta^-}(i)}$ be the standard intervals refering to the stackings $A_{\theta^-}(i)$ for each $i \in \{ 1,..., J_{\theta^-} \}.$ For every $\theta^- \in \mathcal{W}^-(c_{14}\rho),$ we define $\displaystyle{ K_{\theta^-} := \bigcup_{i=1}^{J_{\theta^-}} I_{a_{\theta^-}(i)}}.$

The candidate for recurrent compact is $$\displaystyle{ K := \bigcup_{\theta^- \in \mathcal{W}^-(c_{14}\rho)} K_{\theta^-} }.$$

\end{defn}

We denote $K \cap H_{\theta^-}$ by $K_{\theta^-}$ and observe that $Leb(K_{\theta^-}) > c_{27}$ for every $\theta^- \in \mathcal{W}^-(c_{14}\rho)$ and that $K_{\theta^-} = K_{\tilde{\theta}^-}$ for any leaves $\theta^-$ and $\tilde{\theta}^-$ in a same block in $\mathcal{W}^-(c_{14}\rho).$ This is important in order to get the relaxations $K_{-\rho^2}$ of $K.$

\subsection{ Proof of proposition \ref{prob} }
\label{demonstr_teore_prob}

Lets reenunciate property \ref{prob}.

\begin{prop}{Main proposition}
\label{prop__prob}

There are positive constants $c_{13} > 0$ and $c_{14} > 0$ such that for every $\rho>0$ sufficiently small, there is $K \subset H$ and a subset of leaves, $\mathcal{W}^-,$ $c_{14}\rho-$dense in the leaves whose projections contains $K$ such that if $(x,\theta^-) \in K \cap \mathcal{W}^-,$ then $$\mathbb{P} \Bigl( \Omega \setminus \Omega_{0,\rho^2}(x,\theta^-) \Bigr) \leq exp(- c_{13} \rho^{-\frac{c}{k}(\bar{d}_s-1)} ).$$

\end{prop}

The following lemma is a consequence of lemma \ref{prop_erranc_folha}

\begin{lem}{Dispersion control of the strong stable foliation in never-recurrent leaves}
\label{folha_esta_forte_nao_se_move}

For any $\theta^- \in \mathcal{W}^-,$ $\displaystyle{ \left( \Pi^{\underline{\omega}}_{\theta^-} \right)'(z) \underline{\bar{\omega}} = 0, }$ for every $z \in W_{\theta^-},$ $\underline{\omega} \in \Omega$ and $\underline{\bar{\omega}} \in \Gamma_{\theta^-,\underline{\omega}},$ where $\Gamma_{\theta^-,\underline{\omega}}$ are parameters $\underline{\bar{\omega}}$ such that the values of the coordinates corresponding to the first $L$ preimages of pieces in $\Sigma_1$ not intersecting $W_{\theta^-}$ are fixed in the values of the corresponding coordinates of $\underline{\omega}.$

\end{lem}

\begin{proof}

It is enough to observe that the preimages (till iterate $L$) of the pieces in $\Sigma_1$ intersecting $\theta^-$ do not also intersect $\theta^-,$ by non-recurrence of $\theta^-.$ In the sequel we apply lemma \ref{prop_erranc_folha}.

\end{proof}

The following lemma will help us to prove that the non-recurrent pieces in the never-recurrent leaves, $\theta^- \in \mathcal{W}^-,$ present independent movements relative to $\mathcal{F}^{ss}\cap W_{\theta^-}.$

\begin{lem}
\label{renormaliza_0_igual_omega}

\

$\mathcal{F}_{\theta^-\underline{\theta}}^{\underline{\omega}^{'},ss} = \mathcal{F}_{\theta^-\underline{\theta}}^{\underline{\omega},ss}$ for every $\theta^- \in \mathcal{W}^-,$ $\underline{\theta} \in \Theta_{\theta^-}(\rho)$ and $\underline{\omega},$ $\underline{\omega}^{'}$ in $\Omega$ satisfying $\omega_{\underline{a}} = \omega_{\underline{a}}^{'}$ for every $\displaystyle{ \underline{a} \in \left( \Sigma_1(\theta^-,\underline{\theta}) \cup \left( \Sigma_1 \backslash \Sigma_1(W_{\theta^-}) \right) \right) }.$

\end{lem}

\begin{proof}

Analogous to the proof of lemma \ref{prop_erranc_folha}. It is enough to observe that $\underline{\theta}$ is a non-recurrent piece in $\theta^-,$ that $\theta^-$ is a never-recurrent leaf in $\theta^-$ and that, then, we are not touching the parameters exerting influence in the forward iterates of $\theta^-\underline{\theta}.$

\end{proof}

\begin{lem}
\label{c_4}

\

There is a constant $c_{32} > 0$ such that for $L$ fixed as in lemma \ref{sobre os cruzamentos dos tubos} and $\kappa$ given by lemma $\ref{prop_erranc_peca},$ refering to this $L,$ there is $\rho_0$ such that if $0 < \rho < \rho_0,$ then:

(a) If $(\theta^-,\underline{\theta}) \in \mathcal{W}^-(c_{14}\rho) \times \Theta_{\theta^-}(\rho),$ and $\theta \in (\theta^-,\underline{\theta}),$ then $\sigma^j(\theta)$ falls in a same element $\underline{a}$ in partition $\Sigma_1$ at most once for all those $0 \leq j \leq |\underline{\theta}|.$

(b) If $(\theta^-,\underline{\theta}) \in \mathcal{W}^-(c_{14}\rho) \times \Theta_{\theta^-}(\rho)$ is such that $\sigma^j(\theta^-,\underline{\theta}) \subset \underline{a} \in \Sigma_1$ with $0 \leq j \leq L,$ then $$\lambda^j c_3 c_{32} < \displaystyle{ \biggl| \frac{\partial }{\partial {\omega}_{\underline{a}}} R_{\underline{\theta}}^{\underline{\omega}}(x,\theta^-) \biggr| < \hat{\lambda}^j c_3 c_{32}, \mbox{ for every } x \in H_{\theta^-} }.$$

(c) If $(\theta^-,\underline{\theta}) \in \mathcal{W}^-(c_{14}\rho) \times \Theta_{\theta^-}(\rho)$ is such that $\sigma^i(\theta^-,\underline{\theta}) \subsetneq \underline{a} \in \Sigma_1$ for every $0 \leq i \leq L,$ then $$ \displaystyle{ \biggl| \frac{\partial }{\partial {\omega}_{\underline{a}}} R_{\underline{\theta}}^{\underline{\omega}}(x,\theta^-) \biggr| < \hat{\lambda}^i c_{32}c_3 }, \mbox{ for every } x \in H_{\theta^-}.$$

(d) If $(\theta^-,\underline{\theta}) \in \mathcal{W}^- \times \Theta_{\theta^-}(\rho)$ and $\underline{a} \in \Sigma_1$ are such that $\underline{a} \in \Sigma_1(W_{\theta^-}) \backslash \Sigma_1(\theta^-,\underline{\theta}),$ then $$\displaystyle{  \frac{\partial }{\partial {\omega}_{\underline{a}}} R_{\underline{\theta}}^{\underline{\omega}}(x,\theta^-) = 0} \mbox{ for every } x \in H_{\theta^-} .$$

\end{lem}

\begin{proof}

Part (a) follows from the same part in lemma \ref{prop_erranc_peca}. Parts (b) and (c) follow from the same parts in lemma \ref{prop_erranc_peca}, \ref{folha_esta_forte_nao_se_move} an the fact that the renormalization operator associated to pieces with scale $\rho$ expand them till they strike scale 1. This means if the piece with scale $\rho$ moves with velocity approximately $\rho$ with respect to a strong stable leaf, then the corresponding backward iterate of this strong stable leaf moves with velocity approximately 1. Part (d) follows from lemma \ref{renormaliza_0_igual_omega}, from part (d) of lemma \ref{prop_erranc_peca} and the fact that if $\underline{\theta} \in \Sigma_{\theta^-}(\rho)$ does not back in the leaf $\theta^-$ by backwards iterates till $(\theta^-,\underline{\theta})$ turns itself a leaf.

\end{proof}

\begin{thm}
\label{pro_maior_que_P}

\

If $c_3 > 0$ is sufficiently big, there is a constant $P > 0,$ such that for every $\rho > 0$ sufficiently small, $\theta^- \in \mathcal{W}^-(c_{14}\rho),$ $x \in H_{\theta^-},$ $\underline{\theta} \in \Theta_{\theta^-}(\rho)$ and $\underline{\omega} \in \Omega$ with $x \in int(\Pi_{\theta^-}^{\underline{\omega}}(\underline{\theta})),$ then $$\mathbb{P}_{\Omega_1^{\underline{\omega}}(\theta^-,\underline{\theta})} \Bigl( R^{\underline{\omega}}_{\underline{\theta}}(x,\theta^-) \in K_{-\rho^2} \Bigr) > P.$$

\end{thm}

\begin{proof}

Let $\tilde{\Omega}_1^{\underline{\omega}}(\theta^-,\underline{\theta}) \subset {\Omega_1}^{\underline{\omega}}(\theta^-,\underline{\theta})$ be a thin tube (for example, with width $\rho$) around $\underline{0}$ in directions $\underline{a}^i \in \Sigma_1(\theta^-,\underline{\theta})$ for each $i$ between $1$ and $L.$

That is, $\tilde{\Omega}_{1}^{\underline{\omega}}(\theta^-,\underline{\theta}):= \left( [-1,1] \underline{a}^0 \times \prod_{i=1}^L [-\rho,\rho] \underline{a}^i \times \prod_{\underline{b} \notin \{ \underline{a}^0,\underline{a}^1,...,\underline{a}^L \} } [-1,1]\underline{b} \right) \cap \Omega_1^{\underline{\omega}}(\theta^-,\underline{\theta}).$

As $R^{\underline{0}}_{\underline{\theta}}(x,\theta^-) \in H_{\theta^- \underline{\theta}}$ if $x \in \Pi^{\underline{0}}_{\theta^-}(\underline{\theta}),$ then $R^{\underline{\tilde{\omega}}}_{\underline{\theta}}(x,\theta^-)$ is sufficiently close to $H_{\theta^- \underline{\theta}}$ for every $\underline{\tilde{\omega}} \in 0 \underline{a}^0 \times \prod_{i=1}^L 0 \underline{a}^i \times \prod_{\underline{b} \in \{ \underline{a}^0,\underline{a}^1,...,\underline{a}^L \} } [-1,1]\underline{b}$ (Being enough, for this sake, choose sufficiently big $L$). Beyond it, $\theta^-\underline{\theta}$ is in some block in $\mathcal{W}^-(\frac{1}{2}c_{14}\rho),$ by definition of $\Theta_{\theta^-}(\rho).$

Therefore, $P_{ \tilde{\Omega}_1^{\underline{\omega}}(\theta^-,\underline{\theta}) } \left(  R^{\underline{\omega}}_{\underline{\theta}}(x,\theta^-) \in K_{-\rho^2} \right) > \tilde{P}$ for some positive constant $\tilde{P} > 0,$ if $\rho > 0$ is chosen sufficiently big in such a way that we can apply lemma \ref{c_4} and if $c_3 > 0$ is chosen sufficiently big in such a way that the pieces traverses the strong stable leaves. Therefore, $P_{ {\Omega}_1^{\underline{\omega}}(\theta^-,\underline{\theta}) } \left(  R^{\underline{\omega}}_{\underline{\theta}}(x,\theta^-) \in K_{-\rho^2} \right) > {P}.$

\end{proof}

Now we can prove that the candidate for recurrent compact set, $K,$ satisfies property \ref{prob}.

\begin{proof}[Proof of proposition \ref{prop__prob}]

As, by lemma \ref{lema_acav_bem_distr_omega}, each of the stackings defining $K$ has $c_{28} \rho^{-(\bar{d}_s-1)\frac{c}{k}}$ well-separated pieces with scale $\rho,$ then any $x \in K$ is in the relaxed projection of $c_{28} \rho^{-(\bar{d}_s-1)\frac{c}{k}}$ well separated pieces with scale $\rho.$

Let $\Omega_1^{\underline{\omega}}(S) := \{ \tilde{\underline{\omega}} \in \Omega \mbox{ such that } \tilde{\omega}_{\underline{a}} = \omega_{\underline{a}} \mbox{ for every } \underline{a} \in \Sigma \backslash \Sigma_1(S) \}.$

Then, for every $\underline{\omega} \in \Omega,$ $\theta^- \in \mathcal{W}^-,$ $(x,\theta^-) \in K$ and $1 \leq i \leq J_{\theta^-},$ the events $$\left\{ \underline{\tilde{\omega}} \in \Omega_1^{\underline{\omega}}(A_{\theta^-}(i))  \mbox{ such that } R^{\underline{\tilde{\omega}}}_{\underline{a}}(x,\theta^-) \in K_{-\rho^2}   \right\}_{\underline{a} \in A_{\theta^-}(i) }$$ are mutually independents and, beyond that,

$$\mathbb{P}_{\Omega_1^{\underline{\omega}}(\theta^-,\underline{a})} \Bigl(  R^{\underline{\omega}}_{\underline{a}}(x,\theta^-) \in K_{-\rho^2} \Bigr) > P,$$ for some constant $P>0.$

To see this it is enough to use item (d) of lemma \ref{c_4} and \ref{pro_maior_que_P}. (It is necessary to use item (d) of lemma \ref{c_4} because when perturbing a piece, this perturbation can not exert influence with scale $\rho^2$ on the renormalization operators associated to the others pieces in the same stacking, this bring on in a lack of independence of the events, since the length of the intervals forming $K$ have scale $\rho$).

\vspace{10mm}

In this way, for every $(x,\theta^-) \in K \cap \mathcal{W}^-,$ there is some $i$ with $1 \leq i \leq J_{\theta^-},$ such that

\begin{tabular}{rl}

$\mathbb{P} \Bigl( \Omega \setminus \Omega_{0,\rho^2}(x,\theta^-) \Bigr) = $ & $ \mathbb{P} \Bigl( \bigcap_{ \underline{a} \in A_{\theta^-}(i) } \left\{ \underline{\tilde{\omega}} \in \Omega_1^{\underline{\omega}}(A_{\theta^-}(i))  \mbox{ such that } R^{\underline{\tilde{\omega}}}_{\underline{a}}(x,\theta^-) \notin K_{-\rho^2}   \right\} \Bigr)$\\

$ = $ & $\prod_{ \underline{a} \in A_{\theta^-}(i) } \mathbb{P}_{\Omega_1^{\underline{\omega}}(\theta^-,\underline{a})} \Bigl(  R^{\underline{\omega}}_{\underline{a}}(x,\theta^-) \in K_{-\rho^2} \Bigr)$\\

$ \leq $ & $(1-P)^{c_{28} \rho^{-(\bar{d}_s - 1)\frac{c}{k}}}$\\

$ \leq $ & $exp \left(- c_{13} \rho^{-\frac{c}{k}(\bar{d}_s-1)} \right),$

\end{tabular}

for some constant $c_{13} > 0$ independent on $\rho.$

\end{proof}

\end{document}